\documentclass[12pt]{article}
\usepackage[backend=biber]{biblatex}
\usepackage{amsmath, amssymb, amsthm, esint, hyperref}
\usepackage{geometry, mathabx}
\newcommand{\bd}{\boldsymbol}

\usepackage{xcolor}
\usepackage{graphicx} 

\def\R{{\mathbb R}}

\theoremstyle{plain}

\newtheorem{theorem}{Theorem}[section]

\newtheorem{proposition}{Proposition}[section]
\newtheorem{remark}{Remark}[section]
\newtheorem{lemma}{Lemma}[section]

\begin{document}

\title{\sc{Existence and long-time behavior of global strong solutions to a nonlinear model of tumor growth}}
\author{Jeffrey Kuan and Konstantina Trivisa}
\maketitle

\begin{abstract}
    In this manuscript, we study a nonlinear model of tumor growth, described by a coupled hyperbolic-elliptic system of partial differential equations. In this model, the compressible flow of tumor cells is modeled by a transport equation for the cell density, which takes into account transport via a background flow (given by a potential solving a Brinkman-type equation), and which has a source term modeling cell growth and death. In this manuscript, we show that for sufficiently large viscosity, the tumor growth system admits nontrivial global strong solutions for positive initial data having a gradient with sufficiently small norm. This illustrates the regularizing effects of the source term representing tumor cell growth and death on the resulting transport dynamics of the equation. Furthermore, we characterize the long-time behavior of global strong solutions to the tumor growth system using a level-set analysis, in which we analyze how level sets evolve as they are transported by the flow, in terms of expansion/contraction and accretion/depletion of cells. While there has been past work on global existence of weak solutions for this tumor growth system \cite{WeberTrivisa, TrivisaWeberDrug}, this manuscript opens the study of well-posedness in terms of more regular strong/classical solutions, which exist globally in time.
\end{abstract}

\section{Description of the model}

In this manuscript, we consider a model of tumor growth, in which cells are transported by a background flow given by a Brinkman type potential, and in which tumor cell proliferation effects are driven by internal pressure within the tumor, where the internal pressure is given by a constitutive relationship involving the cell density. We consider a tumor region $\Omega \subset \R^{d}$ (where $d = 2, 3$) which is a bounded connected domain with smooth boundary $\partial \Omega$. We model the cell density $n(t, x): [0, T] \times \Omega \to [0, \infty)$ and potential $W(t, x): [0, T] \times \Omega \to \R$ by the equations:
\begin{equation*}
\begin{cases}
    \partial_{t}n - \text{div}(n\nabla W) = nG(p), \\
    -\mu \Delta W + W = p,
\end{cases}
\end{equation*}
where $G(p): [0, \infty) \to \R$ is a growth rate function and $p$ is the internal pressure of the tumor. We assume that $G(\cdot)$ is a $C^{1}$ decreasing function with a unique positive zero $p_{0}$ (referred to as the \textit{homeostatic pressure}), which is motivated by physically relevant considerations \cite{Tumor1, Tumor2}. The internal pressure $p(\rho)$ is specified as a function of $\rho$ via a constitutive relationship.

The first equation models the transport of tumor cells via a background flow, where the background velocity $\bd{u}$ is given in terms of the potential as $\bd{u} := -\nabla W$. It behaves as a continuity equation for the tumor, which resembles a compressible flow. The second equation for the potential $W$ is a momentum equation, which is an elliptic equation of Brinkman type. As is common in compressible flow applications \cite{FeireislCompressible, LionsBook}, we assume a power law constitutive relationship for the internal pressure, $p(n) = an^{\gamma}$ for some constant $a > 0$ and $\gamma \ge 1$. For concreteness, we also consider a specific form of the growth rate $G(n)$, given by
\begin{equation}\label{Gdef}
G(n) := \alpha - \beta n^{\gamma \theta},
\end{equation}
for positive constants $\theta > 0$ and $\alpha, \beta > 0$, as is done in \cite{WeberTrivisa}. Therefore, the resulting model under consideration is the following coupled elliptic-hyperbolic system:
\begin{equation}\label{equations}
\begin{cases}
    \partial_{t}n - \text{div}(n\nabla W) = \alpha n - \beta n^{\gamma \theta + 1} \\
    -\mu \Delta W + W = an^{\gamma}
\end{cases} \quad \text{ on } \Omega,
\end{equation}
for some adiabatic constant $\gamma \ge 1$, $\theta > 0$, and positive constants $\alpha, \beta, a > 0$. We prescribe Neumann boundary conditions for the potential $W$:
\begin{equation*}
\nabla W \cdot \bd{n}|_{\partial \Omega} = 0,
\end{equation*}
and we prescribe initial data:
\begin{equation*}
n(0, \cdot) = n_{0},
\end{equation*}
where $n_{0}$ is a smooth function in $C^{\infty}(\overline{\Omega})$ that is nonnegative so that $\displaystyle \min_{x \in \overline{\Omega}} n_{0}(x) := m \ge 0.$ 

\subsection{Summary of results}

The goal of this manuscript is to show that the equations \eqref{equations} admit nontrivial smooth global (classical) solutions. While past analysis for this equation has shown existence of global weak solutions for $L^{\infty}(\Omega)$ nonnegative initial data \cite{WeberTrivisa, TrivisaWeberDrug}, we have the alternative goal in this manuscript of showing that these equations \eqref{equations} admit unique global \textit{strong solutions} for initial data with sufficiently small gradient, and furthermore, analyzing the long-time behavior of global strong solutions, when they exist. One can verify that the equations \eqref{equations} admit trivial global smooth solutions where $n(t, x)$ is constant for each time $t \ge 0$, $n(t, x) := \overline{n}(t)$, for $\overline{n}(t)$ satisfying an ODE for arbitrary constant initial data $m \ge 0$:
\begin{equation*}
\frac{d\overline{n}}{dt} = \alpha \overline{n} - \beta \overline{n}^{1 + \gamma \theta} - a\mu^{-1}\overline{n}^{1 + \gamma}, \qquad \overline{n}(0) = m.
\end{equation*}
However, our goal in this manuscript is to show the existence of \textit{nontrivial} (non-constant) global smooth solutions.

In particular, we have the following global existence result of unique smooth solutions, for strictly positive initial data $n_{0} > 0$ with $\|\nabla n_{0}\|_{L^{6}(\Omega)}$ sufficiently small, under the assumption of sufficiently large viscosity parameter $\mu$ in the Brinkman equation.

\begin{theorem}\label{global}
Let $n_{0} > 0$ be a smooth nonnegative function on $\overline{\Omega}$ with $m := \min_{x \in \overline{\Omega}} n_{0}(x) > 0$ and $M := \max_{x \in \overline{\Omega}} n_{0}(x) > 0$. Then, there exists a constant $C(M, a, \gamma)$ such that if $\mu$ is sufficiently large so that it satisfies 
\begin{equation*}
C(M, a, \gamma)\mu^{-1} < \frac{\alpha}{8}\min(\gamma \theta, 1), 
\end{equation*}
then there exists $\delta$ depending on $m$, $M$, $a$, $\alpha$, $\gamma$, $\theta$, and $\mu$, so that there exists a unique global classical smooth solution $(n, W)$ to \eqref{equations} whenever $\|\nabla n_{0}\|_{L^{6}(\Omega)} < \delta$. Furthermore, if the initial data satisfies $0 < n(t, x) \le n^{*}$ where $n^{*}$ is the unique positive zero of $G(n) := \alpha - \beta n^{\gamma \theta} = 0$, then the solution satisfies the bound 
\begin{equation}\label{boundnstar}
0 \le n(t, x) \le n^{*}, \qquad \text{ if } 0 \le n_{0} \le n_{*},
\end{equation}
for all $t \in [0, T]$. Otherwise, the solution satisfies:
\begin{equation}\label{boundngeneral}
0 < n(t, x) \le n^{*} + (M - n^{*})e^{-rt},
\end{equation}
where $M := \max_{\overline{\Omega}} n_{0}$ and the constant $r > 0$ depends only on $\alpha$ and $\beta$ (and not on $M$).
\end{theorem}
We note that the structure of the growth rate term \eqref{Gdef} introduces additional dissipation into the system \eqref{equations}, which helps to regularize the dynamics of the resulting system. Hence, this global existence result in Theorem \ref{global} for initial data with small gradient is the analogue of a similar small data global existence result for the damped 3D compressible Euler equations, see Theorem 5.2 in \cite{Sideris}. We emphasize that in our result in Theorem \ref{global}, there is no restriction on the $L^{\infty}(\Omega)$ norm of the initial data; there is only a restriction on the size of the $L^{p}(\Omega)$ norm of the \textit{gradient} of the initial data. 

After showing that nontrivial global smooth solutions exist for our system \eqref{equations}, our next goal is to \textbf{characterize the asymptotic behavior of these global strong solutions as $t \to \infty$}. Let $n^{*} > 0$ be defined as the unique positive value for which
\begin{equation}\label{nstardef}
G(n^*) := \alpha - \beta (n^*)^{\gamma \theta} = 0.
\end{equation}
From inspection, we see that there are two trivial stationary solutions to the system \eqref{equations}:
\begin{equation*}
(n, W) = (0, 0) \qquad \text{ and } \qquad (n, W) = (n^{*}, a(n^{*})^{\gamma}),
\end{equation*}
which gives preliminary hints as to the potential asymptotic behavior of the system. It is clear that the system started at $n_{0}(x) = 0$ stays stationary at the first stationary solution $(n, W) = (0, 0)$. We show that any global strong solution $(n(t, x), W(t, x))$ with $n \in C(0, T; W^{1, p}(\Omega))$ for some $p > d$, converges as $t \to \infty$ in an appropriate strong topology to the second stationary solution $(n^{*}, a(n^{*})^{\gamma})$. In particular, we obtain the following result long-time behavior result:

\begin{theorem}\label{longtime}
Any global strong solution $(n(t, x), W(t, x))$ to \eqref{equations} with the regularity $n \in C(0, T; W^{1, p}(\Omega))$ and $W \in C(0, T; W^{3, p}(\Omega))$ for some $p > d$ and for some smooth initial data $n_{0} \ge 0$, satisfies the following long-time behavior asymptotic estimates in time:
\begin{equation*}
\lim_{t \to \infty} \|n(t, x) - n^{*}\|_{L^{q}(\Omega)} = 0, \qquad \lim_{t \to \infty} \|W(t, x) - a(n^{*})^{\gamma}\|_{W^{2, q}(\Omega)} = 0,
\end{equation*}
for all $1 \le q < \infty$. 
\end{theorem}

\subsection{Literature review}

The mathematical analysis of tumor growth has been widely studied in the mathematical literature. Such models often involve coupled systems of PDEs, and the questions that we study in the current manuscript, namely questions of global existence of solutions, uniqueness, stationary solutions, and long-time convergence to stationary states, are classical mathematical problems of interest in the study of tumor growth models throughout the years. We review some of the developments in the mathematical modeling of tumor growth.

One important class of models of tumor growth are free boundary models in which tumor cells are transported by a background flow, typically given as the gradient of an internal pressure via Darcy's law, where the free boundary evolves according to the dynamics of the tumor, for example in relation to the internal pressure along the moving interface. Typically, one models the concentration of nutrients (and potentially other substances) and the movement of the boundary. In the case of spherical symmetry (which allows for the consideration of a 1D model), questions of global asymptotic stability to steady states and the mathematical analysis of such models has been carried out in \cite{Friedman1D_3, Friedman1D_2, Friedman1D_1}. Local well-posedness has also been studied in more general contexts for radially symmetric tumor growth models, for example in the case of a coupled cytokine-tumor cell system in \cite{CytokineChenFriedman}.

For higher-dimensional free boundary tumor growth models which do not assume radial symmetry, the analysis of non-radially symmetric steady states under certain conditions is performed in \cite{FriedmanBifurcation2, FriedmanBifurcation} for two dimensions and in \cite{FriedmanBifurcation3D} for three dimensions, where stationary solutions which break radial symmetry are analyzed (referred to as ``symmetry-breaking bifurcations"). In such models, the movement of the boundary is driven by the internal pressure on the boundary, which must satisfy a condition of agreeing with the surface tension (proportion to the mean curvature) along the free boundary. The questions of local well-posedness \cite{BazilyFriedman1} and global well-posedness/asymptotic stability for solutions sufficiently close to radially symmetric solutions under certain assumptions on the proliferation rate \cite{BazilyFriedman2} have also been studied for this non-radially symmetric free boundary model of tumor growth.

There are also extensions to free boundary problems with mixed states, namely cells of various types where the total combined density is constant throughout the tumor. One such model (an $M3$ model) consisting of proliferating, quiescent, and dead/necrotic cells is studied in \cite{ChenFriedmanM3}, in terms of local well-posedness (without any assumption of radial symmetry). For such models with multiple (two or three) possible states for the tumor cells, there are results on existence of unique stationary solutions \cite{CuiFriedmanM3}, and linear asymptotic stability of stationary solutions \cite{ChenCuiFriedman}. There are also global well-posedness results for such models in the case of radial symmetry, for an elliptic-hyperbolic system with a Brinkman-type equation for the nutrient concentration \cite{CuiFriedmanGlobal}, for a parabolic-hyperbolic system with a parabolic equation for the nutrient concentration \cite{CuiWei}, and for a parabolic-hyperbolic system which also includes the effects of drug application \cite{Zhao}. For more information on free boundary tumor models, we refer the reader to the survey article \cite{FriedmanOverview} and the references therein.

There are extensions of these models in which there is \textit{variable tumor cell density} throughout the tissue, which gives these free boundary models a compressible nature, as tumor cells can compress/expand while being transported by the background flow. Such a model is analyzed in \cite{DonatelliTrivisa}, in which there is a continuity equation for the cell density, and the background flow is given by a Brinkman-type (elliptic) equation, and global existence of weak solutions is established in this case. In this model, the Brinkman equation is instead used, which is a regularized form of Darcy's law. These global existence results for weak solutions were extended in \cite{DonatelliTrivisa3} to the case of tumor growth with drug application and also in \cite{DonatelliTrivisa2}, where in addition to the continuity equation for the total density, there is a (parabolic) Forchhemier's equation (which generalizes the Brinkman law). We remark that these works use a lot of the mathematical machinery developed for compressible Navier-Stokes flows, see \cite{FeireislCompressible, LionsBook}.

In contrast to the models discussed thus far, in the current manuscript, we consider a tumor growth model on fixed domain. This model is motivated by past works \cite{Perthame2, Perthame3, Perthame4, Perthame1}, which consider the tumor as a medium in which cells are actively transported by a background flow and grow/deplete according to some given pressure law related to the density via a power law. Most closely related to the model in the current manuscript is the work of \cite{Perthame4}, which uses a Brinkman equation for the potential. These works \cite{Perthame2, Perthame3, Perthame4, Perthame1} are particularly interesting, as they relate such models to free boundary Hele-Shaw type models using a ``stiff pressure law limit," in which the power in the power law for the pressure goes to infinity. For such models of active transport and growth, particularly \cite{Perthame4}, global existence of weak solutions and development of a convergent finite difference scheme is accomplished in \cite{WeberTrivisa}, and this is later extended to a coupled system with drug application and nutrient consumption in \cite{TrivisaWeberDrug}, using methods of compressible flow dynamics \cite{FeireislCompressible, LionsBook}. 

We remark that the past existence results on the model in consideration in this manuscript are for global existence of \textit{weak solutions} (see \cite{WeberTrivisa}), but due to the dissipation present in \eqref{equations}, we anticipate a better result. This is in the spirit of the well-known work by \cite{Sideris}, where it is shown that the addition of damping to the 3D compressible Euler equations gives rise to global strong solutions for small initial data. Thus, the contribution of the current manuscript will be to study the well-posedness of \eqref{equations} in terms of global \textit{strong solutions} for appropriate initial data (with ``small" gradient), and to develop methods for studying the asymptotic long-time behavior of strong solutions to the system. In contrast to the methods for weak solutions in the spirit of \cite{FeireislCompressible, LionsBook} for compressible flows, we use different methods here, for showing well-posedness in terms of strong solutions. We also introduce a new perspective of studying the equations \eqref{equations}, which involves a geometric level set analysis (see Section \ref{levelsetsec}) that has the advantage of providing direct hands-on information about how the tumor growth system evolves.

\subsection{Outline}

The first goal of this manuscript is to establish global existence and uniqueness, which depends on careful a priori estimates. Along these lines, we first begin by proving some essential \textit{a priori estimates} and minimum/maximum principles in Section \ref{apriori} that will be useful in our analysis. 

In Section \ref{localexistencesec}, we show local well-posedness for the system with \textit{strictly positive} initial data. This consists of deriving estimates for a linearized equation in Section \ref{linearized}, and then applying the Schauder fixed point theorem to obtain a local solution as a fixed point of a map in Section \ref{fixedpointsec}. 

We then extend our local existence result for positive initial data to a global existence result for initial data with ``small" gradient in Section \ref{globalsec}, hence proving Theorem \ref{global}. The essential part of the analysis is showing an a priori dissipative estimate on the $L^{p}(\Omega)$ norm of $\nabla n$ in Proposition \ref{W1pcontrol}, which hence gives global control of these norms in time for appropriate solutions to \eqref{equations}. We also show a weaker (non-dissipative) a priori estimate in Proposition \ref{W1pgeneral}. We then obtain higher regularity results in Proposition \ref{regularity} to show that local solutions are actually smooth. Then, in Section \ref{globalproofs}, we combine all of these results to prove the global existence result in Theorem \ref{global}. 

Finally, since we now have that nontrivial strong solutions exist, we study the long-time asymptotics of global strong solutions to \eqref{equations}, when they exist in Section \ref{vacuumlongtime}. We begin by discussing a flow map $\Phi_{t}(x)$ of the background flow $\bd{u} := -\nabla W$ and its associated properties in Section \ref{transport}. We then use a level-set analysis to show properties of the expansion/contraction of regions near vacuum in Section \ref{levelsetsec}. Finally, we use all of this information to prove the long-time behavior result in Theorem \ref{longtime}. In the proof of this result, we use several results about the Green's function of the Brinkman equation, which we establish in the appendix in Section \ref{appendix}.

\section{A priori estimates}\label{apriori}

We will derive some important Sobolev and maximum/minimum a priori estimates that will have an important role in our analysis. We derive a priori estimates under the assumption that $n \in L^{\infty}(0, T; W^{1, p}(\Omega))$ for some $p > d$. The first is an observation about the  Brinkman equation, which quantifies the Calder\'{o}n-Zygmund regularity properties of the elliptic equation for the potential $W$. We first recall some classical Calder\'{o}n-Zygmund estimates for Neumann boundary problems, from \cite{ADN, Grisvard}.

\begin{proposition}\label{CZest}
Suppose that $W$ satisfies the Neumann boundary problem:
\begin{equation*}
-\mu \Delta W + W = f \text{ on } \Omega, \qquad \qquad \nabla W \cdot \bd{n}|_{\partial \Omega} = 0.
\end{equation*}
Then, for all positive integers $m$, 
\begin{equation*}
\|W\|_{W^{m + 2, p}(\Omega)} \le C(m, p) \Big(\|f\|_{W^{m, p}(\Omega)} + \|W\|_{L^{p}(\Omega)}\Big).
\end{equation*}
\end{proposition}
\begin{proof}
To obtain the estimates for the case of $m = 0$, we can appeal to Theorem 2.3.3.6 in \cite{Grisvard}, and from here, we can obtain the general result for general positive integers $m \ge 1$ using Theorem 15.2 in \cite{ADN}.
\end{proof}

Next, we apply the Calder\'{o}n-Zygmund estimates in Proposition \ref{CZest} to obtain the following \textit{a priori estimate} for the potential $W$. 

\begin{proposition}\label{nonnegativeW}
    Consider the equation:
    \begin{equation}\label{brinkmann}
        -\mu \Delta W + W = an^{\gamma} \quad \text{ on } \Omega,
    \end{equation}
    with Neumann boundary conditions for a nonnegative $n \in L^{\infty}(0, T; W^{1, p}(\Omega))$ for $p > d$. Then, there exists a unique classical solution $W \in L^{\infty}(0, T; W^{3, p}(\Omega))$ to \eqref{brinkmann} such that:
    \begin{equation*}
        \|W\|_{L^{\infty}(0, T; W^{3, p}(\Omega))} \le C_p\|n\|_{L^{\infty}(0, T; W^{1, p}(\Omega))}^{\gamma},
    \end{equation*}
    for a constant $C_p$ depending on $p$ that is independent of $T$. Also, $W \ge 0$ is nonnegative.
\end{proposition}

\begin{proof}
Note that the choice of $p$ ensures that by Sobolev embedding, $n \in L^{\infty}(0, T;  C(\overline{\Omega}))$ and hence, $n^{\gamma} \in L^{\infty}(0, T; L^{2}(\Omega))$, so the existence of a unique weak solution $W \in L^{\infty}(0, T; H^{1}(\Omega))$ is guaranteed by the Lax-Milgram theorem, see \cite{EvansPDE}. Testing the equation \eqref{brinkmann} by $W^{p - 1}$ and using Young's inequality, we obtain:
\begin{equation}\label{WH1}
\mu (p - 1)\int_{\Omega} W^{p - 2} |\nabla W|^{2} + \frac{1}{2} \int_{\Omega} W^{p} \le C\int_{\Omega} n^{p\gamma} \le C\|n^{\gamma}\|^{p}_{L^{\infty}(0, T; L^{p}(\Omega))} \le C\|n\|_{L^{\infty}(0, T; W^{1, p}(\Omega))}^{\gamma p},
\end{equation}
since $W^{1, p}(\Omega) \subset C(\overline{\Omega})$. Hence,
\begin{equation}\label{WLp}
\|W\|_{L^{\infty}(0, T; L^{p}(\Omega))} \le C_p\|n\|^{\gamma}_{L^{\infty}(0, T; W^{1, p}(\Omega))}.
\end{equation}

Since $\nabla (n^{\gamma}) = \gamma n^{\gamma - 1} \nabla n$ and $\gamma \ge 1$ and $W^{1, p}(\Omega) \subset C(\overline{\Omega})$ since $p > d$, we see that $n^{\gamma} \in L^{\infty}(0, T; W^{1, p}(\Omega))$ also. Therefore, by Calder\'{o}n-Zygmund estimates (Theorem \ref{CZest}) and the estimate in \eqref{WLp}, we obtain the desired regularity estimate:
\begin{equation*}
\|W\|_{L^{\infty}(0, T; W^{3, p}(\Omega))} \le C_p\Big(\|n^{\gamma}\|_{L^{\infty}(0, T; W^{1, p}(\Omega))} + \|W\|_{L^{\infty}(0, T; L^{p}(\Omega))}\Big) \le C_p\|n\|^{\gamma}_{L^{\infty}(0, T; W^{1, p}(\Omega))}.
\end{equation*}

Finally, it remains to show that $W$ is nonnegative. We test the equation \eqref{brinkmann} by $-\text{sgn}(-W)$, which is $-1$ if $W \le 0$ and $0$ if $W > 0$. We integrate over $\Omega$ to obtain:
\begin{equation}\label{Wminus}
\int_{\Omega} W^{-} = -\int_{\Omega} an^{\gamma} \text{sgn}(-W) + \mu \int_{\Omega} \Delta (-W) \text{sgn}(-W),
\end{equation}
where $W^{-} := -\min(W, 0)$. By the result on pg. 64 of \cite{FeireislNovotny}, since $W \in L^{\infty}(0, T; W^{3, p}(\Omega)) \subset L^{\infty}(0, T; H^{2}(\Omega))$ with $\nabla W \cdot \bd{n}|_{\partial \Omega} = 0$, we have $\displaystyle
\mu \int_{\Omega} \Delta (-W) \text{sgn}(-W) \le 0$. Hence, by $n \ge 0$ and \eqref{Wminus}, we obtain $\displaystyle \int_{\Omega} W^{-} \le 0$, and thus $W^{-}$ is identically zero, so $W(t, \cdot) \ge 0$ for all $t$. 
\end{proof}

Applying Hopf's lemma to the elliptic equation for $W$, we can also obtain minimum and maximum bounds on the potential $W$. 

\begin{proposition}\label{Wbound}
Let $W(x)$ and $n(x)$ be spatially smooth functions on $\overline{\Omega}$ with $n \ge 0$, which satisfy
\begin{equation*}
-\mu \Delta W + W = an^{\gamma}, \qquad \text{ on } \Omega.
\end{equation*}
Then, for all $x \in \overline{\Omega}$:
\begin{equation*}
a\left(\min_{x \in \overline{\Omega}} n(x)\right)^{\gamma} \le W(x) \le a\left(\max_{x \in \overline{\Omega}} n(x)\right)^{\gamma}.
\end{equation*}
\end{proposition}

\begin{proof}
    This follows by using Hopf's lemma, and we refer the interested reader to the complete proof provided in the proof of Lemma 3.2 in \cite{WeberTrivisa}.
\end{proof}

The next is an observation about the first equation, which is the continuity equation for the cell density. It is a maximum principle for the continuity equation, which will give uniform in time $\|n\|_{L^{\infty}(0, T; L^{\infty}(\Omega))}$ bounds. In particular, the result shows that the solution is always bounded above by $n^{*}$ if it is initially bounded above by $n^{*}$; else, the maximum value of the solution decays exponentially to $n^{*}$ asymptotically as $t \to \infty$. We refer the reader to \eqref{nstardef} for the definition of $n^{*}$.

\begin{proposition}\label{maxprinciple}
    Let $n$ and $W$ be spatially smooth functions on $\overline{\Omega}$ for each time $t \ge 0$, which satisfy \eqref{equations} with smooth initial data $n_{0} \ge 0$ and Neumann boundary conditions $\nabla W \cdot \bd{n}|_{\partial \Omega} = 0$. Recall the definition of $n^{*}$ from Theorem \ref{global}. Suppose that $\partial_{t}n(t, \cdot)$ is also spatially smooth on $\overline{\Omega}$ for all $t \ge 0$, and let $M = \max_{x \in \overline{\Omega}} n_{0}(x)$.
    \begin{itemize}
        \item If $0 \le M \le n^{*}$, we have that the following upper bound holds: 
    \begin{equation*}
        0 \le n(t, x) \le n^{*}, \qquad \text{ for all $t \in [0, T]$ and $x \in \Omega$}.
    \end{equation*}
    \item If $M > n^{*}$, then there exists a positive constant $r > 0$ (independent of $M$, depending only on the parameters $\alpha$ and $\beta$) such that
    \begin{equation*}
    0 \le n(t, x) \le n^{*} + (M - n^{*}) e^{-rt}, \qquad \text{ for all $t \in [0, T]$ and $x \in \Omega$.}
    \end{equation*}
    \end{itemize}
\end{proposition}  

\begin{proof}
We will hence prove the second assertion, since the first assertion follows similarly. Since $n^{*}$ is a constant, we obtain the following equation for $n - n^{*}$:
\begin{equation*}
\partial_{t}(n - n^{*}) - \text{div}((n - n^{*})\nabla W) = \alpha n - \beta n^{1 + \gamma \theta} + n^{*}\Delta W.
\end{equation*}
Setting $w := e^{rt}(n - n^{*})$ for $r > 0$ and using the second equation in \eqref{equations}, we obtain
\begin{equation*}
\partial_{t}w - \nabla w \cdot \nabla W = e^{rt}(\alpha n - \beta n^{1 + \gamma \theta}) + \mu^{-1} (w + e^{rt}n^*)(W - an^{\gamma}) + rw.
\end{equation*}
Note that since $\alpha n^{*} - \beta (n^{*})^{1 + \gamma \theta} = 0$ and the function $z \to \alpha z - \beta z^{1 + \gamma \theta}$ is concave down, there exists a positive constant $C$ such that
\begin{equation*}
\alpha n - \beta n^{1 + \gamma \theta} < -C(n - n^{*}), \qquad \text{ for all } n \ge n^{*},
\end{equation*}
so we can fix $r$ to be any constant such that $0 < r < C$. 

We then claim that $w(t, x) \le (M - n^{*})$ for all $t \ge 0$ and for all $x \in \overline{\Omega}$, where we note that $\max_{x \in \overline{\Omega}} w(0, x) = M - n^{*}$ by construction. To show this, we argue by contradiction and suppose that the strict maximum over $[0, T] \times \overline{\Omega}$ for arbitrary $T$ is attained at a point $(t_0, x_0) \in (0, T] \times \overline{\Omega}$, so that $w(t_0, x_0) > M - n^{*}$. In this case, we have:
\begin{itemize}
    \item $\nabla w(t_0, x_0) = 0$ (due to Neumann boundary conditions, this holds even if $x_0 \in \partial \Omega$). 
    \item Furthermore, since $w(t_0, x_0) > M - n^{*}$, we have that $n(t_0, x_0) > n^{*}$ and hence
    \begin{equation*}
    e^{rt}(\alpha n - \beta n^{1 + \gamma \theta}) + rw|_{(t_0, x_0)} \le -(C - r) w(t_0, x_0) < 0,
    \end{equation*}
    where $C - r > 0$. 
    \item Finally, by Proposition \ref{Wbound}, $W(t_0, x_0) - a(n(t_0, x_0))^{\gamma} \le 0$ since $n(t_0, x_0) = \max_{x \in \overline{\Omega}} n(t_0, x)$. Note also that $w + e^{rt}n^{*} > 0$.
\end{itemize}
Therefore, we conclude that $\partial_{t}w(t_0, x_0) < 0$, but this contradicts that $w(t_0, x_0) > M - n^{*}$ is the strict maximum in $[0, T] \times \overline{\Omega}$. So we conclude that $w(t, x) \le M - n^{*}$ for all $t \in [0, T]$, $x \in \overline{\Omega}$, and hence,
\begin{equation*}
0 \le n(t, x) \le n^{*} + (M - n^{*}) e^{-rt}, \qquad \text{ for all } t \in [0, T], x \in \overline{\Omega}.
\end{equation*}
We defer the proof of $n(t, x) \ge 0$ to the next proposition, which is about minimum principles.
\end{proof}

Finally, we conclude this section on a priori estimates by showing a corresponding lower bound for the solution away from vacuum if the initial data is initially bounded below away from vacuum. Since the proof of this lower bound result does not use Hopf's lemma, we can establish this lower bound for lower regularity solutions.

\begin{proposition}\label{minprinciple}
Consider smooth initial data $n_{0} \ge 0$ such that $n_{0} \ge m \ge 0$ for some nonnegative constant $m$. Then, a strong solution $(n, W)$ to \eqref{equations} with $n \in C(0, T; H^{2}(\Omega))$ and $\partial_{t}n \in C(0, T; H^{1}(\Omega))$ satisfies the uniform bound for all $t \in [0, T]$:
\begin{equation*}
n(t, x) \ge \overline{n}(t),
\end{equation*}
where $\overline{n}$ is the unique solution to the ODE initial value problem:
\begin{equation*}
\frac{d\overline{n}}{dt} = \alpha \overline{n} - \beta \overline{n}^{1 + \gamma \theta} - a\mu^{-1} \overline{n}^{1 + \gamma}, \qquad \overline{n}(0) = m.
\end{equation*}
As a consequence, for all $t \in [0, T]$:
\begin{equation*}
n(t, x) \ge n_{*} := \min(m, \eta),
\end{equation*}
where $\eta > 0$ is defined to be the unique solution to the equation:
\begin{equation}\label{nlowerstar}
\alpha - \beta n^{\gamma \theta} - a\mu^{-1}n^{\gamma} = 0.
\end{equation}
\end{proposition}

\begin{proof}
Using the equation $-\mu\Delta W + W = an^{\gamma}$, we obtain from the continuity equation that
\begin{equation*}
\partial_{t}n - \nabla n \cdot \nabla W = \alpha n - \beta n^{1 + \gamma \theta} + \mu^{-1}nW - a\mu^{-1}n^{1 + \gamma}.
\end{equation*}
Since $W \ge 0$ by Proposition \ref{nonnegativeW}, we use the ODE for $\overline{n}$ and the fact that $\overline{n}$ is spatially constant to conclude that
\begin{equation*}
\partial_{t}\overline{n} - \nabla \overline{n} \cdot \nabla W = \partial_{t}\overline{n} \le \alpha \overline{n} - \beta \overline{n}^{1 + \gamma \theta} + \mu^{-1}\overline{n}W - a\mu^{-1} \overline{n}^{1 + \gamma}.
\end{equation*}
Subtracting the equations for $n$ and $\overline{n}$ and multiplying by $\text{sgn}(\overline{n} - n)$, we obtain:
\begin{align*}
&\partial_{t} (\overline{n} - n)^{+} \\
&\le \nabla (\overline{n} - n)^{+} \cdot \nabla W + \alpha (\overline{n} - n)^{+} - \beta (\overline{n}^{1 + \gamma \theta} - n^{1 + \gamma \theta})^{+} - a\mu^{-1} (\overline{n}^{1 + \gamma} - n^{1 + \gamma})^{+} + \mu^{-1}(\overline{n} - n)^{+} W\\
&\le \nabla (\overline{n} - n)^{+} \cdot \nabla W + \alpha (\overline{n} - n)^{+} + \mu^{-1}(\overline{n} - n)^{+}W.
\end{align*}
This is justified because $\overline{n} - n \in H^{1}([0, T] \times \Omega)$ by the regularity assumptions on $n$, and by pg.~64 in \cite{FeireislNovotny}, we have the identities $\text{sgn}(f)\partial_{t}f = \partial_{t}f^+$ and $\text{sgn} (f)\nabla f = \nabla f^{+}$ for $f \in H^{1}([0, T] \times \Omega)$. So integrating over $\Omega$, we obtain that
\begin{align*}
\partial_{t}\left(\int_{\Omega} (\overline{n} - n)^{+}\right) &\le \left|\int_{\Omega} \nabla (\overline{n} - n)^{+} \cdot \nabla W\right| + \alpha \int_{\Omega} (\overline{n} - n)^{+} + \mu^{-1} \int_{\Omega} (\overline{n} - n)^{+}W \\
&= \left|\int_{\Omega} (\overline{n} - n)^{+} \Delta W\right| + \alpha \int_{\Omega} (\overline{n} - n)^{+} + \mu^{-1} \int_{\Omega} (\overline{n} - n)^{+}W.
\end{align*}
So since $W \in L^{\infty}(0, T; W^{3, 6}(\Omega))$ by Proposition \ref{nonnegativeW} and Sobolev embedding for all $T > 0$, we obtain that for almost all $t \in [0, T]$:
\begin{equation*}
\int_{\Omega} (\overline{n} - n)^{+}(t) \le \Big(C(1 + \mu^{-1})\|W\|_{L^{\infty}(0, T; W^{3, 6}(\Omega))} + \alpha\Big) \int_{0}^{t}  \int_{\Omega} (\overline{n} - n)^{+}(s) ds,
\end{equation*}
since $\overline{n}(0) = m \le n_{0}$ by definition so that $\displaystyle \int_{\Omega} (\overline{n} - n)^{+}(0) = 0$. So by Gronwall's lemma (Lemma 1.4.2 in \cite{BFH18}), $(\overline{n} - n)^{+}(t) = 0$ for all $t \ge 0$, and hence $n(t, x) \ge \overline{n}(t) > 0$ for all $t \ge 0$.

A simple stability analysis of the ODE shows that for $\overline{n}(0) = m$ where $0 < m < \eta$, $n(t)$ is strictly increasing with $\lim_{t \to \infty} \overline{n}(t) = \eta$. If $m = 0$, then $\overline{n}(t) = 0$ for all $t$ and if $m \ge \eta$, then $\overline{n}(t) \ge \eta$ for all $t$. This verifies the second claim, given the previous result that $n(t, x) \ge \overline{n}(t)$ for all $t \in [0, T]$, since $\overline{n}(t) \ge n_{*} := \min(m, \eta)$ for all $t \ge 0$.
\end{proof}

Finally, we conclude this section by establishing the \textit{uniqueness} claim in Theorem \ref{global} on uniqueness of strong solutions $(n, W)$, with $n \in C(0, T; W^{1, p}(\Omega))$ and $W \in C(0, T; W^{3, p}(\Omega))$ for some $p > d$, to \eqref{equations}.

\begin{proof}[Proof of uniqueness claim in Theorem \ref{global}]
Consider $(n_{1}, W_{1})$ and $(n_{2}, W_{2})$, which are both strong solutions to \eqref{equations} for the same smooth initial data $n_{0} \ge 0$, with regularity $n_{i} \in C(0, T; W^{1, p}(\Omega))$ and $W_{i} \in C(0, T; W^{3, p}(\Omega))$ for some $p > d$ and $i = 1, 2$. Then,
\begin{multline}\label{diff1}
\partial_{t}(n_{1} - n_{2}) - \nabla (n_{1} - n_{2}) \cdot \nabla W_{1} - \nabla n_{2} \cdot \nabla (W_{1} - W_{2}) \\
- (n_{1} - n_{2}) \Delta W_{1} - n_{2} \Delta (W_{1} - W_{2}) = (\alpha n_{1} - \beta n_{1}^{1 + \gamma \theta}) - (\alpha n_{2} - \beta n_{2}^{1 + \gamma \theta})
\end{multline}
\begin{equation}\label{diff2}
-\mu \Delta (W_{1} - W_{2}) + (W_{1} - W_{2}) = an_{1}^{\gamma} - an_{2}^{\gamma}.
\end{equation}
We test the first equation \eqref{diff1} by $n_{1} - n_{2}$. We compute that
\begin{multline}\label{diff1result}
\frac{1}{2} \int_{\Omega} (n_{1} - n_{2})^{2}(t) - \frac{1}{2} \int_{0}^{t} \int_{\Omega} (n_{1} - n_{2})^{2} \Delta W_{1} - \int_{0}^{t} \int_{\Omega} (n_{1} - n_{2}) \nabla n_{2} \cdot \nabla (W_{1} - W_{2}) \\
- \int_{0}^{t} \int_{\Omega} n_{2}(n_{1} - n_{2})\Delta (W_{1} - W_{2}) = \int_{0}^{t} \int_{\Omega} \Big((\alpha n_{1} - \beta n_{1}^{1 + \gamma \theta}) - (\alpha n_{2} - \beta n_{2}^{1 + \gamma \theta})\Big) (n_{1} - n_{2}).
\end{multline}
For the estimates, we define $n_{max, T}$ to be any nonnegative number for which $n_{1}(t) \le n_{max, T}$ and $n_{2}(t) \le n_{max, T}$ for all $t \in [0, T]$. We estimate the resulting terms as follows:
\begin{itemize}
\item For the first term, since $W_{1} \in C(0, T; W^{3, p}(\Omega))$ for $p > d$, $\|\Delta W_{1}\|_{L^{\infty}(0, T; L^{\infty}(\Omega))} \le C$.
\item For the next term, we estimate that $\|an_{1}^{\gamma} - an_{2}^{\gamma}\|_{L^{2}(\Omega)} \le C(n_{max, T})\|n_{1} - n_{2}\|_{L^{2}(\Omega)}$, so
\begin{equation*}
\|\nabla (W_{1} - W_{2})\|_{L^{6}(\Omega)} \le C\|W_{1} - W_{2}\|_{H^{2}(\Omega)} \le C(n_{max, T})\|n_{1} - n_{2}\|_{L^{2}(\Omega)},
\end{equation*}
by Calder\'{o}n-Zygmund estimates applied to \eqref{diff2}. Therefore, since $\|\nabla n_{2}\|_{L^{3}(\Omega)} \le C$ by assumption, we estimate:
\begin{equation*}
\left|\int_{0}^{t} \int_{\Omega} (n_{1} - n_{2}) \nabla n_{2} \cdot \nabla (W_{1} - W_{2})\right| \le C(n_{max, T})\int_{0}^{t} \|n_{1} - n_{2}\|^{2}_{L^{2}(\Omega)}.
\end{equation*}
\item We compute
\begin{align*}
\left|\int_{0}^{t} \int_{\Omega} n_{2}(n_{1} - n_{2}) \Delta(W_{1} - W_{2})\right| &\le n_{max, T} \int_{0}^{t} \|n_{1} - n_{2}\|_{L^{2}(\Omega)} \|W_{1} - W_{2}\|_{H^{2}(\Omega)} \\
&\le C(n_{max, T}) \int_{0}^{t} \|n_{1} - n_{2}\|^{2}_{L^{2}(\Omega)}.
\end{align*}
\item Finally, we can estimate that
\begin{equation*}
|(\alpha n_{1} - \beta n_{1}^{1 + \gamma \theta}) - (\alpha n_{2} - \beta n_{2}^{1 + \gamma \theta})| \le C(n_{max, T}) |n_{1} - n_{2}|.
\end{equation*}
\end{itemize}
Next, we test the equation for the potential \eqref{diff2} by $W_{1} - W_{2}$ and obtain:
\begin{equation}\label{diff2result}
\mu \int_{0}^{t} \int_{\Omega} |\nabla (W_{1} - W_{2})|^{2} + \int_{0}^{t} \int_{\Omega} (W_{1} - W_{2})^{2} = \int_{0}^{t} \int_{\Omega} a(n_{1}^{\gamma} - n_{2}^{\gamma})(W_{1} - W_{2}).
\end{equation}
We estimate $|n_{1}^{\gamma} - n_{2}^{\gamma}| \le C(n_{max, T}) |n_{1} - n_{2}|$ and hence,
\begin{equation*}
\left|\int_{0}^{t} \int_{\Omega} a(n_{1}^{\gamma} - n_{2}^{\gamma}) (W_{1} - W_{2})\right| \le \frac{1}{2} \int_{0}^{t} \int_{\Omega} (W_{1} - W_{2})^{2} + C(n_{max, T}) \int_{0}^{t} \int_{\Omega} (n_{1} - n_{2})^{2}.
\end{equation*}
So combining all of the estimates with \eqref{diff1result} and \eqref{diff2result}, we obtain:
\begin{equation*}
\frac{1}{2} \int_{\Omega} (n_{1} - n_{2})^{2}(t) + \frac{1}{2} \int_{0}^{t} \int_{\Omega} (W_{1} - W_{2})^{2} + \mu \int_{0}^{t} \int_{\Omega} |\nabla(W_{1} - W_{2})|^{2} \le C(n_{max, T}) \int_{0}^{t} \int_{\Omega} (n_{1} - n_{2})^{2}.
\end{equation*}
By Gronwall's lemma, we conclude that $n_{1}(t) = n_{2}(t)$ and $W_{1}(t) = W_{2}(t)$ for all $t \in [0, T]$. 
\end{proof}

\section{Local existence of a solution for positive initial data}\label{localexistencesec}

Having proved the uniqueness claim in Theorem \ref{global} at the end of Section \ref{apriori}, it suffices to show existence of global smooth solutions for appropriate initial data (satisfying a ``small" gradient assumption, as described in Theorem \ref{global}). To show this global existence result, as a first step, we consider the \textbf{local existence of a solution for smooth positive initial data}, and then in the following Section \ref{globalsec}, we show that the local solution can be prolonged for all time (for appropriate initial data), using careful a priori estimates. In this section and the next, we consider strictly positive initial data $n_{0} \in C^{\infty}(\overline{\Omega})$, namely smooth initial data such that:
\begin{equation*}
M \ge n_{0}(x) \ge m > 0, \qquad \text{ for all } x \in \overline{\Omega},
\end{equation*}
where
\begin{equation}\label{Mmdef}
m := \min_{x \in \overline{\Omega}} n_{0}(x), \qquad M := \max_{x \in \overline{\Omega}} n_{0}(x).
\end{equation}
We will first establish local existence of a solution in an appropriate function space, up to some (sufficiently small) time $T^{*} > 0$ to the given system by using a Schauder fixed point argument, in the spirit of for example \cite{Valli}. This will require first obtaining estimates on a linearized form of the given continuity equation, and then defining an appropriate map for which a fixed point of this map would be a solution to the system. In this section, we will hence analyze a linear equation related to \eqref{equations}, which will allow us to show local existence of solutions to \eqref{equations} for positive smooth initial data.

\subsection{Analysis of the linearized continuity equation}\label{linearized}

To set up the local existence proof for \eqref{equations}, in which we will express the solution to \eqref{equations} as the fixed point of some abstract map, we first examine the linearized problem for the continuity equation, given by
\begin{equation}\label{continuitylinear}
\partial_{t}n - \nabla n \cdot \nabla W - n \Delta W = F,
\end{equation}
where $F$ is a given source term, and $W$ is a given a priori known function. To set up a fixed point argument for \eqref{equations} using the Schauder fixed point theorem, we need appropriate notions of existence, uniqueness, and continuous dependence for the \textit{linearized} equation \eqref{continuitylinear}. We have the following existence result and associated estimate for this linearized equation.

\begin{proposition}[Existence and uniqueness for the linearized problem]\label{linearestimate}
    For any integer $m \ge 2$, given a function $W \in L^{\infty}(0, T; H^{m + 2}(\Omega))$ and a source term $F \in L^{\infty}(0, T; H^{m}(\Omega))$, there exists a unique solution $n \in L^{\infty}(0, T; H^{m}(\Omega))$ to \eqref{continuitylinear}. Furthermore, this solution satisfies the following energy estimate for almost every $t \in [0, T]$:
    \begin{equation}\label{Hmest}
    \|n(t)\|_{H^{m}(\Omega)} \le C_{0}\Big(\|n_{0}\|_{H^{m}(\Omega)} + \|F\|_{L^{2}(0, T; H^{m}(\Omega))}\Big) \text{exp}\Big(Ct(1 + \|W\|_{L^{\infty}(0, T; H^{m + 2}(\Omega))})\Big),
    \end{equation}
    and $\partial_{t}n \in L^{2}(0, T; H^{m - 1}(\Omega))$ with the additional estimate:
    \begin{multline*}
        \|\partial_{t}n\|_{L^{2}(0, T; H^{m - 1}(\Omega))} \le \|F\|_{L^{2}(0, T; H^{m}(\Omega))} \\
        + C\|W\|_{L^{2}(0, T; H^{m + 2}(\Omega))}\Big(\|n_{0}\|_{H^{m}(\Omega)} + \|F\|_{L^{2}(0, T; H^{m}(\Omega))}\Big) \text{exp}\Big(CT(1 + \|W\|_{L^{\infty}(0, T; H^{m + 2}(\Omega))}\Big).
    \end{multline*}
\end{proposition}

\begin{remark}
In the estimate \eqref{Hmest}, we label the constant by $C_{0}$ since its actual value will play an important role in the Schauder fixed point estimates.
\end{remark}

\begin{proof}
    Existence of a solution is standard. We prove an estimate for the norms of $n$ and $\nabla n$, and then describe the general estimate for $D^{m}n$. We use that for $n \in L^{\infty}(0, T; H^{m}(\Omega))$, $n \in L^{\infty}(0, T; W^{m - 2, \infty}(\Omega))$ and $D^{m - 1}n \in L^{\infty}(0, T; L^{6}(\Omega))$ by Sobolev embedding in $\R^{d}$, $d = 2, 3$.

    \medskip

    \noindent \textbf{Estimate for $n$.} We test the PDE by $n$ and integrate by parts in $\displaystyle \int_{\Omega} n(\nabla n \cdot \nabla W)$ to obtain:
    \begin{equation*}
    \frac{1}{2} \int_{\Omega} |n(t)|^{2} = \frac{1}{2} \int_{\Omega} |n_{0}|^{2} + \frac{1}{2} \int_{0}^{t} \int_{\Omega} n^{2} \Delta W + \int_{0}^{t} \int_{\Omega} nF.
    \end{equation*}
    We estimate $\displaystyle 
    \left|\int_{0}^{t} \int_{\Omega} nF\right| \le \int_{0}^{t} \Big(\|n\|^{2}_{L^{2}(\Omega)} + \|F\|^{2}_{L^{2}(\Omega)}\Big)$. Since $m \ge 2$, we use Sobolev embedding to obtain:
    \begin{equation*}
    \|n(t)\|_{L^{2}(\Omega)}^{2} \le \|n_{0}\|_{L^{2}(\Omega)}^{2} + \|F\|^{2}_{L^{2}(0, T; L^{2}(\Omega))} + C\int_{0}^{t} \|n(s)\|_{H^{m}(\Omega)}^{2} \Big(1 + \|W(s)\|_{H^{m + 2}(\Omega)}\Big)ds.
    \end{equation*}

    \medskip

    \noindent \textbf{Estimate for $\nabla n$.} We consider $\partial_{j}n$ by differentiating the equation, after which we obtain:
    \begin{equation*}
        \partial_{t}n_{j} - \sum_{i = 1}^{3} \Big(n_{ij} W_{i} + n_{i}W_{ij} + n_{j}W_{ii} + nW_{iij}\Big) = F_{j}.
    \end{equation*}
    Next, we test by $n_{j}$ and integrate by parts in $\displaystyle \int_{\Omega} n_{ij} n_{j} W_{i} = -\frac{1}{2}\int_{\Omega} n_{j}^{2} W_{ii}$ to obtain:
    \begin{equation*}
    \frac{1}{2} \int_{\Omega} |\partial_{j}n(t)|^{2} = \frac{1}{2} \int_{\Omega} |\partial_{j}n_{0}|^{2} + \sum_{i = 1}^{3} \int_{0}^{t} \int_{\Omega} \Big(n_{i}W_{ij} + \frac{1}{2}n_{j}W_{ii} + nW_{iij}\Big) n_{j} + \int_{0}^{t} \int_{\Omega} F_{j}n_{j}.
    \end{equation*}
    Since $m + 2 \ge 4$, we obtain the estimate:
    \begin{equation*}
    \int_{\Omega} |\partial_{j}n(t)|^{2} \le \|n_{0}\|_{H^{1}(\Omega)}^{2} + \|F\|_{L^{2}(0, T; H^{1}(\Omega))}^{2} + C\int_{0}^{t} \|n(s)\|_{H^{m}(\Omega)}^{2} \Big(1 + \|W(s)\|_{H^{m + 2}(\Omega)}\Big) ds.
    \end{equation*}

    \medskip

    \noindent \textbf{Estimate for general $D^{m}n$ for $m \ge 2$.} We differentiate the equation with a general multi-index $\partial_{\alpha}$ with $|\alpha| = m$, and obtain:
    \begin{equation*}
        \partial_{t}\partial_{\alpha}n - \sum_{i = 1}^{3} \Big(\partial_{i}\partial_{\alpha}n \partial_{i}W + \partial_{\alpha} n \partial_{i}^{2}W + R_{\alpha}\Big) = \partial_{\alpha}F,
    \end{equation*}
    where $R_{\alpha}$ is a sum of terms of the form $\partial_{\beta}n \partial_{\gamma} W$ where $0 \le |\beta| \le m$ and $|\gamma| = m + 2 - |\beta|$. We then test by $\partial_{\alpha} n$, and by integrating by parts, we obtain:
    \begin{equation*}
        \frac{1}{2} \int_{\Omega} |\partial_{\alpha}n(t)|^{2} = \frac{1}{2} \int_{\Omega} |\partial_{\alpha} n_{0}|^{2} + \frac{1}{2} \int_{0}^{t} \int_{\Omega} (\partial_{\alpha} n)^{2} \Delta W + \sum_{i = 1}^{3} \int_{0}^{t} \int_{\Omega} R_{\alpha} \partial_{\alpha} n + \int_{0}^{t} \int_{\Omega} \partial_{\alpha} F \partial_{\alpha} n.
    \end{equation*}
    Since $W \in L^{\infty}(0, T; H^{m + 2}(\Omega))$ for $m \ge 2$, by Sobolev embedding,
    \begin{equation*}
    \left|\int_{\Omega} (\partial_{\alpha} n)^{2} \Delta W\right| \le \|n\|_{H^{m}(\Omega)}^{2} \|W\|_{H^{m + 2}(\Omega)}.
    \end{equation*}
    We obtain a similar estimate for the rest of the terms in $\displaystyle \int_{\Omega} R_{\alpha} \partial_{\alpha} n$, which is the sum of terms of the form $\partial_{\beta} n \partial_{\gamma} W$ for $0 \le |\beta| \le m$ and $|\gamma| = m + 2 - |\beta|$, using casework:
    \begin{itemize}
        \item Case 1: $0 \le |\beta| \le m - 2$. In this case, $\partial_{\beta} n$ is in $L^{\infty}(\Omega)$ by Sobolev embedding, so
        \begin{equation*}
        \left|\int_{\Omega} \partial_{\alpha} n \partial_{\beta} n \partial_{\gamma} W\right| \le \|\partial_{\alpha} n\|_{L^{2}(\Omega)} \|\partial_{\beta} n\|_{L^{\infty}(\Omega)} \|\partial_{\gamma} W\|_{L^{2}(\Omega)} \le C\|n\|_{H^{m}(\Omega)}^{2} \|W\|_{H^{m + 2}(\Omega)}.
        \end{equation*}
        \item Case 2: $|\beta| = m$. In this case, $|\gamma| = 2$, so since $m + 2 \ge 4$, by Sobolev embedding,
        \begin{equation*}
        \left|\int_{\Omega} \partial_{\alpha} n \partial_{\beta} n \partial_{\gamma} W\right| \le \|\partial_{\alpha}n\|_{L^{2}(\Omega)} \|\partial_{\beta}n\|_{L^{2}(\Omega)} \|\partial_{\gamma} W\|_{L^{\infty}(\Omega)} \le C\|n\|_{H^{m}(\Omega)}^{2} \|W\|_{H^{m + 2}(\Omega)}. 
        \end{equation*}
        \item Case 3: $|\beta| = m - 1$. In this case, $|\gamma| = 3$, so by Sobolev embedding and the fact $m + 2 \ge 4$, we can estimate $\partial_{\gamma}W$ in $L^{6}(\Omega)$. Similarly, we can estimate $\partial^{\beta} n$ in $L^{6}(\Omega)$ since $n \in H^{m}(\Omega)$ and $|\beta| = m - 1$. Hence,
        \begin{equation*}
        \left|\int_{\Omega} \partial_{\alpha} n \partial_{\beta} n \partial_{\gamma} W\right| \le C\|\partial_{\alpha}n\|_{L^{2}(\Omega)} \|\partial_{\beta}n\|_{L^{6}(\Omega)} \|\partial_{\gamma}W\|_{L^{6}(\Omega)} \le C\|n\|^{2}_{H^{m}(\Omega)} \|W\|_{H^{m + 2}(\Omega)}.
        \end{equation*}
    \end{itemize}
    So we obtain the final estimate:
    \begin{equation}\label{mdiffineq}
    \int_{\Omega} |\partial_{\alpha} n(t)|^{2} \le \|n_{0}\|_{H^{m}(\Omega)}^{2} + \|F\|_{L^{2}(0, T; H^{m}(\Omega))}^{2} + C\int_{0}^{t} \|n(s)\|^{2}_{H^{m}(\Omega)} \Big(1 + \|W(s)\|_{H^{m + 2}(\Omega)}\Big).
    \end{equation}

\medskip

\noindent \textbf{Conclusion of proof.} By adding the estimate \eqref{mdiffineq} over all multi-indices $0 \le |\alpha| \le m$:
\begin{equation*}
\|n(t)\|^{2}_{H^{m}(\Omega)} \le C\left(\|n_{0}\|^{2}_{H^{m}(\Omega)} + \|F\|_{L^{2}(0, T; H^{m}(\Omega))}^{2} + \int_{0}^{t} \|n(s)\|^{2}_{H^{m}(\Omega)} \Big(1 + \|W(s)\|_{H^{m + 2}(\Omega)}\Big)\right),
\end{equation*}
for a constant $C$ that does not depend on $T$. Therefore, by Gronwall's lemma (Lemma 1.4.2 in \cite{BFH18}), we have the following bound for almost every $t \in [0, T]$:
\begin{equation*}
\|n(t)\|_{H^{m}(\Omega)}^{2} \le C\Big(\|n_{0}\|_{H^{m}(\Omega)}^{2} + \|F\|^{2}_{L^{2}(0, T; H^{m}(\Omega))}\Big) \text{exp}\left(Ct\left(1 + \|W\|_{L^{\infty}(0, T; H^{m + 2}(\Omega))}\right)\right).
\end{equation*}
Then, using the equation itself, $\partial_{t}n = \nabla n \cdot \nabla W + n \Delta W + F$, and by Sobolev embedding,
\begin{equation*}
\|\partial_{t}n\|_{L^{2}(0, T; H^{m - 1}(\Omega))} \le 2\|n\|_{L^{\infty}(0, T; H^{m}(\Omega))} \|W\|_{L^{2}(0, T; H^{m + 2}(\Omega))} + \|F\|_{L^{2}(0, T; H^{m}(\Omega))},
\end{equation*}
which implies the desired bound:
\begin{multline*}
\|\partial_{t}n\|_{L^{2}(0, T; H^{m - 1}(\Omega))} \le \|F\|_{L^{2}(0, T; H^{m}(\Omega))} \\
+ C\|W\|_{L^{2}(0, T; H^{m + 2}(\Omega))} \Big(\|n_{0}\|_{H^{m}(\Omega)}^{2} + \|F\|_{L^{\infty}(0, T; H^{m}(\Omega))}^{2}\Big) \text{exp}\Big(CT(1 + \|W\|_{L^{\infty}(0, T; H^{m + 2}(\Omega))}\Big).
\end{multline*}
\end{proof}

\begin{proposition}[Continuous dependence for the linearized problem]\label{continuousdependence}
    Suppose that 
    \begin{equation*}
    W_{k} \to W \text{ in } L^{\infty}(0, T; H^{m + 2}(\Omega)), \qquad F_{k} \to F \text{ in } L^{\infty}(0, T; H^{m}(\Omega)),
    \end{equation*}
    for some integer $m \ge 2$, and let $n_{k}$ and $n$ be the associated solutions to \eqref{continuitylinear} with $W_{k}$ and $F_{k}$, and $W$ and $F$ respectively. Furthermore, suppose that $n \in L^{\infty}(0, T; H^{m + 1}(\Omega))$. Then, $n_{k} \to n$ in $L^{\infty}(0, T; H^{m}(\Omega))$.
\end{proposition}

\begin{proof}
If we subtract the equations for $n$ and $n_{k}$, we obtain:
\begin{equation*}
    \partial_{t}(n - n_{k}) - \nabla n \cdot \nabla (W - W_{k}) - \nabla (n - n_{k}) \cdot \nabla W_{k} - n \Delta (W - W_{k}) - (n - n_{k}) \Delta W_{k} = F - F_{k},
\end{equation*}
with initial data $(n - n_{k})(0) = 0$. By testing the equation and its spatial derivatives with $\partial_{\alpha}(n - n_{k})$ for $0 \le |\alpha| \le m$, we obtain an analogous estimate for the norms of the difference:
\begin{multline*}
    \|(n - n_{k})(t)\|^{2}_{H^{m}(\Omega)} \le C\Bigg(\int_{0}^{t} \|n(s)\|_{H^{m + 1}(\Omega)} \|(n - n_{k})(s)\|_{H^{m}(\Omega)} \|(W - W_{k})(s)\|_{H^{m + 2}(\Omega)} ds \\
    + \int_{0}^{t} \|(n - n_{k})(s)\|^{2}_{H^{m}(\Omega)} \|W(s)\|_{H^{m + 2}(\Omega)} ds + \int_{0}^{t} \|(n - n_{k})(s)\|_{H^{m}(\Omega)} \|(F - F_{k})(s)\|_{H^{m}(\Omega)} ds \Bigg).
\end{multline*}
Therefore,
\begin{multline*}
\|(n - n_{k})(t)\|^{2}_{H^{m}(\Omega)} \le C\Big(\|W - W_{k}\|^{2}_{L^{2}(0, T; H^{m + 2}(\Omega))} + \|F - F_{k}\|^{2}_{L^{2}(0, T; H^{m}(\Omega))}\Big) \\
+ C\int_{0}^{t} \Big(1 + \|n(s)\|^{2}_{H^{m + 1}(\Omega)} + \|W(s)\|_{H^{m + 2}(\Omega)}\Big) \|(n - n_{k})(s)\|_{H^{m}(\Omega)}^{2}.
\end{multline*}
So by Gronwall's inequality (Lemma 1.4.2 in \cite{BFH18}), for almost every $t \in [0, T]$:
\begin{multline*}
\|(n - n_{k})(t)\|_{H^{m}(\Omega)}^{2} \le C\Big(\|W - W_{k}\|^{2}_{L^{2}(0, T; H^{m + 2}(\Omega))} + \|F - F_{k}\|_{L^{2}(0, T; H^{m}(\Omega))}^{2}\Big) \\
\cdot \text{exp}\Big(Ct(1 + \|n\|_{L^{\infty}(0, T; H^{m + 1}(\Omega))}^{2} + \|W\|_{L^{\infty}(0, T; H^{m + 2}(\Omega))})\Big).
\end{multline*}
Therefore, the convergence of $n_{k}$ to $n$ in $L^{\infty}(0, T; H^{m}(\Omega))$ follows from the convergences $W_{k} \to W$ in $L^{\infty}(0, T; H^{m + 2}(\Omega))$ and $F_{k} \to F$ in $L^{\infty}(0, T; H^{m}(\Omega))$, and the regularity assumption that $n \in L^{\infty}(0, T; H^{m + 1}(\Omega))$.

\end{proof}

\subsection{A fixed point argument}\label{fixedpointsec}

Next, we use the estimates on the linearized equations in Proposition \ref{linearestimate} to establish the existence of a local solution, using the Schauder fixed point theorem. We consider the following solution space for $n$:
\begin{equation*}
X = C(0, T; H^{2}(\Omega)),
\end{equation*}
and we define the following subset of $X$:
\begin{multline}\label{BRdef}
B_{R} := \{n \in L^{\infty}(0, T; H^{3}(\Omega)) : R^{-1} \le n(t, \cdot) \le R \text{ for $t \in [0, T]$}, \\
n(0, \cdot) = n_{0}, \|n\|_{L^{\infty}(0, T; H^{3}(\Omega))} \le R, \|\partial_{t}n\|_{L^{2}(0, T; H^{2}(\Omega))} \le R\},
\end{multline}
where we recall that the smooth initial data $n_{0}(x)$ satisfies $n_{0} \ge m := \min_{x \in \overline{\Omega}} n_{0}(x) > 0$ by assumption. By the Aubin-Lions compactness theorem, $B_{R}$ is a compact set in $X$, and it is convex and closed under the topology of $X$.

We define a map $\Phi: B_{R} \to X$, as follows. Given $\tilde{n} \in B_{R}$, note that $a\tilde{n}^{\gamma} \in C(0, T; H^{3}(\Omega))$ also, and let $\tilde{W}$ be the unique solution in $C(0, T; H^{5}(\Omega))$ to:
\begin{equation}\label{phidef1}
-\mu \Delta \tilde{W} + \tilde{W} = a\tilde{n}^{\gamma} \ \ \text{ on } \Omega, \qquad \nabla \tilde{W} \cdot \bd{n} |_{\partial \Omega} = 0.
\end{equation}
and then let $n$ be the unique solution to
\begin{equation}\label{phidef2}
\partial_{t}n + \text{div}(n\nabla \tilde{W}) = \alpha \tilde{n} - \beta \tilde{n}^{1 + \gamma \theta},
\end{equation}
satisfying the regularity $n \in C(0, T; H^{3}(\Omega))$ and $\partial_{t} n \in C(0, T; H^{2}(\Omega))$, whose existence and uniqueness is guaranteed by Proposition \ref{linearestimate}. 

We will choose $R > 0$ so that
\begin{equation}\label{Rchoice}
R := \max\Big((2C_{0})^{-1}\|n_{0}\|_{H^{3}(\Omega)}, 2M, 2m^{-1}\Big), \qquad \text{ for the constant $C_{0}$ from \eqref{Hmest}}.
\end{equation}
To apply the Schauder's fixed point theorem, which would then allow us to conclude the existence of a solution $n \in B_{R}$ locally up to a sufficiently small time $T > 0$, we just need to verify the following claims:
\begin{itemize}
    \item \textbf{Claim 1.} For $R$ sufficiently large and $T > 0$ sufficiently small, $\Phi: B_{R} \to B_{R}$. This claim consists of:
    \begin{itemize}
        \item \textbf{Claim 1A.} For $\tilde{n} \in B_{R}$, $\|n\|_{L^{\infty}(0, T; H^{3}(\Omega))} \le R$ and $\|\partial_{t}n\|_{L^{2}(0, T; H^{2}(\Omega))} \le R$, for $T$ sufficiently small, where $n := \phi(\tilde{n})$. 
        \item \textbf{Claim 1B.} For $\tilde{n} \in B_{R}$, $n := \phi(\tilde{n})$ satisfies $1/R \le n(t, \cdot) \le R$ for all $t \in [0, T]$ for $T$ sufficiently small.
    \end{itemize}
    \item \textbf{Claim 2.} $\Phi: B_{R} \to B_{R}$ is a continuous map with respect to the topology of $X$.
\end{itemize}

\medskip

We verify these claims in the following series of propositions.

\begin{proposition}[Claim 1A: Boundedness of solution norms]\label{1A}
    There exists $R > 0$ sufficiently large and $T > 0$ sufficiently small such that for all $\tilde{n} \in B_{R}$ and for $n := \Phi(\tilde{n})$:
    \begin{equation*}
        \|n\|_{L^{\infty}(0, T; H^{3}(\Omega))} \le R, \qquad \|\partial_{t}n\|_{L^{2}(0, T; H^{2}(\Omega))} \le R.
    \end{equation*}
\end{proposition}
\begin{proof}
To establish this estimate, we use the estimates in Proposition \ref{linearestimate}. We hence will estimate the norm $\|\alpha \tilde{n} + \beta \tilde{n}^{1 + \gamma \theta}\|_{H^{3}(\Omega)}$ for $\tilde{n} \in B_{R}$; since it is possible that $1 \le 1 + \gamma \theta < 3$, in order to estimate the $L^{2}$ norm of the second and third spatial derivatives:
\begin{equation*}
\partial_{j}\partial_{k}(\alpha \tilde{n} + \beta \tilde{n}^{1 + \gamma \theta}) = \alpha \partial_{j}\partial_{k} \tilde{n} + \beta (1 + \gamma \theta) \gamma \theta \tilde{n}^{\gamma \theta - 1} \partial_{j}\tilde{n} \partial_{k}\tilde{n} + \beta (1 + \gamma \theta) \tilde{n}^{\gamma \theta} \partial_{j}\partial_{k}\tilde{n},
\end{equation*}
\begin{align*}
\partial_{j}\partial_{k}\partial_{l}&(\alpha \tilde{n} + \beta \tilde{n}^{1 + \gamma \theta}) = \alpha \partial_{j}\partial_{k}\partial_{l}\tilde{n} + \beta(1 + \gamma \theta) \gamma \theta (\gamma \theta - 1) \tilde{n}^{\gamma \theta - 2} \partial_{j}\tilde{n} \partial_{k}\tilde{n} \partial_{l}\tilde{n} \\
&+ \beta (1 + \gamma \theta) \gamma \theta \tilde{n}^{\gamma \theta - 1} \Big(\partial_{j}\partial_{k} \tilde{n} \partial_{l}\tilde{n} + \partial_{j}\partial_{l}\tilde{n}\partial_{k}\tilde{n} + \partial_{k}\partial_{l}\tilde{n}\partial_{j}\tilde{n}\Big) + \beta (1 + \gamma \theta) \tilde{n}^{\gamma \theta} \partial_{j}\partial_{k}\partial_{l}\tilde{n},
\end{align*}
we must use the lower bound $\tilde{n} \ge R^{-1} > 0$ for $\tilde{n} \in B_{R}$, to estimate for example $\tilde{n}^{\gamma \theta - 1}$ and $\tilde{n}^{\gamma \theta - 2}$, which can have negative exponents, in addition to the upper bound $\tilde{n} \le R$ to estimate positive powers of $\tilde{n}$. Using the Sobolev embeddings $H^{3}(\Omega) \subset C(\Omega)$,
\begin{equation*}
\|\alpha \tilde{n} + \beta \tilde{n}^{1 + \gamma \theta}\|_{H^{3}(\Omega)} \le C\left((1 + R^{\gamma \theta} )\|\tilde{n}\|_{H^{3}(\Omega)} + R^{|\gamma \theta - 1|}\|\tilde{n}\|^{2}_{H^{3}(\Omega)} + R^{|\gamma \theta - 2|} \|\tilde{n}\|^{3}_{H^{3}(\Omega)}\right).
\end{equation*}
Therefore,
\begin{multline*}
\|\alpha \tilde{n} + \beta \tilde{n}^{1 + \gamma \theta}\|_{L^{2}(0, T; H^{3}(\Omega))} \le T^{1/2} \|\alpha \tilde{n} + \beta \tilde{n}^{1 + \gamma \theta}\|_{L^{\infty}(0, T; H^{3}(\Omega))} \\
\le CT^{1/2}\Big((1 + R^{\gamma \theta}) \|\tilde{n}\|_{L^{\infty}(0, T; H^{3}(\Omega))} + R^{|\gamma \theta - 1|} \|\tilde{n}\|^{2}_{L^{\infty}(0, T; H^{3}(\Omega))} + R^{|\gamma \theta - 2|} \|\tilde{n}\|_{L^{\infty}(0, T; H^{3}(\Omega))}^{3}\Big).
\end{multline*}
By the Calderon-Zygmund estimates in Propositions \ref{CZest} and \ref{nonnegativeW} and similar estimates on higher derivatives of $a\tilde{n}^{\gamma}$ for $\tilde{n} \in B_{R}$:
\begin{multline*}
\|\tilde{W}\|_{L^{2}(0, T; H^{5}(\Omega))} \le CT^{1/2} \Big(\|\tilde{n}^{\gamma}\|_{L^{\infty}(0, T; H^{3}(\Omega))} + \|W\|_{L^{\infty}(0, T; L^{2}(\Omega))}\Big) \le CT^{1/2}\Big(R^{\gamma - 1}\|\tilde{n}\|_{L^{\infty}(0, T; H^{3}(\Omega))} \\
+ R^{|\gamma - 2|}\|\tilde{n}\|_{L^{\infty}(0, T; H^{3}(\Omega))}^{2} + R^{|\gamma - 3|}\|\tilde{n}\|_{L^{\infty}(0, T; H^{3}(\Omega))}^{3} + \|\tilde{n}\|_{L^{\infty}(0, T; H^{3}(\Omega))}^{\gamma}\Big).
\end{multline*}

Recall the choice of $R$ in \eqref{Rchoice}. Then, by the estimate in Proposition \ref{linearestimate}, and by combining the definition of $B_{R}$ \eqref{BRdef} with the previous estimates, we obtain:
\begin{multline}\label{H3depend}
\|n\|_{L^{\infty}(0, T; H^{3}(\Omega))} \le C\Big(\|n_{0}\|_{H^{3}(\Omega)} + \|\alpha \tilde{n} + \beta \tilde{n}^{1 + \gamma \theta}\|_{L^{2}(0, T; H^{3}(\Omega))}\Big) \text{exp}\Big(CT(1 + \|W\|_{L^{\infty}(0, T; H^{5}(\Omega))})\Big) \\
\le C_{0}\left(\frac{R}{2C_{0}} + CT^{1/2}\left(R + R^{3 + |\gamma \theta - 2|} + R^{2 + |\gamma \theta - 1|} + R^{1 + \gamma \theta}\right) \right) \\
\cdot \exp\left(CT(1 + C(R^{\gamma} + R^{2 + |\gamma - 2|} + R^{3 + |\gamma - 3|})\right).
\end{multline}
From Proposition \ref{linearestimate}, we have the estimate:
\begin{small}
\begin{multline*}
\|\partial_{t}n\|_{L^{2}(0, T; H^{2}(\Omega))} \le \|\alpha \tilde{n} + \beta \tilde{n}^{1 + \gamma \theta}\|_{L^{2}(0, T; H^{3}(\Omega))} \\
+ C\|W\|_{L^{2}(0, T; H^{5}(\Omega))} \Big(\|n_{0}\|_{H^{3}(\Omega)} + \|\alpha \tilde{n} + \beta \tilde{n}^{1 + \gamma \theta}\|_{L^{2}(0, T; H^{3}(\Omega))} \Big) \text{exp}\Big(CT(1 + \|W\|_{L^{\infty}(0, T; H^{5}(\Omega))})\Big)
\end{multline*}
\end{small}
and hence:
\begin{multline}\label{H3depend2}
\|\partial_{t}n\|_{L^{2}(0, T; H^{2}(\Omega))} \le CT^{1/2}\left(R + R^{3 + |\gamma \theta - 2|} + R^{2 + |\gamma \theta - 1|} + R^{1 + \gamma \theta}\right) \\
+ CT^{1/2}(R^{\gamma} + R^{2 + |\gamma - 2|} + R^{3 + |\gamma - 3|}) \left(\|n_0\|_{H^{3}(\Omega)} + CT^{1/2}\left(R + R^{3 + |\gamma \theta - 2|} + R^{2 + |\gamma \theta - 1|} + R^{1 + \gamma \theta}\right)\right) \\
\cdot \text{exp}(CT(1 + C(R^{\gamma} + R^{2 + |\gamma - 2|} + R^{3 + |\gamma - 3|})).
\end{multline}
Therefore, we can choose $T > 0$ sufficiently small such that for $n := \Phi(\tilde{n})$, we have the estimates $\|n\|_{L^{\infty}(0, T; H^{3}(\Omega))} \le R$ and $\|\partial_{t}n\|_{L^{2}(0, T; H^{2}(\Omega))} \le R$ for all $\tilde{n} \in B_{R}$. Since we chose $R$ so that $\|n_{0}\|_{H^{3}(\Omega)} \le R/(2C_0)$ which is only crucial for the first term in \eqref{H3depend} which is the only term in \eqref{H3depend} and \eqref{H3depend2} that does not have a factor of $T$ to some positive power, the time $T > 0$ depends only on $\|n_{0}(x)\|_{H^{3}(\Omega)}$.
\end{proof}

\medskip

Next, we have the following proposition which bounds the change in the spatial minimum and maximum of the solution $n$ in time, to the linearized equation.

\begin{proposition}[Claim 1B: Maximum and minimum bounds]\label{1B}
There exists $R$ sufficiently large and $T > 0$ sufficiently small such that given $\tilde{n} \in B_{R}$, the solution $n := \Phi(\tilde{n})$ defined by \eqref{phidef1} and \eqref{phidef2} satisfies $1/R \le n \le R$ for all $t \in [0, T]$.
\end{proposition}
\begin{proof}
Since this is a linear equation in $n$, the result will follow from a maximum/minimum principle argument. Let $m := \min_{x \in \overline{\Omega}} n_{0}(x)$ and $M := \max_{x \in \overline{\Omega}} n_{0}(x)$ be minimum and maximum values of the initial data. We claim that for all $\tilde{n} \in B_{R}$ such that $\tilde{n}$ and $\partial_{t}\tilde{n}$ are spatially smooth on $\overline{\Omega}$, we have that
\begin{equation*}
me^{-bt} - bt \le \min_{x \in \overline{\Omega}} n(t, \cdot) \le \max_{x \in \overline{\Omega}} n(t, \cdot) \le e^{Bt}(M + Bt),
\end{equation*}
for all $t \ge 0$, for $b > 0, B > 0$. It suffices to consider $\tilde{n} \in B_{R}$ with $\tilde{n}, \partial_{t}\tilde{n}$ spatially smooth, since the minimum/maximum inequality holds for general $\tilde{n} \in B_{R}$ by approximation. 

\medskip

\noindent \textbf{Verification of maximum bound.} Note that we can rewrite the first equation as:
\begin{equation}\label{maxReqn}
\partial_{t}n - \nabla n \cdot \nabla \tilde{W} = \alpha \tilde{n} - \beta \tilde{n}^{1 + \gamma \theta} + \mu^{-1}n(\tilde{W} - a\tilde{n}^{\gamma}).
\end{equation}
By Proposition \ref{Wbound} and the fact that $\tilde{n} \in B_{R}$, $0 \le \tilde{W}(t, \cdot) \le a\left(\max_{t \in [0, T], x \in \overline{\Omega}} \tilde{n}\right)^{\gamma} \le aR^{\gamma}$, and the function $\alpha z - \beta z^{1 + \gamma \theta}$ is uniformly bounded from above for $z \ge 0$. Hence, for our choice of $R$ in \eqref{Rchoice}, we can define $B := \max\Big(a\mu^{-1}R^{\gamma}, 1 + \max_{z \ge 0} (\alpha z - \beta z^{1 + \gamma \theta})\Big)$ which satisfies:
\begin{equation*}
\alpha z - \beta z^{1 + \gamma \theta} < B \ \ \text{ for all } z \ge 0, \qquad 
0 \le \mu^{-1}\tilde{W}(t, x) \le a\mu^{-1}R^{\gamma} \le B \ \ \text{for all $t \in [0, T]$, $x \in \overline{\Omega}$}.
\end{equation*}
Therefore, $N := e^{-Bt}n - Bt$ satisfies $\partial_{t}N - \nabla N \cdot \nabla \tilde{W} < 0$, and we observe that $N$ is spatially smooth by regularity theory, since $\tilde{n}$ is. Then, $N$ cannot attain a strict maximum at $t > 0$, since at such a strict maximum, $\nabla N \cdot \nabla \tilde{W} = 0$ and hence $\partial_{t} N < 0$ at this strict maximum, which gives a contradiction. Hence, $N(t, \cdot) \le \max_{x \in \overline{\Omega}} N(0, \cdot) = M$, so that 
\begin{equation}\label{nBmax}
n(t, x) \le e^{Bt}(M + Bt), \qquad \text{ for all $t \ge 0$.}
\end{equation}

\medskip

\noindent \textbf{Verification of minimum bound.} We consider \eqref{maxReqn} and note that $\tilde{W} \ge 0$ by Proposition \ref{nonnegativeW}. Since $\tilde{n} \in B_{R}$, we have $1/R \le \min_{t \in [0, T], x \in \overline{\Omega}} \tilde{n} \le \max_{t \in [0, T], x \in \overline{\Omega}} \tilde{n} \le R$.
Thus, there exists a constant $b > a\mu^{-1}R^{\gamma}$ such that for all $z \in [0, R]$, $\alpha z - \beta z^{1 + \gamma \theta} > -b$. Hence, $\overline{n} := e^{bt}(n + bt)$ satisfies $\partial_{t}\overline{n} - \nabla \overline{n} \cdot \nabla \tilde{W} > 0$, and by a similar argument to the maximum principle, we deduce that $\overline{n}$ cannot attain a strict minimum for $t > 0$. Hence, $\overline{n}(t, x) \ge \min_{x \in \overline{\Omega}} \overline{n}(0, \cdot) = m$, so
\begin{equation}\label{nbmin}
n(t, x) \ge me^{-bt} - bt.
\end{equation}

\medskip

\noindent \textbf{Conclusion of the proof.} Since $R > \max(2m^{-1}, 2M)$ by \eqref{Rchoice}, we can choose $T > 0$ sufficiently small such that $1/R \le me^{-bT} - bT \le e^{BT}(M + BT) \le R$, and then the desired result follows from \eqref{nBmax} and \eqref{nbmin}. Note that the choice of $T > 0$ here depends on the minimum and maximum values $m$ and $M$ of the initial data $n_{0}(x)$.

\end{proof}

Finally, we establish continuity of the map $\Phi$, which will be a direct consequence of the result on continuous dependence for the linearized equation, Proposition \ref{continuousdependence}.

\begin{proposition}
The map $\Phi: B_{R} \to B_{R}$ is a continuous map with respect to the topology of $X := C(0, T; H^{2}(\Omega))$.
\end{proposition}

\begin{proof}
    Suppose $\tilde{n}_{k} \to \tilde{n}$ in $C(0, T; H^{2}(\Omega))$, with $\tilde{n}_{k} \in B_{R}$ for all $k$, $\tilde{n} \in B_{R}$. We want to show that the corresponding solutions $n_{k} \to n$ in $C(0, T; H^{2}(\Omega))$. We observe from Proposition \ref{1A} that $n \in L^{\infty}(0, T; H^{3}(\Omega))$. Therefore, it suffices to show that
    \begin{equation}\label{continuous1}
        \tilde{W}_{k} \to \tilde{W} \ \text{ in } L^{\infty}(0, T; H^{4}(\Omega)), \qquad 
        \alpha \tilde{n}_{k} - \beta \tilde{n}_{k}^{1 + \gamma \theta} \to \alpha \tilde{n} - \beta \tilde{n}^{1 + \gamma \theta}, \ \text{ in } L^{\infty}(0, T; H^{2}(\Omega)).
    \end{equation}
    by Proposition \ref{continuousdependence}. For this purpose, it suffices to verify the claim that:
    \begin{equation}\label{powercontinuous}
    \tilde{n}_{k}^{p} \to \tilde{n}^{p}, \qquad \text{ in } C(0, T; H^{2}(\Omega)) \text{ for any } p \ge 1,
    \end{equation}
    as this will directly imply \eqref{continuous1}, via Calder\'{o}n-Zygmund estimates (Propositions \ref{CZest} and \ref{nonnegativeW}).

    We hence complete the proof by verifying \eqref{powercontinuous}. We note that $\tilde{n}_{k} \to \tilde{n}$ in $C(0, T; C(\overline{\Omega}))$, $\nabla \tilde{n}_{k} \to \nabla \tilde{n}$ in $C(0, T; L^{6}(\Omega))$, and $\nabla^{2} \tilde{n}_{k} \to \nabla^{2} n$ in $C(0, T; L^{2}(\Omega))$ by Sobolev embedding. Therefore, we directly obtain by using the mean-value theorem and the fact that $1/R \le \tilde{n}_{k} \le R$, $1/R \le \tilde{n} \le R$, since $\tilde{n}_{k}, \tilde{n} \in B_{R}$, that $\lim_{k \to \infty} \|\tilde{n}^{p} - \tilde{n}_{k}^{p}\|_{C(0, T; L^{2}(\Omega))} = 0$. Similarly, we calculate $\nabla \tilde{n}_{k}^{p} = p\tilde{n}_{k}^{p - 1} \nabla \tilde{n}_{k}$ and $\partial_{i}\partial_{j} (\tilde{n}_{k}^{p}) = p(p - 1) \tilde{n}^{p - 2}_{k} \partial_{i}\tilde{n}_{k} \partial_{j} \tilde{n}_{k} + p\tilde{n}_{k}^{p - 1} \partial^{2}_{ij} \tilde{n}_{k}$, and hence:
    \begin{equation*}
    \lim_{k \to \infty} \|\nabla(\tilde{n}_{k}^{p} - \tilde{n}^{p})\|_{C(0, T; L^{2}(\Omega))} = 0, \qquad \lim_{k \to \infty} \|\nabla^{2}(\tilde{n}_{k}^{p} - \tilde{n}^{p})\|_{C(0, T; L^{2}(\Omega))} = 0,
    \end{equation*}
    where we use the fact $1/R \le \tilde{n}_{k} \le R$ and $1/R \le \tilde{n} \le R$. For estimating terms such as $|\tilde{n}^{p - 1} - \tilde{n}_{k}^{p - 1}|$ and $|\tilde{n}^{p - 2} - \tilde{n}_{k}^{p - 2}|$, it is important that we have both a uniform lower bound $\tilde{n}_{k}, \tilde{n} \ge 1/R$ and upper bound $\tilde{n}_{k}, \tilde{n} \le R$, which will allow us to use the mean value theorem to estimate these terms, since $p - 1$ and $p - 2$ can be either positive or negative depending on $p$. This completes the proof of \eqref{powercontinuous}, and hence also the proof of the proposition.
\end{proof}

Then, a direct application of the Schauder fixed point argument gives us the following result for local existence, where we emphasize that the time of local existence $T_{loc} > 0$ depends only on the norm $\|n_{0}\|_{H^{3}(\Omega)}$, minimum, and maximum of the initial data.

\begin{proposition}\label{localexistence}
    Let $n_{0}(x) \in C^{\infty}(\overline{\Omega})$, and define $m := \min_{x \in \overline{\Omega}} n_{0}(x)$ and $M := \max_{x \in \overline{\Omega}} n_{0}(x)$,     where we assume that $m > 0$. Then, there exists a time $T_{loc} > 0$ depending only on $\|n_{0}\|_{H^{3}(\Omega)}$, $m$, and $M$, such that there exists a unique strong solution $(n, W)$, with $n \in C(0, T_{loc}; H^{2}(\Omega))$, $\partial_{t}n \in C(0, T_{loc}; H^{1}(\Omega))$, and $W \in C(0, T_{loc}; H^{4}(\Omega))$. 
\end{proposition}

\section{Global existence of strong solutions}\label{globalsec}

In this section, we use the local existence result in Proposition \ref{localexistence} to prove global existence of unique smooth solutions for smooth initial data $n_{0}(x)$ with sufficiently small $L^{6}(\Omega)$ norm. We note that the choice of $L^{6}(\Omega)$ here is mostly for concreteness. We obtain global a priori estimates which show that a solution exists for all time, not just locally in time, which will allow us to prolong the time of existence of the solution. 

To do this, we show a first a priori estimate where if the solution is sufficiently far away from vacuum (namely $n_{0}(x) \ge \eta - \epsilon$ for $\eta$ defined in Proposition \ref{minprinciple} and $\epsilon > 0$ sufficiently small), the estimate for the first derivative norms is dissipative (see Proposition \ref{W1pcontrol}), which gives global control of the norms of the gradient. This is in the spirit of results on global well-posedness for damped 3D compressible Euler equations for example, see Theorems 5.1 and 5.2 in \cite{Sideris}. 

In the case where the smooth positive initial data $n_{0}(x)$ takes values below $\eta - \epsilon$ (so that the dissipative a priori estimate in Proposition \ref{W1pcontrol} does not directly apply), we show using a more general (non-dissipative) a priori estimate in Proposition \ref{W1pgeneral} that the solution exists on a time interval $[0, T]$ for some $T > 0$. By the fact that the minimum value increases according to the minimum principle involving $\overline{n}(t)$ in Proposition \ref{minprinciple}, we can choose $T$ so that at time $T$, $n(T, \cdot)$ is bounded below by $\eta - \epsilon$, which then allows us to apply the a priori bound in Theorem \ref{W1pcontrol} to obtain global existence. Thus, the a priori bounds in Proposition \ref{W1pcontrol} and Proposition \ref{W1pgeneral} will be the main focus of this section, as they will naturally lead to a proof of global existence for initial data with small norm of the gradient in Theorem \ref{global}.

\subsection{Global a priori estimates for positive initial data}\label{globalpositive}

In this section, assume that the initial data satisfies $n_{0}(x) \ge m > 0$, for some positive constant $m := \min_{x \in \overline{\Omega}} n_{0}(x)$. We already have local existence of $n \in C(0, T; H^{2}(\Omega))$ and $W \in C(0, T; H^{4}(\Omega))$ for a sufficiently small $T > 0$. We want to obtain global estimates on the norms of $n$, with the goal of proving Theorem \ref{global}. We accomplish this by establishing the following result, where we recall that $d$ is the spatial dimension $d = 2, 3$.

\begin{proposition}[Dissipative estimate of $L^{p}$ gradient norm for small data]\label{W1pcontrol} 
Suppose that $(n, W)$ with regularity $n \in C(0, T; W^{1, p}(\Omega))$, $\partial_{t}n \in C(0, T; W^{1, p}(\Omega))$, and $W \in C(0, T; W^{3, p}(\Omega))$ for some fixed even positive integer $p > d$, is a strong solution to \eqref{equations} for smooth initial data $n_{0}(x) > 0$, with $\min_{(x, t) \in [0, T] \times \overline{\Omega}} n(t, x) \ge \eta - \epsilon$ where $\epsilon > 0$ is any positive number for which
\begin{equation}\label{epsilondef}
\alpha - \beta n^{\gamma \theta} - a\mu^{-1} n^{\gamma} \le \frac{\alpha}{4}\min(\gamma \theta, 1) \qquad \text{ for all } n \ge \eta - \epsilon,
\end{equation}
and $\eta$ is defined in \eqref{nlowerstar} as the unique solution to $\alpha - \beta n^{\gamma \theta} - a\mu^{-1}n^{\gamma} = 0$. We also define $\max_{[0, T] \times \overline{\Omega}} n(t, x) := n_{max, T}$. Then, there exists a constant $C(n_{max, T}, a, \gamma)$ depending only on those parameters listed, such that if $\delta$ and $\mu > 0$ satisfy
\begin{equation*}
C(n_{max, T}, a, \gamma) (\delta + \mu^{-1}) \le \frac{\alpha}{8}\min(\gamma \theta, 1),
\end{equation*}
and we have the a priori bound that
\begin{equation}\label{aprioridelta}
\left(\sum_{j = 1}^{d} \|\partial_{i} n\|_{C(0, T; L^{p}(\Omega))}^{p}\right)^{1/p} < \delta,
\end{equation}
then the following inequality holds for all $t \in [0, T]$ and $1 \le j \le d$:
\begin{equation*}
\sum_{j = 1}^{d} \int_{\Omega} (n_{j}(t))^{p} + \frac{\alpha}{8} \min(\gamma \theta, 1) p \sum_{j = 1}^{d} \int_{0}^{t} \int_{\Omega} (n_{j}(s))^{p} \le \sum_{j = 1}^{d} \int_{\Omega} n_{0, j}^{p}.
\end{equation*}
\end{proposition}

\medskip

Before proving this result, we show the following technical bound.

\begin{lemma}\label{dissipativelemma}
Suppose that a strictly positive $n \in C(0, T; W^{1, p}(\Omega))$ for a positive even integer $p > d$, with an upper bound $\max_{[0, T] \times \overline{\Omega}} n(t, x) := n_{max, T}$, satisfies 
\begin{equation*}
-\mu \Delta W + W = an^{\gamma} \text{ on } \Omega, \qquad \nabla W \cdot \bd{n}|_{\partial \Omega} = 0.
\end{equation*}
Then,
\begin{equation}\label{gradp}
\|\nabla W\|_{W^{2, p}(\Omega)} \le C(p, n_{max, T}, a, \gamma) \left(\sum_{j = 1}^{d} \|\partial_{j}n\|^{p}_{L^{p}(\Omega)}\right)^{1/p},
\end{equation}
\begin{equation}\label{gradLip}
\|\nabla W\|_{W^{1, \infty}(\Omega)} \le C(n_{max, T}, a, \gamma) \left(\sum_{j = 1}^{d} \|\partial_{j} n\|_{L^{p}(\Omega)}^{p}\right)^{1/p}.
\end{equation}
\end{lemma}

\begin{proof}
Let $n_{avg} = \displaystyle \frac{1}{|\Omega|} \int_{\Omega} n(t, x) dx$, and note that for $W_{avg} = a(n_{avg})^{\gamma}$,
\begin{equation*}
-\mu \Delta (W - W_{avg}) + (W - W_{avg}) = an^{\gamma} - an_{avg}^{\gamma} \text{ on } \Omega, \qquad \nabla (W - W_{avg}) \cdot \bd{n}|_{\partial \Omega} = 0.
\end{equation*}
Since $n - n_{avg}$ has mean zero, by Poincar\'{e}'s inequality:
\begin{equation*}
\|n - n_{avg}\|_{W^{1, p}(\Omega)} \le \|\nabla(n - n_{avg})\|_{L^{p}(\Omega)} = \|\nabla n \|_{L^{p}(\Omega)},
\end{equation*}
and hence by Calder\'{o}n-Zygmund estimates (Propositions \ref{CZest} and \ref{nonnegativeW}), 
\begin{align*}
\|W - W_{avg}\|_{W^{3, p}(\Omega)} \le C_{p} \|an^{\gamma} - an^{\gamma}_{avg}\|_{W^{1, p}(\Omega)} &\le C(p, n_{max, T}, a, \gamma) \|n - n_{avg}\|_{W^{1, p}(\Omega)} \\
&\le C(p, n_{max, T}, a, \gamma) \|\nabla n\|_{L^{p}(\Omega)}.
\end{align*}
Since $W_{avg}$ is a constant, we directly obtain that $\nabla W = \nabla (W - W_{avg})$ and hence:
\begin{equation*}
\|\nabla W\|_{W^{2, p}(\Omega)} \le C(p, n_{max, T}, a, \gamma) \|\nabla n\|_{L^{p}(\Omega)}.
\end{equation*}
To obtain the estimate \eqref{gradLip} and eliminate the dependence on $p$, note that $p \ge 4$ since it is even and $p > d$. Hence, we can use a specific $L^{4}(\Omega)$ Calder\'{o}n-Zygmund constant and obtain:
\begin{equation*}
\|\nabla W\|_{W^{1, \infty}(\Omega)} \le C\|\nabla W\|_{W^{2, 4}(\Omega)} \le C(n_{max, T}, a, \gamma) \|\nabla n\|_{L^{4}(\Omega)} \le C(n_{max, T}, a, \gamma) \|\nabla n\|_{L^{p}(\Omega)}.
\end{equation*}
\end{proof}

Now, we have all of the ingredients needed to prove the essential a priori estimate in Proposition \ref{W1pcontrol}.

\begin{proof}
The proof will follow from \textit{a priori estimates} on the first spatial derivatives of $n$. We note that for convenience, we assume that $p > d$ is also an even positive integer so that $x^{p}$ is always a nonnegative quantity for any $x \in \R$.

\medskip

\noindent \textbf{Estimate on the gradient of the density.} Taking the gradient of the equation for the evolution of the density and the potential, we obtain
\begin{equation}\label{ni}
\partial_{t}n_{j} - \sum_{i = 1}^{d} \Big(n_{ij}W_{i} + n_{i}W_{ij} + n_{j}W_{ii} + nW_{iij}\Big) = \alpha n_{j} - \beta(\gamma \theta + 1)n^{\gamma \theta} n_{j},
\end{equation}
\begin{equation}\label{Wiij}
\sum_{i = 1}^{d} W_{iij} = \mu^{-1}\Big(W_{j} - a\gamma n^{\gamma - 1}n_{j}\Big).
\end{equation}
We test \eqref{ni} by $p n_{j}^{p - 1}$ for $1 \le p < \infty$, and we estimate the right-hand side of the result:
\begin{multline}\label{gradequality}
\int_{\Omega} (n_{j}(t))^{p} = \int_{\Omega} n_{0, j}^{p} + p\sum_{i = 1}^{d}\int_{0}^{t} \int_{\Omega} \Big(n_{ij}n_{j}^{p - 1} W_{i} + n_{i}n_{j}^{p - 1}W_{ij} + n_{j}^{p} W_{ii} + nn_{j}^{p - 1} W_{iij}\Big) \\
+ \alpha p \int_{0}^{t} \int_{\Omega} (n_{j}(s))^{p} - \beta (\gamma \theta + 1) p \int_{0}^{t} \int_{\Omega} n^{\gamma\theta} n_{j}^{p}.
\end{multline}
We rewrite some of the terms on the right-hand side as follows, recalling the assumption on the choice of $\epsilon$ in \eqref{epsilondef}.

\medskip

\noindent \textbf{Terms 2 and 4.} By integration by parts, $\displaystyle p\sum_{i = 1}^{d} \int_{0}^{t} \int_{\Omega} n_{ij}n_{j}^{p - 1} W_{i} = - \int_{0}^{t} \int_{\Omega} n_{j}^{p} \Delta W$. Hence, combining this with the fourth term and using \eqref{equations}, we obtain:
    \begin{align*}
    p\sum_{i = 1}^{d} \int_{0}^{t} \int_{\Omega} \Big(n_{ij} n_{i}^{p - 1} W_{i} + n_{j}^{p}W_{ii}\Big) &= (p - 1) \int_{0}^{t} \int_{\Omega} n_{j}^{p} \Delta W \\
    &= \mu^{-1}(p - 1) \int_{0}^{t} \int_{\Omega} n_{j}^{p} W - a\mu^{-1}(p - 1) \int_{0}^{t} \int_{\Omega} n_{j}^{p} n^{\gamma}.
    \end{align*}

\medskip

\noindent \textbf{Term 5.} We use the equation for $W_{iij}$ in \eqref{Wiij} to conclude that
    \begin{equation*}
    p \sum_{i = 1}^{d} \int_{0}^{t} \int_{\Omega} nn_{j}^{p - 1} W_{iij} = \mu^{-1}p \int_{0}^{t} \int_{\Omega} nn_{j}^{p - 1} W_{j} - a\gamma \mu^{-1} p \int_{0}^{t} \int_{\Omega} n^{\gamma} n_{j}^{p}.
    \end{equation*}

\medskip

Hence, we can rewrite \eqref{gradequality} as

\begin{multline}\label{simplifyfirst}
\int_{\Omega} (n_{j}(t))^{p} = \int_{\Omega} n_{0, j}^{p} + p\sum_{i = 1}^{d} \int_{0}^{t} \int_{\Omega} n_{i}n_{j}^{p - 1} W_{ij} + p\int_{0}^{t} \int_{\Omega} \Big(\alpha - \beta n^{\gamma \theta} - a\gamma \mu^{-1}n^{\gamma}\Big) n_{j}^{p} \\
+ \mu^{-1}(p - 1) \int_{0}^{t} \int_{\Omega} n_{j}^{p}W + \mu^{-1}p \int_{0}^{t} \int_{\Omega} nn^{p - 1}_{j}W_{j} - a\mu^{-1}(p - 1) \int_{0}^{t} \int_{\Omega} n^{\gamma} n_{j}^{p} - \beta\gamma \theta p \int_{0}^{t} \int_{\Omega} n^{\gamma \theta} n^{p}_{j}.
\end{multline}
Using the assumption on $\epsilon$ in \eqref{epsilondef}, the fact that $n \ge \eta - \epsilon$ for all $t \in [0, T]$ and $x \in \overline{\Omega}$, and the fact that $p > d \ge 2$ is even and hence $\displaystyle \frac{2p}{3} < p - 1$, we estimate:
\begin{equation*}
p \int_{0}^{t} \int_{\Omega} \left(\alpha - \beta n^{\gamma \theta} - a\gamma \mu^{-1}n^{\gamma}\right)n^{p}_{j} \le \frac{\alpha}{4}\min(\gamma \theta, 1)p \int_{0}^{t} \int_{\Omega} n_{j}^{p} \qquad \text{(since $\gamma \ge 1$)},
\end{equation*}
\begin{align*}
-a\mu^{-1}(p - 1) \int_{0}^{t} \int_{\Omega} n^{\gamma} n_{j}^{p} - \beta \gamma \theta p \int_{0}^{t} \int_{\Omega} n^{\gamma \theta} n_{j}^{p} &\le \frac{2}{3}\min(\gamma \theta, 1) p \int_{0}^{t} \int_{\Omega} (-a\mu^{-1}n^{\gamma} - \beta n^{\gamma \theta})n^{p}_{j} \\
&\le -\frac{\alpha}{2}\min(\gamma \theta, 1) p \int_{0}^{t} \int_{\Omega} n^{p}_{j}.
\end{align*}
We thus obtain from \eqref{simplifyfirst} the following, after summing over $j = 1$ to $j = d$:
\begin{multline}\label{apriorimid}
\sum_{j = 1}^{d} \int_{\Omega} (n_{j}(t))^{p} \le \sum_{j = 1}^{d} \int_{\Omega} n_{0, j}^{p} + p \sum_{i, j = 1}^{d} \int_{0}^{t} \int_{\Omega} n_{i}n_{j}^{p - 1} W_{ij} + \mu^{-1}(p - 1) \sum_{j = 1}^{d} \int_{0}^{t} \int_{\Omega} n_{j}^{p}W \\
+ \mu^{-1} p \sum_{j = 1}^{d} \int_{0}^{t} \int_{\Omega} nn_{j}^{p - 1} W_{j} - \frac{\alpha}{4} \min(\gamma \theta, 1) p \sum_{j = 1}^{d} \int_{0}^{t} \int_{\Omega} n_{j}^{p}.
\end{multline}
\if 1 = 0
We estimate by Calder\'{o}n-Zygmund estimates (as in Proposition \ref{nonnegativeW}) applied to
\begin{equation}\label{ellipticder}
-\mu \Delta \partial_{j}W + \partial_{j}W = a\gamma n^{\gamma - 1} \partial_{j}n
\end{equation}
that for a constant $C(n_{max, T}, a, \gamma)$:
\begin{align}\label{Wjestimate}
\|W_{j}(t, \cdot)\|_{W^{1, \infty}(\Omega)} &\le C\|W_{j}(t, \cdot)\|_{W^{2, 4}(\Omega)} \nonumber \\
&\le C(n_{max, T}, a, \gamma) \Big(\|n_{j}(t, \cdot)\|_{L^{4}(\Omega)} + \|W_{j}(t, \cdot)\|_{L^{4}(\Omega)}\Big) \nonumber \\
&\le C(n_{max, T}, a, \gamma)\|n_{j}(t, \cdot)\|_{L^{p}(\Omega)},
\end{align}
since $p > d$ is an even integer, and hence $p \ge 4$. (Note that we can estimate $\|W_{j}(t, \cdot)\|_{L^{4}(\Omega)} \le C(n_{max, T}, a, \gamma) \|n_{j}(t, \cdot)\|_{L^{p}(\Omega)}$ by testing \eqref{ellipticder} with $(\partial_{j}W)^{3}$ and proceeding as in Proposition \ref{nonnegativeW}).
\fi

Using (1) the a priori bound in \eqref{aprioridelta} to estimate $\|n_{i}\|_{C(0, T; L^{p}(\Omega))}$ in the term involving $n_{i}n_{j}^{p - 1} W_{ij}$, (2) the estimate \eqref{gradLip} in Lemma \ref{dissipativelemma}, and (3) $0 \le n \le n_{max, T}$ and hence $W \le C(n_{max, T})$ by Proposition \ref{CZest} and \ref{nonnegativeW}, we conclude:
\begin{align*}
\Bigg|p\sum_{i, j = 1}^{d} &\int_{0}^{t} \int_{\Omega} n_{i}n_{j}^{p - 1} W_{ij} + \mu^{-1}(p - 1) \sum_{j = 1}^{d} \int_{0}^{t} \int_{\Omega} n_{j}^{p}W + \mu^{-1}p \sum_{j = 1}^{d} \int_{0}^{t} \int_{\Omega} nn^{p - 1}_{j}W_{j}\Bigg| \\
&\le p \left(\sum_{i, j = 1}^{d} \int_{0}^{t} \|n_{i}\|_{L^{p}(\Omega)} \|n_{j}\|_{L^{p}(\Omega)}^{p - 1} \|\nabla W\|_{W^{1, \infty}(\Omega)}\right) \\
&\qquad \qquad + \mu^{-1}p \sum_{j = 1}^{d} \int_{0}^{t} \Big(\|W\|_{L^{\infty}(\Omega)} \|n_{j}\|_{L^{p}(\Omega)}^{p} + \|n\|_{L^{p}(\Omega)}\|n_{j}\|_{L^{p}(\Omega)}^{p - 1} \|\nabla W\|_{L^{\infty}(\Omega)}\Big) \\
&\le C(n_{max, T}, a, \gamma) p(\delta + \mu^{-1}) \sum_{j = 1}^{d} \int_{0}^{t} \int_{\Omega} n_{j}^{p}.
\end{align*}
So assuming that $\delta$ is sufficiently small and $\mu$ is sufficiently large so that the resulting constant $\displaystyle C(n_{max, T}, a, \gamma) (\delta + \mu^{-1}) \le \frac{\alpha}{8}\min(\gamma \theta, 1)$, we conclude from \eqref{apriorimid} that for every $1 \le j \le d$:
\begin{equation*}
\sum_{j = 1}^{d} \int_{\Omega} (n_{j}(t))^{p} + \frac{\alpha}{8}\min(\gamma \theta, 1) p \int_{0}^{t} \left(\sum_{j = 1}^{d} \int_{\Omega} (n_{j}(s))^{p} \right) ds \le \sum_{j = 1}^{d} \int_{\Omega} n_{0, j}^{p}.
\end{equation*}

\end{proof}

As a direct consequence of the previous estimates, we can obtain a less refined (non-dissipative) estimate on the $L^{p}$ norm of the gradient in the case where we do not require the minimum $n_{min, T} := \min_{(t, x) \in [0, T] \times \overline{\Omega}} n(t, x)$ to be greater than or equal to $\eta - \epsilon$. This less refined estimate will be useful for showing global existence for \textit{any} nonnegative initial data with small gradient and no vacuum.

\begin{proposition}[General estimate on $L^{p}$ norm of gradient]\label{W1pgeneral}
Suppose that $(n, W)$ has regularity $n \in C(0, T; W^{1, p}(\Omega))$ and $W \in C(0, T; W^{3, p}(\Omega))$ for some fixed positive even integer $p > d$ and $\min_{(t, x) \in [0, T] \times \overline{\Omega}} n(t, x) := n_{min, T} > 0$. Define $n_{max, T} := \max_{(t, x) \in [0, T] \times \overline{\Omega}} n(t, x)$. Then, if $$\left(\sum_{j = 1}^{d} \|\partial_{i}n\|^{p}_{C(0, T; L^{p}(\Omega))}\right)^{1/p} < \delta,$$ then the following inequality holds for all $t \in [0, T]$ and $1 \le j \le d$:
\begin{equation*}
\sum_{j = 1}^{d} \int_{\Omega} (n_{j}(t))^{p} dx \le \sum_{j = 1}^{d} \int_{\Omega} n_{0, j}^{p} dx + C(\delta, p, a, \alpha, \beta, \theta, \mu, \gamma, n_{max, T}) \int_{0}^{t} \left(\sum_{j = 1}^{d} \int_{\Omega} (n_{j}(s))^{p} dx\right) ds.
\end{equation*}
\end{proposition}

\begin{proof}
From the previous proof in \eqref{simplifyfirst}, we have the equality:
\begin{multline*}
\int_{\Omega} (n_{j}(t))^{p} = \int_{\Omega} n_{0, j}^{p} + p\sum_{i = 1}^{d} \int_{0}^{t} \int_{\Omega} n_{i}n_{j}^{p - 1} W_{ij} + p\int_{0}^{t} \int_{\Omega} \Big(\alpha - \beta n^{\gamma \theta} - a \gamma \mu^{-1}n^{\gamma}\Big) n_{j}^{p} \\
+ \mu^{-1}(p - 1) \int_{0}^{t} \int_{\Omega} n_{j}^{p}W + \mu^{-1}p \int_{0}^{t} \int_{\Omega} nn^{p - 1}_{j}W_{j} - a\mu^{-1}(p - 1) \int_{0}^{t} \int_{\Omega} n^{\gamma} n_{j}^{p} - \beta\gamma \theta p \int_{0}^{t} \int_{\Omega} n^{\gamma \theta} n^{p}_{j}.
\end{multline*}
From the estimate \eqref{gradp} in Lemma \ref{dissipativelemma} applied to the given $p > d$, we have the bound
\begin{equation*}
\|\nabla W(t)\|_{W^{2, p}(\Omega)} \le C(n_{max, T}, p, a, \gamma) \left(\sum_{j = 1}^{d} \|n_{j}\|^{p}_{L^{p}(\Omega)}\right)^{1/p}.
\end{equation*}
Hence, we can estimate:
\begin{equation*}
\left|\int_{0}^{t} \int_{\Omega} nn_{j}^{p - 1} W_{j}\right| \le n_{max, T} \int_{0}^{t} \|n_{j}\|_{L^{p}(\Omega)}^{p - 1} \|W_{j}\|_{L^{p}(\Omega)} \le C(n_{max, T}, p, a, \gamma) \sum_{j = 1}^{d} \int_{0}^{t} \|n_{j}\|_{L^{p}(\Omega)}^{p}.
\end{equation*}
Using the assumption that $\displaystyle \left(\sum_{j = 1}^{d} \|\partial_{j} n\|^{p}_{C(0, T; L^{p}(\Omega))}\right)^{1/p} < \delta$, we can also estimate
\begin{align*}
\left|\int_{0}^{t} \int_{\Omega} n_{i} n_{j}^{p - 1} W_{ij} \right| \le \int_{0}^{t} \|n_{i}\|_{L^{p}(\Omega)} \|n_{j}\|_{L^{p}(\Omega)}^{p - 1} \|W_{ij}\|_{L^{\infty}(\Omega)} &\le \delta \int_{0}^{t} \|n_{j}\|^{p - 1}_{L^{p}(\Omega)} \|W_{j}\|_{W^{2, p}(\Omega)} \\
&\le C(\delta, n_{max, T}, p, a, \gamma) \sum_{j = 1}^{d} \int_{0}^{t} \|n_{j}\|_{L^{p}(\Omega)}^{p}.
\end{align*}
The result therefore follows from these estimates, the fact that $\|W\|_{L^{\infty}(\Omega)} \le C\|W\|_{W^{2, p}(\Omega)} \le C(n_{max, T}, p, a, \gamma)$ by Calder\'{o}n-Zygmund estimates, and the identity \eqref{simplifyfirst}.
\end{proof}

Next, we improve the regularity of solutions as follows. While we have only shown local existence in Proposition \ref{localexistence} in $C(0, T; H^{2}(\Omega))$, we show that these solutions have higher regularity as a result of the fact that the initial data is actually smooth.

\begin{proposition}[Regularity of local-in-time strong solutions]\label{regularity}
Suppose that $(n, W)$ with regularity $n \in C(0, T; W^{1, p}(\Omega))$, $\partial_{t}n \in C(0, T; W^{1, p}(\Omega))$, and $W \in C(0, T; W^{3, p}(\Omega))$ for some positive even integer $p > d$ is a strong solution to \eqref{equations} for smooth initial data $n_{0}(x) > 0$, with $\min_{(x, t) \in [0, T] \times \overline{\Omega}} n(t, x) := n_{min, T} > 0$. Define $\max_{[0, T] \times \overline{\Omega}} n(t, x) := n_{max, T}$. Then, $n(t, \cdot)$ and $\partial_{t}n(t, \cdot)$ are spatially smooth functions for each $t \in [0, T]$ on $\overline{\Omega}$, with the estimates:
\begin{equation*}
\|D^{k}n(t)\|_{L^{p}(\Omega))}^{p} \le e^{C^{*}t}\Big(1 + \|D^{k}n_{0}\|_{L^{p}(\Omega)}^{p}\Big), \qquad \|W(t)\|_{W^{k + 2, p}(\Omega)} \le C^{*}\|n(t)\|_{W^{k, p}(\Omega)},
\end{equation*}
for all positive integers $k$ and for a constant $C^{*}$ depending only on $$C^{*} := C(p, \|n\|_{L^{\infty}(0, T; W^{k - 1, p}(\Omega))}, n_{min, T}, n_{max, T}),$$ which is increasing as a function of $\|n\|_{L^{\infty}(0, T; W^{k - 1, p}(\Omega))}$.
\end{proposition}

\begin{proof}

We show this through a priori estimates. Because we are already assuming that we have a solution on $[0, T]$, we can use bounds for lower derivatives, namely for $n$ in $C(0, T; W^{1, p}(\Omega))$, to help close the \textit{a priori estimates} for the higher derivatives.

\medskip

\noindent \textbf{Estimates on the second derivatives of the density.} We take the second derivative $\partial_{j} \partial_{k}$ of both equations and obtain:
\begin{multline}\label{njk}
\partial_{t} n_{jk} - \sum_{i = 1}^{d} \Big(n_{ijk}W_{i} + n_{ij}W_{ik} + n_{ik}W_{ij} + n_{i}W_{ijk} + n_{jk}W_{ii} + n_{j}W_{iik} + n_{k}W_{iij} + nW_{iijk}\Big) \\
= \alpha n_{jk} - \beta(\gamma \theta + 1) \gamma \theta n^{\gamma \theta - 1} n_{j}n_{k} - \beta(\gamma \theta + 1) n^{\gamma \theta} n_{jk},
\end{multline}
\begin{equation}\label{njk2}
\mu \sum_{i = 1}^{d} W_{iijk} = W_{jk} - a\gamma(\gamma - 1)n^{\gamma - 2} n_{j}n_{k} - a\gamma n^{\gamma - 1} n_{jk}.
\end{equation}
We test the first equation by $p n_{jk}^{p - 1}$ and recall that we so far have estimates for all $1 \le p < \infty$ on $n \in L^{\infty}(0, T; W^{1, p}(\Omega))$ and $W \in L^{\infty}(0, T; W^{3, p}(\Omega))$. After testing \eqref{njk} with $p n_{jk}^{p - 1}$, we obtain the following estimates:
\begin{itemize}
    \item First, we obtain after integration by parts:
    \begin{align*}
    p\Big|\sum_{i = 1}^{d} \int_{0}^{t} \int_{\Omega} \Big(n_{ijk}n_{jk}^{p - 1}W_{i} &+ n_{jk}^{p}W_{ii}\Big)\Big| = (p - 1) \left|\int_{0}^{t} \int_{\Omega} n_{jk}^{p} \Delta W\right| \\
    &\le \mu^{-1}(p - 1) \left|\int_{0}^{t} \int_{\Omega} n_{jk}^{p}W\right| + a\mu^{-1}(p - 1) \left|\int_{0}^{t} n_{jk}^{p}n^{\gamma}\right| \\
    &\le C_{p}\int_{0}^{t} \left(\int_{\Omega} |D^{2}n|^{p}\right) \Big(\|W\|_{L^{\infty}(0, T; H^{2}(\Omega))} + \|n\|_{L^{\infty}(0, T; L^{\infty}(\Omega))}^{\gamma}\Big) \\
    &\le C(p, n_{max, T}) \int_{0}^{t} \int_{\Omega} |D^{2}n|^{p}.
    \end{align*}
    \item Next, we use \eqref{njk2} to estimate
    \begin{align*}
    \Big|p\sum_{i = 1}^{d} &\int_{0}^{t} \int_{\Omega} nn_{jk}^{p - 1}W_{iijk}\Big| \\ &\le C_{p}\left(\left|\int_{0}^{t} \int_{\Omega} nn_{jk}^{p - 1}W_{jk}\right| + \left|\int_{0}^{t} \int_{\Omega} n^{\gamma - 1} n_{j}n_{k} n_{jk}^{p - 1}\right| + \left|\int_{0}^{t} \int_{\Omega} n^{\gamma}n_{jk}^{p}\right|\right) \\
    &\le C_{p}\left(\int_{0}^{t} \int_{\Omega} \Big(n^{p}|W_{jk}|^{p} + n^{(\gamma - 1)p} |\nabla n|^{2p}\Big) + \int_{0}^{t} \int_{\Omega} \Big(1 + n^{\gamma}\Big)|D^{2}n|^{p}\right) \\
    &\le C(p, n_{min, T}, n_{max, T})\left(\int_{0}^{t} \Big(\|W\|^{p}_{W^{2, p}(\Omega)} + \|\nabla n\|_{L^{\infty}(\Omega)}^{p} \|\nabla n\|_{L^{p}(\Omega)}^{p}\Big) + \int_{0}^{t} \int_{\Omega} |D^{2}n|^{p}\right) \\
    &\le C(p, \|n\|_{C(0, T; W^{1, p}(\Omega))}, n_{min, T}, n_{max, T}) \int_{0}^{t} \int_{\Omega}\Big(1 + |D^{2}n|^{p}\Big),
    \end{align*}
    using $p > d$ along with the Sobolev embedding $W^{2, p}(\Omega) \subset W^{1, \infty}(\Omega)$ to estimate $\|\nabla n\|_{L^{\infty}(\Omega)}$ and the assumption that $W \in C(0, T; W^{3, p}(\Omega)) \subset C(0, T; W^{2, \infty}(\Omega))$. Note that by Calder\'{o}n-Zygmund estimates (Proposition \ref{CZest} and \ref{nonnegativeW}), $\|W\|_{C(0, T; W^{2, p}(\Omega))} \le C_{p}\|n\|_{C(0, T; L^{p}(\Omega))}$.
    \item We estimate $\partial_{jk}(n^{\gamma})$ in $L^{p}(\Omega)$ as:
    \begin{align*}
    \|n^{\gamma - 2} n_{j}n_{k} + n^{\gamma - 1} n_{jk}\|_{L^{p}(\Omega)} &\le C(n_{min, T}, n_{max, T}) \Big(\|n_{j}\|_{L^{p}(\Omega)} \|n_{k}\|_{L^{\infty}(\Omega)} + \|n_{jk}\|_{L^{p}(\Omega)}\Big) \\
    &\le C(n_{min, T}, n_{max, T}) \Big(\|n_{j}\|_{L^{p}(\Omega)} \|n\|_{W^{2, p}(\Omega)} + \|n_{jk}\|_{L^{p}(\Omega)}\Big) \\
    &\le C(n_{min, T}, n_{max, T}, \|n\|_{C(0, T; W^{1, p}(\Omega))}) \Big(1 + \|D^{2}n\|_{L^{p}(\Omega)}\Big).
    \end{align*}
    So by Calder\'{o}n-Zygmund estimates (Proposition \ref{CZest} and \ref{nonnegativeW}),
    \begin{equation*}
    \|W\|_{W^{3, \infty}(\Omega)} \le C\|W\|_{W^{4, p}(\Omega)} \le C(p, \|n\|_{C(0, T; W^{1, p}(\Omega))}, n_{min, T}, n_{max, T}) \left(1 + \|D^{2}n\|_{L^{p}(\Omega)}\right).
    \end{equation*}
    Hence, we can estimate:
    \begin{align*}
    \Bigg|\sum_{i = 1}^{d} \int_{0}^{t} \int_{\Omega} &\Big(n_{ij}n_{jk}^{p - 1}W_{ik} + n_{ik}n_{jk}^{p - 1}W_{ij} + n_{i}n_{jk}^{p - 1}W_{ijk} + n_{j}n_{jk}^{p - 1}W_{iik} + n_{k}n_{jk}^{p - 1}W_{iij}\Big)\Bigg| \\
    &\le C_{p}\left(\int_{0}^{t} \int_{\Omega} |\nabla n|^{p} |D^{3}W|^{p} + \int_{0}^{t} \left(\int_{\Omega} |D^{2}n|^{p}\right)\Big(1 +  \|W\|_{L^{\infty}(0, T; W^{2, \infty}(\Omega))}^{p}\Big)\right) \\
    &\le C_{p} \left(\int_{0}^{t} \|\nabla n\|_{L^{\infty}(\Omega)}^{p} \|W\|_{W^{3, p}(\Omega)}^{p} + \int_{0}^{t} \left(\int_{\Omega} |D^{2}n|^{p}\right) \Big(1 + \|W\|_{L^{\infty}(0, T; W^{2, \infty}(\Omega)}^{p}\Big)\right) \\
    &\le C(p, \|n\|_{C(0, T; W^{1, p}(\Omega))}, n_{min, T}, n_{max, T})\int_{0}^{t} \int_{\Omega} \Big(1 + |D^{2}n|^{p}\Big),
    \end{align*}
    since $\|\nabla n\|_{L^{\infty}} \le C_p \|n\|_{W^{2, p}(\Omega)}$ and $W \in C(0, T; W^{3, p}(\Omega)) \subset C(0, T; W^{2, \infty}(\Omega))$ with $$\|W\|_{C(0, T; W^{3, p}(\Omega))} \le C(p, n_{max, T})\|n\|_{C(0, T; W^{1, p}(\Omega))}.$$
    \item Finally, we estimate
    \begin{align*}
    \left|\int_{0}^{t} \int_{\Omega} n^{\gamma \theta - 1} n_{j}n_{k} n_{jk}^{p - 1}\right| &\le C(n_{min, T}, n_{max, T}) \int_{0}^{t} \|n_{k}\|_{L^{\infty}(\Omega)} \|n_{j}\|_{L^{p}(\Omega)} \|n_{jk}\|_{L^{p}(\Omega)}^{p - 1} dt \\
    &\le C(\|n\|_{C(0, T; W^{1, p}(\Omega))}, n_{min, T}, n_{max, T}) \int_{0}^{t} \int_{\Omega} \Big(1 + |D^{2}n|^{p}\Big),
    \end{align*}
    by the estimate $\|n_{k}\|_{L^{\infty}(\Omega)} \le C\|n\|_{W^{2, p}(\Omega)}$. We also have the estimate:
    \begin{equation*}
    \beta p (\gamma \theta + 1) \left|\int_{0}^{t} \int_{\Omega} n^{\gamma \theta} n_{jk}^{p}\right| \le C(n_{min, T}, n_{max, T}) \int_{0}^{t} \int_{\Omega} |D^{2}n|^{p}.
    \end{equation*}
\end{itemize}

\noindent \textit{Conclusion.} Combining all of the above estimates, we obtain:
\begin{equation*}
\int_{\Omega} |D^{2}n(t)|^{p} \le \int_{\Omega} |D^{2}n_{0}|^{p} + C(p, \|n\|_{C(0, T; W^{1, p}(\Omega))}, n_{min, T}, n_{max, T}) \int_{0}^{t} \int_{\Omega} \Big(1 + |D^{2}n(s)|^{p}\Big) ds,
\end{equation*}
so by Gronwall's inequality and Calder\'{o}n-Zygmund estimates, we obtain:
\begin{equation*}
\|D^{2}n(t)\|_{L^{p}(\Omega)}^{p} \le e^{C^{*}t}\Big(1 + \|D^{2}n_{0}\|_{L^{p}(\Omega)}^{p}\Big), \qquad \|W(t)\|_{W^{4, p}(\Omega)} \le C^{*}\|n\|_{W^{2, p}(\Omega)},
\end{equation*} 
for some constant $C^{*}$ depending on the parameters $C^{*} := C(p, \|n\|_{C(0, T; W^{1, p}(\Omega))}, n_{min, T}, n_{max, T})$.

\medskip

\noindent \textbf{Estimates on all higher derivatives.} We obtain estimates on higher order spatial derivatives by using an inductive argument, with an inductive assumption that for some positive integer $k$, we have established the following bound:
\begin{equation}\label{inductivek}
\|n\|_{L^{\infty}(0, T; W^{k, p}(\Omega))} \le C, \quad \text{ and hence, by Calder\'{o}n-Zygmund}, \quad \|W\|_{L^{\infty}(0, T; W^{k + 2, p}(\Omega))} \le C.
\end{equation}
We then want to show that under this inductive assumption, we can obtain the estimates:
\begin{equation}\label{inductivestep}
\|D^{k + 1}n\|_{L^{p}(\Omega)}^{p} \le e^{C^{*}t}\Big(1 + \|D^{k + 1}n_{0}\|_{L^{p}(\Omega)}^{p}\Big), \qquad \|W(t)\|_{W^{k + 3, p}(\Omega)} \le C^{*}\|n(t)\|_{W^{k + 1, p}(\Omega)},
\end{equation}
for some constant $C^{*}$, where for simplicity of notation, we will use this shorthand to indicate the following explicit dependence of constants on parameters for the remainder of the proof:
\begin{equation}\label{Cstardepend}
C^{*} = C(p, \|n\|_{L^{\infty}(0, T; W^{k, p}(\Omega))}, n_{min, T}, n_{max, T}).
\end{equation}

To do this, we consider a general multi-index $\partial_{\alpha}$ with $|\alpha| = k + 1$ and we take the derivative with respect to this multi-index of both equations. We obtain:
\begin{equation}\label{alphaderivative}
\partial_{t}\partial_{\alpha} n - \sum_{i = 1}^{d} \Big((\partial_{\alpha} \partial_{i} n) \partial_{i}W + (\partial_{\alpha} n) \partial_{i}^{2} W + n\partial_{\alpha} \partial^{2}_{i}W\Big) + R_{\alpha} = \alpha \partial_{\alpha}n - \beta(\gamma \theta + 1)n^{\gamma \theta} \partial_{\alpha} n + T_{\alpha},
\end{equation}
where:
\begin{itemize}
    \item $R_{\alpha}$ is comprised of terms $(\partial_{\beta} n)(\partial_{\gamma} W)$ for multi-indices with $1 \le |\beta| \le k + 1 = |\alpha|$ and $|\gamma| = k + 3 - |\beta|$.
    \item $T_{\alpha}$ is comprised of terms $n^{\gamma \theta - m} (\partial_{\alpha_{1}}n) (\partial_{\alpha_{2}}n) \cdot \cdot \cdot (\partial_{\alpha_{m+1}}n)$, for a positive integer $1 \le m \le k$ and multi-indices $1 \le |\alpha_{1}| \le |\alpha_{2}| \le ... \le |\alpha_{m+1}| \le k - m + 1$ with $\displaystyle \sum_{i = 1}^{m+1} |\alpha_{i}| = k + 1$.
\end{itemize}
We also have the equation for the potential:
\begin{equation}\label{Winductive}
\mu \sum_{i = 1}^{d} \partial_{\alpha} \partial^{2}_{i}W = \partial_{\alpha} W - a \partial_{\alpha} (n^{\gamma}).
\end{equation}
We test \eqref{alphaderivative} with $p(\partial_{\alpha}n)^{p - 1}$ and integrate over space, and estimate the resulting terms.
\begin{itemize}
\item By integrating by parts and using the Neumann boundary condition on $W$:
\begin{equation*}
p\sum_{i = 1}^{d} \int_{0}^{t} \int_{\Omega} \Big((\partial_{\alpha} \partial_{i}n) \partial_{i}W + (\partial_{\alpha} n)\partial^{2}_{i}W\Big) (\partial_{\alpha} n)^{p - 1} = (p - 1)\int_{0}^{t} \int_{\Omega} (\partial_{\alpha} n)^{p} \Delta W. 
\end{equation*}
Since $k + 2 \ge 4$ in the inductive assumption, we can estimate that $\|\Delta W\|_{L^{\infty}(\Omega)} \le \|W\|_{W^{4, p}(\Omega)}$, and hence by the boundedness of $W$ in $L^{\infty}(0, T; W^{k + 2, p}(\Omega))$:
\begin{equation*}
\left|p\sum_{i = 1}^{d} \int_{0}^{t} \int_{\Omega} \Big((\partial_{\alpha} \partial_{i}n) \partial_{i}W + (\partial_{\alpha}n)\partial^{2}_{i}W\Big) (\partial_{\alpha}n)^{p - 1}\right| \le C^{*}\int_{0}^{t} \int_{\Omega} (\partial_{\alpha}n)^{p},
\end{equation*}
where we recall the shorthand $C^{*}$ from \eqref{Cstardepend}.
\item Next, using the equation for the potential \eqref{Winductive}, we estimate that
\begin{equation}\label{midest}
\left|p\sum_{i = 1}^{d} \int_{0}^{t} \int_{\Omega} n(\partial_{\alpha}n)^{p - 1} \partial_{\alpha}\partial^{2}_{i}W\right| \le C_{p}\left(\left|\int_{0}^{t} \int_{\Omega} n(\partial_{\alpha}n)^{p - 1} \partial_{\alpha} W\right| + \left|\int_{0}^{t} \int_{\Omega} n(\partial_{\alpha}n)^{p - 1} \partial_{\alpha}(n^{\gamma})\right|\right).
\end{equation}
For the first term on the right-hand side of \eqref{midest}, using the inductive assumption on $W$ in \eqref{inductivek}:
\begin{multline*}
\left|\int_{0}^{t} \int_{\Omega} n(\partial_{\alpha} n)^{p - 1} \partial_{\alpha} W\right| \le C_{p}\left(\int_{0}^{t} \int_{\Omega} n^{p} |\partial_{\alpha} W|^{p} + \int_{0}^{t} \int_{\Omega} |\partial_{\alpha}n|^{p}\right) \\
\le C_{p}\left(\int_{0}^{t}\|n\|_{L^{\infty}(0, T; L^{\infty}(\Omega))}^{p} \|W\|_{L^{\infty}(0, T; W^{k + 1, p}(\Omega))}^{p} + \int_{0}^{t} \int_{\Omega} |\partial_{\alpha} n|^{p}\right) \le C^{*}\int_{0}^{t} \int_{\Omega} \Big(1 + |\partial_{\alpha} n|^{p}\Big),
\end{multline*}
where we use the shorthand \eqref{Cstardepend} to indicate the dependence of constants on parameters. For the second term on the right-hand side of \eqref{midest}, we note that because $n \in L^{\infty}(0, T; W^{k, p}(\Omega))$ for some $k \ge 2$ (from the inductive assumption \eqref{inductivek}) and for $p > d$, we have that there are $L^{\infty}(\Omega)$ bounds on the first derivative of $n$ and all higher order derivatives up to order $k - 1$. So therefore,
\begin{equation*}
\left|\int_{0}^{t} \int_{\Omega} n(\partial_{\alpha}n)^{p - 1} \partial_{\alpha}(n^{\gamma})\right| \le \int_{0}^{t} \int_{\Omega} \gamma n^{\gamma} |\partial_{\alpha}n|^{p} + \text{l.o.t.} \le C^{*}\int_{0}^{t} \int_{\Omega} \Big(1 + |D^{k + 1}n|^{p}\Big),
\end{equation*}
where the lower order terms are terms of the form $\displaystyle \int_{0}^{t} \int_{\Omega} n^{\gamma - m} |\partial_{\alpha}n|^{p - 1} |\partial_{\alpha_{1}}n \partial_{\alpha_{2}}n \cdot \cdot \cdot \partial_{\alpha_{m+1}}n|$ where $1 \le m \le k$ and $1 \le |\alpha_{1}| \le |\alpha_{2}| \le ... \le |\alpha_{m+1}| \le k - m + 1$ with $|\alpha_{1}| + |\alpha_{2}| + ... + |\alpha_{m+1}| = k + 1$. The bound on the first term $\displaystyle \int_{0}^{t} \int_{\Omega} \gamma n^{\gamma} |\partial_{\alpha}n|^{p}$ follows from the previous bounds on $\|n\|_{L^{\infty}(0, T; L^{\infty}(\Omega))}$. We can then estimate the lower order terms by estimating that 
\begin{align}\label{lotest}
\int_{0}^{t} \int_{\Omega} n^{\gamma - m} |\partial_{\alpha}n|^{p - 1} |\partial_{\alpha_{1}}n \cdot \cdot \cdot \partial_{\alpha_{m+1}}n| &\le C_{p}\left(\int_{0}^{t} \int_{\Omega} |\partial_{\alpha}n|^{p} + \int_{0}^{t} \int_{\Omega} |n^{p(\gamma - m)}| \cdot  |\partial_{\alpha_{1}}n \cdot \cdot \cdot \partial_{\alpha_{m+1}}n|^{p}\right) \nonumber \\
&\le C^{*}\int_{0}^{t} \int_{\Omega} \Big(1 + |D^{k+1}n|^{p}\Big),
\end{align}
where the estimate subsuming the second integral into $C^{*}$, see the definition of $C^{*}$ in \eqref{Cstardepend}, follows from the inductive assumption that $n \in L^{\infty}(0, T; W^{k, p}(\Omega))$ for some $k \ge 2$ and $p > d$. We also observe that the second integral can only contain at most one $|\alpha_{i}|$ equal to $k$ (which can be estimated in $L^{p}(\Omega)$ by the inductive assumption \eqref{inductivek}), in which case the rest of the $|\alpha_{i}|$ are strictly less than $k$ and hence the $\alpha_{i}$-derivatives can be bounded in $L^{\infty}$ using the inductive bound on $n \in L^{\infty}(0, T; W^{k, p}(\Omega))$.

\item Next, we estimate the contribution of $\displaystyle \int_{0}^{t} \int_{\Omega} R_{\alpha} (\partial_{\alpha}n)^{p - 1}$, where we recall that $R_{\alpha}$ is made up of terms of the form $(\partial_{\beta} n)(\partial_{\gamma} W)$ where $1 \le |\beta| \le k + 1$ and $|\gamma| = k + 3 - |\beta|$. If $|\beta| = k + 1$ so that $|\gamma| = 2$, we can estimate:
\begin{equation*}
\left|\int_{0}^{t} \int_{\Omega} \partial_{\beta} n (\partial_{\alpha}n)^{p - 1} \partial_{\gamma} W\right| \le \|W\|_{L^{\infty}(0, T; W^{2, \infty}(\Omega))} \int_{0}^{t} \int_{\Omega} |D^{k + 1}n|^{p} \le C^{*} \int_{0}^{t} \int_{\Omega} |D^{k + 1}n|^{p},
\end{equation*}
by Sobolev embedding and the inductive assumption $W \in L^{\infty}(0, T; W^{k + 2, p}(\Omega))$ for some $k \ge 2$ and $p > d$ in \eqref{inductivek}. In the case where $|\beta| = 1$, then $|\gamma| = k + 2$, and we can use the inductive bound on $\|n\|_{L^{\infty}(0, T; W^{k, p}(\Omega))}$ and $\|W\|_{L^{\infty}(0, T; W^{k + 2, p}(\Omega))}$ for $k \ge 2$, in \eqref{inductivek} to estimate:
\begin{align*}
\left|\int_{0}^{t} \int_{\Omega} \partial_{\beta} n (\partial_{\alpha} n)^{p - 1}\partial_{\gamma} W\right| &\le C_{p}\left(\int_{0}^{t} \int_{\Omega} |\nabla n|^{p} |D^{k + 2}W|^{p} + \int_{0}^{t} \int_{\Omega} |D^{k + 1}n|^{p}\right) \\
&\le C^{*} \int_{0}^{t} \int_{\Omega} \Big(1 + |D^{k + 1}n|^{p}\Big). 
\end{align*}
Finally, if $2 \le |\beta| \le k$ and hence $3 \le |\gamma| \le k + 1$ with $|\gamma| = k + 3 - |\beta|$, we estimate:
\begin{align*}
\left|\int_{0}^{t} \int_{\Omega} (\partial_{\beta} n)(\partial_{\alpha}n)^{p - 1} \partial_{\gamma} W\right| &\le C_p\left(\int_{0}^{t} \int_{\Omega} |\partial_{\beta}n|^{p} |\partial_{\gamma} W|^{p} + \int_{0}^{t} \int_{\Omega} |D^{k + 1}n|^{p}\right) \\
&\le C^{*}\int_{0}^{t} \int_{\Omega} \Big(1 + |D^{k + 1}n|^{p}\Big),
\end{align*}
since by the inductive assumption \eqref{inductivek}, $n \in L^{\infty}(0, T; W^{k, p}(\Omega))$ and $W \in L^{\infty}(0, T; W^{k + 2, p}(\Omega))$ for all $p > d$ and for some $k \ge 2$. So we conclude that the total contribution is:
\begin{equation*}
\left|\int_{0}^{t} \int_{\Omega} R_{\alpha} (\partial_{\alpha} n)^{p - 1}\right| \le C^{*} \int_{0}^{t} \int_{\Omega} \Big(1 + |D^{k + 1}n|^{p}\Big).
\end{equation*}
\item Next, we estimate $\displaystyle \beta (\gamma \theta + 1) p \left|\int_{0}^{t}  \int_{\Omega} n^{\gamma \theta} (\partial_{\alpha}n)^{p}\right| \le C^{*} \int_{0}^{t} \int_{\Omega} |D^{k + 1}n|^{p}$.
\item Finally, we estimate $\displaystyle \left|\int_{0}^{t} \int_{\Omega} T_{\alpha} (\partial_{\alpha}n)^{p - 1}\right|$, where we recall that $T_{\alpha}$ is comprised of terms of the form $n^{\gamma \theta - m}(\partial_{\alpha_{1}}n)(\partial_{\alpha_{2}}n) \cdot \cdot \cdot (\partial_{\alpha_{m + 1}}n)$ for $1 \le m \le k$ and $1 \le |\alpha_{1}| \le |\alpha_{2}| \le ... \le |\alpha_{m + 1}| \le k - m + 1$ with $\displaystyle \sum_{i = 1}^{m + 1} |\alpha_{i}| = k + 1$. As in \eqref{lotest}, we estimate for $C^{*}$ in \eqref{Cstardepend}:
\begin{equation*}
\left|\int_{0}^{t} \int_{\Omega} T_{\alpha}(\partial_{\alpha}n)^{p - 1}\right| \le C^{*}\int_{0}^{t} \int_{\Omega} \Big(1 + |D^{k + 1}n|^{p}\Big).
\end{equation*}
\end{itemize}

So we obtain the final estimate $\displaystyle 
\int_{\Omega} |D^{k + 1}n(t)|^{p} \le \int_{\Omega} |D^{k + 1}n_{0}|^{p} + C^{*} \int_{0}^{t} \int_{\Omega} \Big(1 + |D^{k + 1}n(s)|^{p}\Big)$, and hence conclude the inductive step \eqref{inductivestep} for $\|n\|_{L^{\infty}(0, T; W^{k + 1, p}(\Omega))}$ by Gronwall's inequality and for $\|W\|_{L^{\infty}(0, T; W^{k + 3, p}(\Omega))}$ by Calder\'{o}n-Zygmund estimates, obtaining that for some constant $C^{*}$ depending on parameters as in \eqref{Cstardepend}:
\begin{equation*}
\|D^{k + 1}n(t)\|_{L^{p}(\Omega)}^{p} \le e^{C^{*}t}\Big(1 + \|D^{k + 1} n_{0}\|_{L^{p}(\Omega)}^{p}\Big), \qquad \|W(t)\|_{W^{k + 3, p}(\Omega)} \le C^{*}\|n(t)\|_{W^{k + 1, p}(\Omega)}.
\end{equation*}

\end{proof}

\subsection{Proof of global existence theorems}\label{globalproofs}

Using the a priori global estimates in Proposition \ref{W1pcontrol}, Proposition \ref{W1pgeneral}, and Proposition \ref{regularity}, we can prove the small gradient result in Theorem \ref{global}. 

\begin{proof}[Proof of Theorem \ref{global}]
Consider initial data $n_{0} \in C^{\infty}(\overline{\Omega})$ such that $m := \min_{x \in \overline{\Omega}} n_{0}(x) > 0$ and $M := \min_{x \in \overline{\Omega}} n_{0}(x) > 0$. Define the constant $C(M, a, \gamma)$ in Theorem \ref{global} to be the constant $C(n_{max, T}, a, \gamma)$ from Proposition \ref{W1pcontrol} for a constant 
\begin{equation*}
n_{max, T} := \max(M, n^{*}),
\end{equation*}
where we recall the definition of $n^{*}$ from Theorem \ref{global} and the desired upper bound $n(t, \cdot) \le n_{max, T}$ for $n_{max, T} := \max(M, n^{*})$, from Proposition \ref{maxprinciple}. Choose $\epsilon > 0$ so that 
\begin{equation*}
\alpha - \beta n^{\gamma \theta} - a \mu^{-1} n^{\gamma} \le \frac{\alpha}{4} \min(\gamma \theta, 1) \quad \text{ for all } n \ge \eta - \epsilon
\end{equation*}
for $\eta > 0$ defined in \eqref{nlowerstar}, and consider an associated $\delta > 0$ so that 
\begin{equation}\label{deltachoice}
C(n_{max, T}, a, \gamma)(\delta + \mu^{-1}) \le \frac{\alpha}{8}\min(\gamma \theta, 1),
\end{equation}
which is possible by the assumption in Theorem \ref{global} that $\displaystyle C(n_{max, T}, a, \gamma) \mu^{-1} < \frac{\alpha}{8}\min(\gamma \theta, 1)$. Then, we first consider initial data $n_{0}$ such that 
\begin{equation*}
n_{0} \ge \eta - \epsilon, \qquad \left(\sum_{j = 1}^{d} \|\partial_{j}n_{0}\|^{6}_{L^{6}(\Omega)}\right)^{1/6} < \delta.
\end{equation*}
Let $S$ be the set of times $T \ge 0$ such that there exists a smooth solution on $[0, T]$ and $\max_{0 \le t \le T} \left(\sum_{j = 1}^{d} \|\partial_{j}n(t)\|^{6}_{L^{6}(\Omega)}\right)^{1/6} < \delta$. By Proposition \ref{localexistence}, there exists a local solution with $n \in C(0, T_{loc}; H^{2}(\Omega))$ which by Proposition \ref{regularity} is smooth and which satisfies the desired lower bound $n(t, \cdot) \ge \eta - \epsilon$ on $[0, T_{loc}]$ by Proposition \ref{minprinciple}. By continuity, since $\left(\sum_{j = 1}^{d} \|\partial_{j}n_{0}\|_{L^{6}(\Omega)}^{6}\right)^{1/6} < \delta$, and Sobolev embedding, we conclude that there exists some positive time $T_{0} \in S$. This same argument shows that $S$ is open.

To show that $S$ is closed, suppose that $T_{k} \in S$ and $T_{k} \nearrow T$. We claim $T \in S$. By the a priori estimate in Proposition \ref{W1pcontrol}, we have that $\left(\sum_{j = 1}^{d} \|\partial_{j}n(t)\|_{L^{6}(\Omega)}^{6}\right)^{1/6} < \delta$ for all $t \in [0, T_{k}]$ for each $k$. By the uniform bound on $\eta - \epsilon \le n(t, \cdot) \le n_{max, T}$ for each $t \in [0, T_{k}]$ for each $k$, as a result of Propositions \ref{maxprinciple} and \ref{minprinciple} and the fact that the solution is smooth on $[0, T_{k}]$ by the estimates in Proposition \ref{regularity}, we have that
\begin{equation*}
\|n\|_{C(0, T_{k}; W^{1, 6}(\Omega))} \le C_{\delta},
\end{equation*}
for a constant $C_{\delta}$ that is \textit{independent of $k$}. So by Proposition \ref{regularity}, there exists a constant $C_{\delta}$ \textit{independent of $k$}, such that
\begin{equation*}
\|n\|_{C(0, T_{k}; H^{3}(\Omega))} \le C_{\delta}.
\end{equation*}
Since the local time of existence in Proposition \ref{localexistence} depends only on $\|n_{0}\|_{H^{3}(\Omega)}$ and on the lower and upper bound on $n_{0}$, there exists a uniform time $\tau > 0$ of existence for any initial data bounded below by $\eta - \epsilon$, bounded above by $n_{max, T}$, and with $H^{3}(\Omega)$ norm less than or equal to $C_{\delta}$. Therefore, we can extend the smooth solution from any $T_{k}$ by $\tau > 0$ so that $T \in S$ also. Hence, $S$ is nonempty, open, and closed, and must be $S = [0, \infty)$, which proves global existence for $\left(\sum_{j = 1}^{d} \|\partial_{j} n_{0}\|^{6}_{L^{6}(\Omega)}\right)^{1/6} < \delta$ and $n_{0} \ge \eta - \epsilon$.

In the case where $m := \min_{x \in \overline{\Omega}} n_{0}(x)$ satisfies $0 < m < \eta - \epsilon$, we are motivated by Proposition \ref{minprinciple} to let $\tau_{m}$ be the time for which the ODE
\begin{equation*}
\frac{d\overline{n}}{dt} = \alpha \overline{n} - \beta \overline{n}^{1 + \gamma \theta} - a\mu^{-1} \overline{n}^{1 + \gamma}, \qquad \overline{n}_{0} = m,
\end{equation*}
has 
\begin{equation}\label{taum}
\overline{n}(\tau_{m}) = \eta - \epsilon.
\end{equation}
(Note that the unique solution to this ODE for initial condition $m \in (0, \eta)$ is strictly increasing in time with $\lim_{t \to \infty} \overline{n}(t) = \eta$, which shows that such a $\tau_{m}$ exists). Choose $\delta_{0} \le \delta/2$ sufficiently small so that
\begin{equation}\label{deltanumerology}
\delta_{0}^{6}e^{C^{*}\tau_{m}} \le (\delta/2)^{6},
\end{equation}
where $\delta$ was the $\delta$ from \eqref{deltachoice}, which gives the global existence result for initial data bounded below by $\eta - \epsilon$ with an $L^{6}(\Omega)$ norm of the gradient strictly bounded above by $\delta$. Here, $C^{*}$ is a shorthand notation for the constant $C^{*} := C(\delta, p, a, \alpha, \beta, \theta, \mu, \gamma, n_{max, T})$ from Proposition \ref{W1pgeneral} for $p = 6$. Then, given $\|\nabla n_{0}\|_{L^{6}(\Omega)} < \delta_{0}$, we can use the more general a priori estimate in Proposition \ref{W1pgeneral} similarly to conclude that a smooth solution exists on $[0, \tau_{m}]$ and for the final time $\tau_{m}$, we have that $\left(\sum_{j = 1}^{d} \|\partial_{j}n(\tau_{m})\|_{L^{6}(\Omega)}^{6}\right)^{1/6} \le \delta/2 < \delta$ via the a priori integral inequality (Proposition \ref{W1pgeneral}) and \eqref{deltanumerology}. Then, we can construct a global solution by using $n(\tau_{m}, \cdot)$ as initial data and applying the previous result for initial data bounded below by $\eta - \epsilon$, since by the choice of $\tau_{m}$ and Proposition \ref{minprinciple}, we have that $n(\tau_{m}, \cdot) \ge \eta - \epsilon$ and $\left(\sum_{j = 1}^{d} \|\partial_{j} n(\tau_{m})\|_{L^{6}(\Omega)}^{6}\right)^{1/6} \le \delta/2$. This allows us to construct a solution on $[\tau_{m}, \infty)$ and combine the two resulting solutions together. 
\end{proof}

\if 1 = 0
\section{Preliminary long-time behavior results}

First, we show that the cell density must asymptotically be bounded above by $n^{*}$, which is the unique positive solution to $\alpha n - \beta n^{1 + \gamma \theta} = 0$, even if the initial data $n_{0}$ has regions that are greater than $n^{*}$. To do this, we let
\begin{equation*}
M := \max_{x \in \overline{\Omega}} n(x),
\end{equation*}
and we claim that for all $\nu > n^{*}$, 
\begin{equation*}
\lim_{t \to \infty} \Big(1 - F_{t}(\nu)\Big) = \lim_{t \to \infty} \text{meas}\{n(t, \cdot) > \nu\} = 0.
\end{equation*}

To show this, we consider the following set $S$:
\begin{equation*}
S = \{\nu \ge 0 : \text{there exists $T > 0$ such that $n(t) \le \nu$ for all $t \ge T$}\}.
\end{equation*}
Note that $S$ is non-empty, since $[M, \infty) \subset S$ by the maximum principle, Proposition TODO. We want to show more specifically that $(n^{*}, \infty) \subset S$.

Note in particular that $S$ is a connected set: if $\nu \in S$, then $[\nu, \infty) \in S$ also. Therefore, to show that $(n^{*}, \infty) \subset S$, we argue by contradiction, in which case we can consider:
\begin{equation*}
\lambda = \inf S > n^{*}.
\end{equation*}
We derive a contradiction by considering the equation for $(n - \lambda)^{+}$. First, note that
\begin{equation*}
\partial_{t}(n - \lambda) - \text{div}((n - \lambda) \nabla W ) = \alpha n - \beta n^{1 + \gamma \theta} + \lambda \Delta W,
\end{equation*}
and hence, by using the second equation for the potential:
\begin{equation*}
\partial_{t}(n - \lambda) - \text{div}((n - \lambda) \nabla W) = \alpha n - \beta n^{1 + \gamma \theta} + \mu^{-1} \lambda \Big(W - a n^{\gamma}\Big).
\end{equation*}
Therefore, multiplying by $\text{sgn}(n - \lambda)$, we obtain:
\begin{equation*}
\partial_{t}(n - \lambda)^{+} - \text{div}((n - \lambda)^{+} \nabla W) = \text{sgn}(n - \lambda)\Big(\alpha n - \beta n^{1 + \gamma \theta} + \mu^{-1}\lambda(W - an^{\gamma})\Big),
\end{equation*}
and hence by integrating and using the Neumann boundary condition for $W$, we obtain that for any $0 < s < t$:
\begin{equation}\label{ddtabovelambda}
\frac{d}{dt} \int_{\Omega} (n - \lambda)^{+}(t) = \int_{\Omega} \text{sgn}(n - \lambda) \Big(\alpha n - \beta n^{1 + \gamma \theta} + \mu^{-1}\lambda(W - an^{\gamma})\Big).
\end{equation}
We assumed that $\lambda > n^{*}$ and we will now obtain a contradiction as follows. Since $S$ is a connected set, $(\lambda, \infty) \subset S$. Furthermore,
\begin{equation}\label{sgnineq}
\int_{\Omega} \text{sgn}(n - \lambda) \Big(\alpha n - \beta n^{1 + \gamma \theta} + \mu^{-1}\lambda (W - an^{\gamma})\Big) \le\int_{\Omega} \Big(\alpha\lambda - \beta \lambda^{1 + \gamma \theta} + \mu^{-1}\lambda(W - a\lambda^{\gamma})\Big).
\end{equation}
Since $\lambda > n^{*}$, note that 
\begin{equation*}
\alpha \lambda - \beta \lambda^{1 + \gamma \theta} < 0.
\end{equation*}
Therefore, we can choose $\lambda_{0} > \lambda$ sufficiently close to $\lambda$ such that:
\begin{equation}\label{lambda0}
a\lambda - \beta\lambda^{1 + \gamma \theta} + \mu^{-1}\lambda a(\lambda_{0}^{\gamma} - \lambda^{\gamma}) := -c_0 < 0.
\end{equation}
We then obtain a contradiction as follows. Since $\lambda_0 > \lambda$, we have that $\lambda_0 \in S$, and hence, we choose $T$ such that $n(t, x) \le \lambda_0$ for all $t \ge T$. Then, by uniqueness of the solution (see Proposition TODO), $n(t, x)$ for $t \ge T$ can be obtained by solving the given equations with initial data $n(T, x)$, which is bounded from above by $\lambda_0$. So by applying the bound on the potential $W$ in Proposition TODO, we obtain:
\begin{equation*}
0 \le W(t, x) \le a\lambda_0^{\gamma}, \qquad \text{ for all } t \ge T.
\end{equation*}
Then, using this bound on the potential along with \eqref{ddtabovelambda}, \eqref{sgnineq}, and \eqref{lambda0}, we obtain:
\begin{equation*}
\frac{d}{dt}\int_{\Omega}(n - \lambda)^{+} (t) \le -c_0 |\Omega| < 0, \qquad \text{ for all $t \ge T$.}
\end{equation*}
Therefore, there will exist a finite time $T_{\lambda}$ for which 

\fi

\if 1 = 0
\section{Long-time behavior in the absence of vacuum}\label{novacuumlongtime}

Next, we establish some global estimates and long-time behavior. We show that solutions to \eqref{equations} converge to a (constant) steady state $n_{\infty} = n^{*}, W_{\infty} = a(n^{*})^{\gamma}$,
as $t \to \infty$. We first establish this in the case of no vacuum and solutions bounded above by $n^{*}$ his result is easiest to achieve first in the absence of vacuum, which will give intuition for the general proof of Theorem \ref{longtime}. Namely, if the initial data $n_{0}(x) > 0$, then the global smooth solution $(n, W)$ in Proposition \ref{global} remains strictly positive for all time by Proposition \ref{minprinciple}.

\if 1 = 0
To do this, we require global estimates on the solution $(n, W)$. In particular, in the previous section where we obtained global existence of solutions, we only required estimates of finiteness of norms on intervals $[0, T]$, but the size of these norms could very well depend on the final time $T > 0$, since we used Gronwall's inequality to estimate these norms and hence we obtain estimates that depend on the final time $T > 0$. So the goal of this section is to obtain estimates that are \textit{independent of $T$}, as these global bounds will allow us to make conclusions about the long-time behavior of solutions.

\medskip

\noindent \textbf{A global lower bound on the density.} We already have a uniform global bound on the density from above, given in Proposition TODO. However, one of the essential issues in studying the long-time behavior of such systems, as is well known for compressible flow dynamics in general, is the problem of potential appearance and uniform control of vacuum, $n = 0$ in this case. Hence, our first essential and crucial estimate is an observation about a uniform bound from below, which follows from a minimum principle argument and from the special structure of the system of partial differential equations. The result is as follows.

\medskip

In the absence of vacuum, we start with smooth initial data $n_{0} \ge \nu > 0$, and from the previous proposition, we know that the resulting solution remains bounded from below strictly away from zero. This gives rise easily to the following long-term behavior result.

\fi

\begin{theorem}
    Let $0 < \nu \le n_{0} \le n^{*}$ be a smooth function on $\overline{\Omega}$. Then, the solution to \eqref{equations} with initial data $n(0, \cdot) = n_{0}$ satisfies:
    \begin{equation*}
        0 \le \int_{\Omega} \Big(n^{*} - n(t, x)\Big) \le ce^{-Ct},
    \end{equation*}
    for constants $c$ and $C$ depending only on the initial data through the lower bound $\nu > 0$.
\end{theorem}

\begin{proof}
Note that $n^{*} - n(t, x) \ge 0$ and furthermore, 
\begin{equation}\label{nstarminusn}
\partial_{t}(n^{*} - n(t, x)) - \text{div}((n^{*} - n)\nabla W) - \text{div}(n^{*}\nabla(W^{*} - W)) = \alpha (n^{*} - n) - \beta \Big((n^{*})^{1 + \gamma \theta} - n^{1 + \gamma \theta}\Big),
\end{equation}
since $\alpha n^{*} - \beta (n^{*})^{1 + \gamma \theta} = 0$ by definition. An analysis of $f(z) = \alpha z - \beta z^{1 + \gamma \theta}$ shows that
\begin{equation*}
-C_{\delta} (n^{*} - n) \le f(n^{*}) - f(n) \le 0, \qquad \text{ for all } z \in [\delta, n^{*}].
\end{equation*}
So since $n(t, x) \ge n_{*} := \min(\nu, \eta)$ (as defined in Proposition \ref{minprinciple}) for all $t$, we can take $\delta = n_{*} > 0$ and hence conclude after integrating \eqref{nstarminusn}, that for some constant $C$:
\begin{equation*}
0 \le \int_{\Omega} (n^{*} - n)(t) \le \int_{\Omega} (n^{*} - n)(0) - C\int_{0}^{t} \int_{\Omega} (n^{*} - n)(s),
\end{equation*}
for a constant $C$ independent of $t$ (depending only on the initial data through $\nu$).
\end{proof}

\fi 

\section{Long-time behavior of global strong solutions}\label{vacuumlongtime}

Now that we have shown that nontrivial global strong solutions can exist in Theorem \ref{global}, we study the long-time behavior of global strong solutions to the system \eqref{equations} with general smooth initial data $n_{0} \ge 0$. In this case, we cannot obtain lower bounds on the density away from zero (see Proposition \ref{minprinciple}) and we must therefore use a more refined analysis to study the long-time behavior of the system. The main idea will be to do a ``level-set analysis", which involves analyzing the equation through a geometric perspective. By Proposition \ref{maxprinciple}, we must have $0 \le n(t, x) \le \max(n^{*}, M)$, for all $t \ge 0$, 
where $\displaystyle M := \text{max}_{x \in \overline{\Omega}} n_{0}(x)$. Hence, consider the level set $\{n_{0}(x) = \xi\}$ for some $0 \le \xi \le \max(n^{*}, M)$. Rewriting the equations using $\bd{u} := -\nabla W$, we obtain:
\begin{equation}\label{rewritten}
\begin{cases}
\partial_{t} n + \bd{u} \cdot \nabla n = \alpha n - \beta n^{1 + \gamma \theta} - a\mu^{-1}n^{1 + \gamma} + \mu^{-1}nW, \\
\text{div}(\bd{u}) = a\mu^{-1} n^{\gamma} - \mu^{-1}W.
\end{cases}
\end{equation}
Rewriting these equations in this form, we note that there are \textit{three contributions} to the way in which this level set $\{n_{0} = \xi\}$ for the initial data evolves over time:

\begin{itemize}
\item \textbf{Transport of the cell density $\bd{n}$ via the flow $\bd{u}$.} The left-hand side of the first equation in \eqref{rewritten} shows the transport of $n(t, x)$ via the velocity field $\bd{u}(t, x)$. Analyzing this transport phenomena involves understanding the flow map associated to the time-dependent vector field $\bd{u} := -\nabla W$, which we will discuss in Section \ref{transport}.
\item \textbf{Accretion and depletion of cells.} From the first equation, we see from the right-hand side that as the cells are transported by the fluid velocity $\bd{u}$, we either have accretion (increase) or depletion (decrease) of cell density along the flow lines of the velocity depending on the sign of the right-hand side. Note that for $\eta$ defined as the solution to $an - \beta n^{1 + \gamma \theta} - a\mu^{-1} n^{\gamma} = 0$, we have that there is \text{always} (strictly positive) accretion of cells for level sets with $0 < \xi < \eta$, due to the non-negativity of $W$. 
\item \textbf{Expansion and contraction of the level set.} Because the level set is transported by the fluid velocity $\bd{u}$, we have local expansion/contraction of the level set depending on the sign of $\text{div}(\bd{u})$. The magnitude of $\text{div}(\bd{u})$ gives the time derivative of the local expansion or contraction of volume of the level set, according to $\text{div}(\bd{u}) = \mu^{-1}(an^{\gamma} - W)$. 
\end{itemize}

To make this level set analysis concrete, we introduce the \textbf{cumulative distribution function} for the cell density at each time $t$, given by $F_{t}: [0, \infty) \to [0, \text{meas}(\Omega)]$, defined as
\begin{equation}\label{cdf}
F_{t}(\xi) = \text{meas}\{n(t, \cdot) \le \xi\}.
\end{equation}
For example, by the maximum principle discussed in Proposition \ref{maxprinciple}, we have that $F_{t}(\xi) = \text{meas}(\Omega)$ for $\xi \ge \max(n^{*}, M)$ where $\displaystyle M := \max_{x \in \overline{\Omega}} n_{0}(x)$ for smooth solutions. The main intermediate result about the dynamics of the tumor growth model is as follows:

\begin{proposition}\label{vanishingcdf}
There exists $\nu \in (0, \eta)$ such that $\lim_{t \to \infty} F_{t}(\xi) = 0$ for all $0 \le \xi \le \nu$. 
\end{proposition}

\medskip

To prove this result, we will first study the flow map of the vector field $\bd{u} := -\nabla W$ to analyze the transport portion of the first equation in \eqref{rewritten} for the cell density, in Section \ref{transport}. Then, we establish Proposition \ref{vanishingcdf} by showing that $W(t, x) \ge w > 0$ for some constant $w$, on a set of measure arbitrarily close to the full measure of $\Omega$. In this case, the second equation in \eqref{rewritten} would imply that the divergence of $\bd{u}$ is strictly negative for almost the entire vacuum set $\{n(t, \cdot) = 0\}$, so that the vacuum is essentially ``exponentially shrinking" in time. Showing this lower bound on $W$ will involve carefully analyzing the elliptic equation $-\mu \Delta W + W = F$, where $F = an^{\gamma}$, which is done in the appendix, Section \ref{appendix}. Finally, we put together all of these components to prove Proposition \ref{vanishingcdf} via a level set analysis in Section \ref{levelsetsec}, and then in Section \ref{finalproof}, we complete the proof of Theorem \ref{longtime}.

\subsection{The flow map}\label{transport}

We assume that we have a global strong solution $(n, W)$ to \eqref{equations} with $n \in C(0, T; W^{1, p}(\Omega))$ and $W \in C(0, T; W^{3, p}(\Omega))$ for all $T > 0$, for some $p > d$. Thus, the flow velocity $\bd{u} := -\nabla W$ is uniformly Lipschitz as a spatial function, with 
\begin{equation*}
\|\bd{u}\|_{C(0, T; W^{1, \infty}(\Omega))} \le C_{T}, \qquad \text{ for all } T > 0,
\end{equation*}
where $C_{T}$ is a constant depending on the final time $T$, increasing in $T$. The spatial Lipschitz regularity of the flow velocity $\bd{u}$ locally in time allows us to define an associated \textit{flow map},
\begin{equation*}
\Phi_{t}(\cdot): \Omega \to \Omega,
\end{equation*}
where we define $\Phi_{t}(x)$ to be the solution to the following ordinary differential equation:
\begin{equation}\label{flow}
\frac{d}{dt}\Phi_{t}(x) = \bd{u}(t, \Phi_{t}(x)), \qquad \Phi_{0}(x) = x.
\end{equation}
Since $\bd{u}$ is uniformly Lipschitz on bounded time intervals, we claim that the flow map is globally well-defined, and furthermore, it is a bijective map.

\begin{proposition}\label{flowprop}
Consider $W \in C(0, T; W^{3, p}(\Omega))$ for $p > d$ defined as a solution to the elliptic equation $-\mu \Delta W + W = an^{\gamma}$ on $\Omega$ for $n \in C(0, T; W^{1, p}(\Omega))$ with Neumann boundary conditions, and $\bd{u} := -\nabla W$. For each $t \ge 0$, the flow map $\Phi_{t}( \cdot): \Omega \to \Omega$ defined above is a bijective, Lipschitz continuous map. In particular, the flow of $\bd{u}$ is defined globally in time. Furthermore, $\partial_{t}\Phi_{t}(x)$ is a well-defined Lipschitz continuous function.
\end{proposition}

\begin{proof}
\noindent \textbf{Global existence.} We extend $\bd{u}$ on $\overline{\Omega}$ to $\R^{n}$, using a linear extension operator $E: W^{3, p}(\Omega) \to W^{3, p}(\R^{n})$ such that $E(\bd{u})|_{\Omega} = \bd{u}$ and $\|E(\bd{u})\|_{W^{k, p}(\R^{n})} \le C\|\bd{u}\|_{W^{k, p}(\Omega))}$ for a constant $C$ depending only on $\Omega$, see Chapter 5 of \cite{AdamsFournier}. To simplify notation, we will denote the extension of $\bd{u}$ to all of $\R^{n}$ still by $\bd{u}$, and we note that the extension satisfies $\bd{u} \in C(0, T; W^{3, p}(\R^{n}))$. 

We then consider an open ball $\tilde{\Omega}$ in $\R^{n}$ containing $\Omega$ such that $\text{dist}(\partial \Omega, \partial \tilde{\Omega}) > R := \text{diam}(\Omega)$ and we consider the ODE \eqref{flow} for $x \in \Omega$. We claim that for every $x \in \Omega$ and arbitrary $T > 0$, this differential equation is solvable up to time $T$. (Since $\bd{u}$ is Lipschitz continuous on $[0, T]$, it is at least locally uniquely solvable for every $x$ in the larger set $\tilde{\Omega}$).

Hence, fix some $x \in \Omega$ and an arbitrary time $T > 0$. We claim that $\Phi_{t}(x)$ is defined up to time $T > 0$. Because $\bd{u} \in C(0, T; W^{2, p}(\R^{n}))$ for $p > d$, there exists a uniform constant $L_T$ such that 
\begin{equation}\label{LT}
\sup_{t \in [0, T]} \|\bd{u}(t, \cdot)\|_{C^{1}(\R^{n})} \le L_T.
\end{equation}
Then, for $R := \text{diam}(\Omega)$, there exists a solution $\Phi_{t}(x) \in C(0, \tau; \R^{d})$ for $0 \le t \le \tau$ satisfying
\begin{equation}\label{Rbound}
\sup_{t \in [0, \tau]} |\Phi_{t}(x) - x| \le R,
\end{equation}
for any positive $\tau(T)$ depending on $T$, which satisfies $\tau(T) < \min(1/L_{T}, R/L_{T})$. To extend the solution in time, we want to iterate this procedure, using $\Phi_{\tau(T)}(x)$ as new initial data to extend the interval of existence, but the time of existence $\tau$ depends on the potential deviation $R$ from the initial data in \eqref{Rbound}, see for example the proof of Theorem 2.1 in \cite{TaylorPDE}. Namely, if $\Phi_{\tau(T)}(x)$ is less than $R$ away from the boundary of the larger set $\tilde\Omega$, we would need to shrink $R$ to continue this procedure, which would further shrink the interval of existence in time that we can extend the solution.

However, this is not the case. Namely, $\Phi_{\tau(T)}(x) \in \Omega$, and hence, we can iterate this procedure from $\tau(T)$ with the same $R := \text{diam}(\Omega)$ in \eqref{Rbound} and extend in time repeatedly by the same $\tau(T)$ until the final $T > 0$. To see this, suppose for contradiction that $\Phi_{\tau(T)}(x)$ leaves the set $\Omega$ within $0 \le t \le \tau$. Then, let $t_{bdry}$ be the least time for which $\Phi_{t_{bdry}}(x) \in \partial \Omega$. Then, $\Phi_{t}(x) \in \Omega$ for all $0 \le t < t_{bdry}$. But this contradicts the fact that the ODE \eqref{flow} has a \textit{unique (local in time) solution} with initial condition $\Phi_{t_{bdry}}(\tilde{x}) = \tilde{x}$, where $\tilde{x} := \Phi_{t_{bdry}}(x)$, on some interval $t \in (t_{bdry} - \epsilon, t_{bdry} + \epsilon)$ for sufficiently small $\epsilon > 0$, since such a solution must lie completely on $\partial \Omega$, as $\bd{u} := -\nabla W$ is tangential to $\partial \Omega$ by the Neumann boundary condition. So by contradiction, $\Phi_{t}(x)$ stays within $\Omega$ up to the local time of existence $\tau(T)$, which allows us to iteratively extend by (the same) $\tau(T)$ until we obtain a solution on $[0, T]$. 

\medskip

\noindent \textbf{Injectivity.} We note that the verification of the injectivity of $\Phi_{t}(\cdot): \Omega \to \Omega$ follows directly from the uniqueness part of the local existence theorem for solutions of differential equations, which we can apply since $\bd{u}$ is Lipschitz. (If $\Phi_{t}(x_1) = \Phi_{t}(x_2)$ for two different points $x_1, x_2 \in \Omega$, this would contradict uniqueness of solutions to the flow map locally).

\medskip

\noindent \textbf{Surjectivity.} Given some fixed $t_0 > 0$, we can find for any $x \in \Omega$ a corresponding $y \in \Omega$ such that $x = \Phi_{t_0}(y)$, by considering the flow (after time $t_0$) of the reversed vector field $\bd{v}(t, x) := -\bd{u}(t_0 - t, x)$ starting at $x$, which would give the point $y$.

\medskip

\noindent \textbf{Lipschitz continuity of $\Phi_{t}(\cdot)$.} By integrating the differential equation \eqref{flow}, we obtain:
\begin{equation*}
|\Phi_{t}(x) - \Phi_{t}(y)| \le |x - y| + \int_{0}^{t} |\bd{u}(s, \Phi_{s}(x)) - \bd{u}(s, \Phi_{s}(y))| ds \le |x - y| + L_{T} \int_{0}^{t} |\Phi_{s}(x) - \Phi_{s}(y)| ds,
\end{equation*}
where $L_{T}$ is defined in \eqref{LT}. Thus by Gronwall's inequality, $|\Phi_{t}(x) - \Phi_{t}(y)| \le |x - y|e^{L_{T}t}$, for $x, y \in \Omega$, so $\Phi_{t}(\cdot)$ is Lipschitz continuous with Lipschitz constant bounded above by $e^{L_{T}t}$.
\end{proof}

Since the flow $\Phi_{t}(x)$ of the vector field $\bd{u}$ is globally defined in time, we can use the method of characteristics to explicitly write the solution to the first equation in \eqref{rewritten} as
\begin{equation}\label{repformula}
n(t, x) = n_{0}(\Phi^{-1}_{t}(x)) + \int_{0}^{t} (\alpha n - \beta n^{1 + \gamma \theta} - a\mu^{-1} n^{1 + \gamma} + \mu^{-1} nW)(s, \Phi_{s}(\Phi^{-1}_{t}(x))) ds,
\end{equation}
since by the chain rule, this equation is the same as 
\begin{equation*}
\frac{d}{ds}n(s, \Phi_{s}(x)) = (\alpha n - \beta n^{1 + \gamma \theta} - a\mu^{-1} n^{1 + \gamma} + \mu^{-1} nW)(s, \Phi_{s}(x)).
\end{equation*}

Since by Proposition \ref{flowprop}, the flow map $\Phi_{t}(x)$ is Lipschitz continuous as a map from $\Omega$ to itself, for each $t \ge 0$, we can establish a relationship between the rate of expansion/contraction of sets under the flow $\Phi_{t}(\cdot)$ and the divergence of the vector field $\bd{u}$ which generates the flow.

\begin{proposition}\label{divuprop}
The flow map $\Phi_{t}(\cdot): \Omega \to \Omega$ satisfies the following differential equation:
\begin{equation}\label{detode}
\frac{d}{dt}\Big(\text{det}(\nabla \Phi_{t}(\cdot))\Big) = \text{div}(\bd{u})(t, \Phi_{t}(\cdot)) \text{det}(\nabla \Phi_{t}(\cdot)), \qquad \text{det}(\nabla \Phi_{0}(\cdot)) = 1.
\end{equation}
Furthermore, for any measurable set $S \subset \Omega$, 
\begin{equation}\label{measidentity}
\text{meas}(\Phi_{t}(S)) = \int_{\Omega} 1_{S}(x) \text{det}(\nabla \Phi_{t}(x)) dx = \int_{\Omega} 1_{S}(x) \text{exp}\left(\int_{0}^{t} \text{div}(\bd{u}(s, \Phi_{s}(x))) ds\right) dx.
\end{equation}
In differential form,
\begin{equation}\label{measidentityode}
\frac{d}{dt}\Big(\text{meas}(\Phi_{t}(S))\Big) = \int_{\Omega} 1_{S}(x) \text{div}\Big(\bd{u}(t, \Phi_{t}(x))\Big) \text{exp}\left(\int_{0}^{t} \text{div}\Big(\bd{u}(s, \Phi_{s}(x))\Big) ds\right) dx.
\end{equation}
\end{proposition}

\begin{proof}
We can verify that the Jacobian of the Lipschitz map $\Phi_{t}(\cdot): \Omega \to \Omega$ satisfies the differential equation in \eqref{detode} by direct computation using the chain rule, and furthermore, $\nabla \Phi_{0}(x)$ is the identity matrix for all $x \in \Omega$. Note that if we integrate \eqref{detode}, we obtain
\begin{equation}\label{integratedet}
\text{det}(\nabla \Phi_{t}(x)) = \text{exp}\left(\int_{0}^{t} \text{div}(\bd{u}(s, \Phi_{s}(x)))\right) ds.
\end{equation}
The first equality in \eqref{measidentity} follows from writing $\displaystyle \text{meas}(\Phi_{t}(S)) = \int_{\Omega} 1_{\Phi(t, S)}(x) dx$ and then using a change of variables $x \to \Phi_{t}(x)$. The change of variables can be accomplished by using a change of variables formula, see Theorem 3.9 in \cite{EvansGariepy}, which we can apply since $\Phi_{t}(\cdot)$ is a Lipschitz continuous bijective map from $\Omega$ to $\Omega$. The second equality in \eqref{measidentity} follows by applying the explicit solution in \eqref{integratedet}. 
\end{proof}

\subsection{Level set analysis}\label{levelsetsec}

To prove Proposition \ref{vanishingcdf}, we carefully analyze the behavior of the system near vacuum. This will involve analyzing the relative balance between the effects of accretion/depletion and  expansion/contraction on the level sets of the initial data, as the system evolves. In particular, we will use the formula in \eqref{measidentity} which governs the expansion/contraction of measurable sets under the flow map of the time-dependent vector field $\bd{u}(t, x)$ to perform a \textit{level set analysis}, namely an analysis of how level sets of the initial data $\{n_{0}(x) = \xi\} \subset \Omega$ evolve under the flow, both in terms of the change in the level itself $\xi$ due to accretion/depletion of cells and also in terms of expansion/contraction due to the divergence of the flow $\bd{u}$. 

From the first equation of \eqref{rewritten}, given an initial level set $E_{\xi} := \{n_{0}(x) = \xi\}$, we can use the nonnegativity of $W$ in Proposition \ref{nonnegativeW} and the comparison principle to conclude that 
\begin{equation}\label{ncomparison}
n(t, x) \ge \overline{n}(t), \qquad \text{ for all } x \in \Phi_{t}(E_{\xi}),
\end{equation}
where $\overline{n}(t)$ (as just a scalar-valued function of time) satisfies the ODE:
\begin{equation}\label{nbarode}
\frac{d\overline{n}}{dt} = \alpha \overline{n} - \beta \overline{n}^{1 + \gamma \theta} - a\mu^{-1}\overline{n}^{1 + \gamma}, \qquad \overline{n}(0) = \xi.
\end{equation}
Using a stability analysis and the fundamental theorem on existence and uniqueness of differential equations (Theorem 2.1 in \cite{TaylorPDE}), we make the following observations:
\begin{itemize}
    \item If $\xi = 0$, then the solution to this ODE is constant, $\overline{n}(t) = 0$. If $\xi = \eta$, then the solution to this ODE is also constant, $\overline{n}(t) = \eta$.
    \item If $0 < \xi < \eta$, then the solution to this ODE is globally defined and increasing in $t$, with $0 < \overline{n}(t, x) < \eta$ for all $t \in (-\infty, \infty)$.
    \item If $\xi > \eta$, then $\overline{n}(t)$ exists for all $t$, and $\overline{n}(t)$ is decreasing with $\overline{n}(t) > \eta$ for all $t \ge 0$.
\end{itemize}
These observations allow us to perform a careful level set analysis, which will  analyze how level sets change under accretion/depletion of cells as they are transported by the flow. This will allow us to analyze how the measure of level sets change under the flow. Namely, the goal is to prove Proposition \ref{vanishingcdf}, namely show that for some $\nu \in (0, \eta)$, $\lim_{t \to \infty} \text{meas}\{n(t, \cdot) \le \nu\} = 0$. To do this, we using these intermediate steps proved in the next three lemmas:
\begin{itemize}
    \item \textbf{Step 1.} We show that 
    \begin{equation}\label{step1eqn}
    \lim_{t \to \infty} F_{t}(0) = r \qquad \text{for some constant $r \ge 0$}.
    \end{equation}
    This shows that the \textbf{measure of the vacuum converges to some constant}, but since $r$ could be positive, \textit{this does not definitively show that the vacuum vanishes in long time} (which is what we would need for Proposition \ref{vanishingcdf}).
    \item \textbf{Step 2.} We show that for some $\nu > 0$ sufficiently small,
    \begin{equation}\label{step2eqn}
        \lim_{t \to \infty} \text{meas}\{x \in \Omega : 0 < n(t, x) < \nu\} = 0.
    \end{equation}
    Thus, to prove Proposition \ref{vanishingcdf}, it remains to improve the result from Step 1 to show that the vacuum $\{n(t, x) = 0\}$ vanishes in long time.
    \item \textbf{Step 3.} In this step, we improve the result in Step 1 and show that $r = 0$ in \eqref{step1eqn}, so that the \textbf{vacuum vanishes in long time}. To do this, we use the elliptic equation $-\mu \Delta W + W = an^{\gamma}$ for $W$ to show that for all $t \ge \tau$ for $\tau$ sufficiently large:
    \begin{equation}\label{step3eqn}
    W(t, x) \ge w > 0 \qquad \text{for some positive constant $w$},
    \end{equation}
    on a set of almost full measure on $\Omega$. This would allow us, using the second equation in \eqref{rewritten}, to conclude that the constant $r$ in Step 1 is equal to zero. This would complete the proof of Proposition \ref{vanishingcdf}. 
\end{itemize}

\begin{lemma}[Step 1, see \eqref{step1eqn}]\label{step1}
Let $n(t, x) \in C(0, T; W^{1, p}(\Omega))$ and $W \in C(0, T; W^{3, p}(\Omega))$ for $p > d$ solve \eqref{equations} for nonnegative smooth initial data $n_{0}(x) \ge 0$. Then, $\lim_{t \to \infty} F_{t}(0)$ exists and this limit is equal to $\lim_{t \to \infty} F_{t}(0) := r$, for some constant $0 \le r \le \text{meas}\{n_{0}(x) = 0\}$, where $F_{t}(\xi)$ is defined in \eqref{cdf}. Furthermore, $F_{t}(0)$ is a decreasing function of time.
\end{lemma}

\begin{proof}
Using \eqref{ncomparison} and the differential equation for $\overline{n}$ in \eqref{nbarode}, we can observe that 
\begin{equation*}
\Phi_{t}^{-1}(V_{t}) := \Phi^{-1}_{t}\Big(\{x \in \Omega : n(t, x) = 0\}\Big) \subset \{x \in \Omega : n_{0}(x) = 0\}.
\end{equation*}
Therefore, using the identity in \eqref{measidentity}, we obtain that
\begin{equation}\label{Vtmeas}
\text{meas}(V_{t}) = \int_{\Omega} 1_{\Phi^{-1}_{t}(V_{t})} \text{exp}\left(\int_{0}^{t} \text{div}(\bd{u}(s, \Phi_{s}(x))) ds\right) dx.
\end{equation}
Note that for $x \in \Phi_{t}^{-1}(V_{t})$, $n(s, \Phi_{s}(x)) = 0$ for all $s \in [0, t]$ by the properties of \eqref{nbarode} (namely that the solution $\overline{n}$ at any strictly positive initial condition remains strictly positive). Therefore, using the second equation in \eqref{rewritten}, $\text{div}(\bd{u}) = a\mu^{-1}n^{\gamma} - \mu^{-1}W$, which is non-positive whenever $n = 0$ by Proposition \ref{nonnegativeW}, we conclude that $\text{div}(\bd{u}(s, \Phi_{x}(s))) = -W(s, \Phi_{x}(s)) \le 0$ for $x \in \Phi_{t}^{-1}(V_{t})$ and for all $s \in [0, t]$. By differentiating \eqref{Vtmeas} with respect to $t$, we conclude that $\displaystyle \frac{d}{dt}\text{meas}(V_{t}) \le 0$, so that the measure $F_{t}(0)$ of the vacuum set is decreasing in time. Therefore, because $F_{t}(0)$ is nonnegative and decreasing in $t$, it must have a nonnegative limit, $\lim_{t \to \infty} F_{t}(0) := r \ge 0$.
\end{proof}

\begin{lemma}[Step 2, see \eqref{step2eqn}]\label{step2}
Let $n(t, x) \in C(0, T; W^{1, p}(\Omega))$ and $W(t, x) \in C(0, T; W^{3, p}(\Omega))$ for $p > d$ be solutions to \eqref{equations} for smooth initial data $n_{0}(x) \ge 0$. Then, there exists $\nu > 0$ sufficiently small such that $\lim_{t \to \infty} \text{meas}\Big(\{x \in \Omega : 0 < n(t, x) \le \nu\}\Big) = 0$, and hence, for any $0 < \xi \le \nu$, $\lim_{t \to \infty} F_{t}(\xi) = r := \lim_{t \to \infty} F_{t}(0)$.
\end{lemma}

\begin{proof}
Recall the comparison principle in \eqref{ncomparison} and the ODE for $\overline{n}$ from \eqref{nbarode}. Choose $\nu \in (0, \eta)$, where we recall that $\eta$ is the positive zero of $\alpha - \beta n^{\gamma \theta} - a\mu^{-1} n^{\gamma} = 0$, such that $\displaystyle \beta \nu^{\gamma \theta} + a \mu^{-1} \nu^{\gamma} < \frac{\alpha}{2}$. Then, for $0 < \xi \le \nu$, any solution of \eqref{nbarode} satisfying $\overline{n}(0) = \xi$ is defined for all $t \in (-\infty, \infty)$, and furthermore, for all $t \le 0$:
\begin{equation*}
\frac{d\overline{n}}{dt} = \alpha \overline{n} - \beta \overline{n}^{1 + \gamma \theta} - a\overline{n}^{1 + \gamma} > \frac{\alpha}{2}\overline{n} > 0.
\end{equation*}
So a solution $\overline{n}$ to \eqref{nbarode} with $\overline{n}(0) = \xi$ for $0 < \xi \le \nu$ satisfies the following decay estimate: $\overline{n}(t) \le \xi e^{\alpha t/2}$, for $t \le 0$. By shifting in time, we have that for any time $t_{0} \ge 0$, 
\begin{equation}\label{nuexpdecay}
\overline{n}(t_{0}) = \xi \text{ for } 0 < \xi \le \nu \quad \Longrightarrow \quad \overline{n}(t) \le \xi e^{-\alpha(t_0 - t)/2} \text{ for all } 0 \le t \le t_0.
\end{equation}

So consider the level set $S_{t, \nu} := \{x \in \Omega : 0 < n(t, x) \le \nu\}$. Then, using \eqref{measidentity}, we obtain
\begin{equation}\label{measnu}
\text{meas}(S_{t, \nu}) = \int_{\Omega} 1_{\Phi_{t}^{-1}(S_{t, \nu})} \text{exp}\left(\int_{0}^{t} \text{div}(\bd{u}(s, \Phi_{s}(x))) ds\right) dx.
\end{equation}
By \eqref{nuexpdecay}, we have that
\begin{equation}\label{nuineq1}
\Phi_{t}^{-1}(S_{t, \nu}) \subset \{x \in \Omega : 0 < n_{0}(x) \le \nu e^{-\alpha t/2}\}.
\end{equation}
Using the second equation of \eqref{rewritten}, $W \ge 0$ from Proposition \ref{nonnegativeW}, and the estimate that for $x \in \Phi_{t}^{-1}(S_{t, \nu})$, 
\begin{equation*}
0 < n(s, \Phi_{s}(x)) \le \nu e^{-\alpha(t - s)/2} \quad \text{ for } s \in [0, t],
\end{equation*}
we obtain that
\begin{equation}\label{nuineq2}
\text{div}(\bd{u}(s, \Phi_{s}(x))) \le a\mu^{-1}\nu^{\gamma} e^{-\alpha \gamma(t - s)/2}, \qquad \text{ for } x \in \Phi_{t}^{-1}(S_{t, \nu}) \text{ and } 0 \le s \le t.
\end{equation}
So applying \eqref{nuineq1} and \eqref{nuineq2} in \eqref{measnu}, we obtain that:
\begin{align*}
\text{meas}\Big(\{x \in \Omega &: 0 < n(t, x) \le \nu\}\Big) \\
&\le \text{meas}\Big(\{x \in \Omega : 0 < n_{0}(x) \le \nu e^{-\alpha t/2}\}\Big) \text{exp} \left(a\mu^{-1}\nu^{\gamma} \int_{0}^{t} e^{-\alpha \gamma (t - s)/2} ds\right) \\
&\le \text{meas}\Big(\{x \in \Omega : 0 < n_{0}(x) \le \nu e^{-\alpha t/2}\}\Big) e^{2a\mu^{-1}\alpha^{-1}\gamma^{-1}\nu^{\gamma}}.
\end{align*}
The result follows once we note that $\lim_{t \to \infty} \text{meas}\Big(\{x \in \Omega : 0 < n_{0}(x) \le \nu e^{-\alpha t/2}\}\Big) = 0$.
\end{proof}

Finally, we complete the proof of Proposition \ref{vanishingcdf} by showing the result of Step 3 in \eqref{step3eqn}. Namely, we want to show that the vacuum set vanishes as $t \to \infty$. As of now, from Lemma \ref{step1}, we have that for a strong solution $n \in C(0, T; W^{1, p}(\Omega))$ with $W \in C(0, T; W^{3, p}(\Omega))$ and $p > d$ to \eqref{equations} for smooth nonnegative initial data $n_0$ that is not identically zero, $\lim_{t \to \infty} F_{t}(0) := r$, for $0 \le r < \text{meas}(\Omega)$. We claim that $r = 0$, by showing that $W(t, x) \ge w > 0$ for all $t \ge 0$ on a subset of $\Omega$ of arbitrarily large measure, so that by \eqref{rewritten}, we see that $\text{div}(\bd{u}) \le -\mu^{-1}w < 0$ on almost the entire vacuum set so that the measure of the vacuum is essentially exponentially decreasing in time by \eqref{measidentity}. 

To show the strict positivity of $W$ away from zero, we analyze the elliptic equation
\begin{equation}\label{elliptic}
-\mu \Delta W + W = f, \qquad \nabla W \cdot \bd{n}|_{\partial \Omega} = 0,
\end{equation}
for a general source term $f$. Recall that the \textbf{Green's function $G(x, y)$} in the context of equation \eqref{elliptic}, is defined to be a function $G(x, y)$ for $x, y \in \Omega$ that satisfies:
\begin{equation*}
-\mu \Delta_{x} G(x, y) + G(x, y) = \delta_{y}, \qquad \nabla_{x} G(x, y) \cdot \bd{n}|_{\partial \Omega} = 0,
\end{equation*}
where $\delta_{y}$ is the Dirac delta distribution centered at the point $y \in \Omega$. Given the Green's function, the solution to the inhomogeneous problem \eqref{elliptic} is:
\begin{equation}\label{greenrep}
W(x) = \int_{\Omega} G(x, y) f(y) dy, \qquad \text{ for } x \in \Omega. 
\end{equation}
We emphasize that the Green's function is dependent on the specific geometry of $\Omega$. By studying properties of the Green's function $G(x, y)$, we are able to obtain lower bounds on $W(t, x)$. To do this, we use standard properties of \eqref{elliptic} and of the Green's function, which we prove for the interested reader in the appendix, see Section \ref{appendix}.

\if 1 = 0
Next, we make the following standard observation about the continuity of the Green's function off of the diagonal. We give the proof of the result in the appendix, along with properties of the fundamental solution of the operator $-\mu \Delta + I$.

\begin{lemma}
Let $\Omega$ be a simply connected open domain with smooth boundary. Then, the Green's function $G(x, y)$ for the uniformly elliptic operator $-\mu \nabla + 1$ is continuous as a function of $(x, y)$ on $\Omega^{2} \setminus D$ where $D := \{(x, y) \in \Omega^{2} : x = y\}$ is the diagonal set.
\end{lemma}

\begin{proof}
The proof of this result is given in the appendix, see Proposition \ref{greenproperties}.
\end{proof}

\medskip

\fi

We will now prove the asymptotic vanishing of vacuum in Step 3 in \eqref{step3eqn} as follows. This improves the result from Step 1 in Lemma \ref{step1} by showing that the limiting measure of the vacuum set $r$ is equal to $0$ in long time.

\begin{lemma}[Step 3, see \eqref{step3eqn}]\label{step3}
Let $n(t, x) \in C(0, T; W^{1, p}(\Omega))$ and $W \in C(0, T; W^{3, p}(\Omega))$ for $p > d$ be solutions to \eqref{equations} for smooth nonnegative initial data $n_{0}(x)$ that is not identically zero. Then, $\lim_{t \to \infty} F_{t}(0) = 0$.
\end{lemma}

\begin{proof}
It suffices to show that $r = 0$ in Lemma \ref{step1}. We assume for contradiction that
\begin{equation}\label{rdef}
\lim_{t \to \infty} \text{meas}\{n(t, \cdot) = 0\} := r > 0,
\end{equation}
where $r < \text{meas}(\Omega)$ since the initial data $n_0(x)$ is not identically zero. For the choice of $\nu \in (0, \eta)$ in Lemma \ref{step2}, $\lim_{t \to \infty} \text{meas}\left\{n(t, \cdot) > \nu\right\} = \text{meas}(\Omega) - r$. Thus, for some $\tau$ sufficiently large, 
\begin{equation}\label{measT}
\text{meas}\{n(t, \cdot) > \nu\} \ge \frac{\text{meas}(\Omega) - r}{2} \quad \text{ and } \quad \text{meas}\{n(t, \cdot) = 0\} \le \frac{5r}{4}, \qquad \text{ for all } t \ge \tau.
\end{equation}

Define the following compact set $K^{\delta}_{y} \subset \Omega$ for each $y \in \Omega$ and $\delta > 0$:
\begin{equation*}
K^{\delta}_{y} := \{x \in \Omega : |x - y| \ge \delta \text{ and } \text{dist}(x, \partial \Omega) \ge \delta\},
\end{equation*}
and let $K^{\delta} \subset \Omega \times \Omega$ be the compact set:
\begin{equation*}
K^{\delta} := \bigcup_{y \in \Omega, \text{dist}(y, \partial \Omega) \ge \delta} \Big(\{y\} \times K_{y}^{\delta}\Big).
\end{equation*}
Choose $\delta$ sufficiently small so that 
\begin{equation}\label{Kdeltameas}
\text{meas}(\Omega \setminus K_{x}^{\delta}) \le \min\left(\frac{\text{meas}(\Omega) - r}{4}, \frac{r}{4}\right), \qquad \text{ for all } x \in \Omega^{\delta} := \{x \in \Omega : \text{dist}(x, \partial \Omega) \ge \delta\}.
\end{equation}
Because $K^{\delta}$ is a compact subset of $\Omega \times \Omega$ on which the Green's function $G(x, y)$ is continuous and strictly positive (see Proposition \ref{greenproperties}), we obtain that for some positive constant $m$:
\begin{equation}\label{Gpositive}
G(x, y) \ge m > 0 \qquad \text{ for $(x, y) \in K^{\delta}$}.
\end{equation}

We will use \eqref{Gpositive} to derive a contradiction, by showing that most of the vacuum set (due to the positivity of the Green's function) is exponentially shrinking, which will contradict the limit in \eqref{rdef} since $r > 0$. By \eqref{greenrep}, $\displaystyle
W(t, x) = a\int_{\Omega} G(x, y) n^{\gamma}(t, y) dy$. Observe that since $n(t, x) \ge 0$ and $G(x, y)$ is strictly positive on $\Omega \times \Omega$ by Proposition \ref{greenproperties}, we can obtain by \eqref{Gpositive} that:
\begin{equation*}
W(t, x) \ge am\int_{K_{x}^{\delta}} n^{\gamma}(t, y) dy, \qquad \text{ for } x \in \Omega \text{ with } \text{dist}(x, \partial \Omega) \ge \delta.
\end{equation*}
By combining \eqref{measT} and \eqref{Kdeltameas}, we see that for all $t \ge \tau$ and for all $x \in \Omega$ with $\text{dist}(x, \partial \Omega) \ge \delta$:
\begin{equation*}
\text{meas}\Big(\{n(t, \cdot) > \nu\} \cap K^{\delta}_{x}\Big) \ge \frac{\text{meas}(\Omega) - r}{4},
\end{equation*}
and hence, for all $t \ge \tau$:
\begin{equation*}
W(t, x) \ge am\left(\frac{\text{meas}(\Omega) - r}{4}\right) \nu^{\gamma} := w > 0, \qquad \text{ for } x \in \Omega^{\delta} := \{x \in \Omega: \text{dist}(x, \partial \Omega) \ge \delta\}.
\end{equation*}

Note that the condition \eqref{Kdeltameas} implies that $\text{meas}(\Omega \setminus \Omega^{\delta})\le r/4$, and thus, $W(t, x) \ge w > 0$ for all $x \in \Omega^\delta$, which is a set of measure at least $\displaystyle \text{meas}(\Omega) - \frac{r}{4}$. So since $\{n(t, \cdot) = 0\}$ has measure at least $r$ for all $t$ by Lemma \ref{step1}, we have that for all $t \ge \tau$:
\begin{equation}\label{3r4}
\text{meas}\Big(\{n(t, \cdot) = 0\} \cap \{W(t, \cdot) \ge w\}\Big) \ge \frac{3r}{4},
\end{equation}
and on this set $\{n(t, \cdot) = 0\} \cap \{W(t, \cdot) \ge w\}$, which has measure at least $3r/4$ for all $t \ge \tau$, $\text{div}(\bd{u}) = -\mu^{-1} W \le -\mu^{-1}w$. Therefore, by the vacuum bound in \eqref{measT} and \eqref{3r4}, for all $t \ge \tau$:
\begin{equation*}
\text{meas}\Big(\{n(t, \cdot) = 0\} \cap \{W(t, \cdot) < w\}\Big) \le \frac{5r}{4} - \frac{3r}{4} = \frac{r}{2},
\end{equation*}
on this set, $\text{div}(\bd{u}) = -\mu^{-1}W \le 0$ by Proposition \ref{nonnegativeW}. To summarize, for each $t \ge \tau$, the vacuum $\{n(t, \cdot) = 0\}$ consists of two disjoint sets:
\begin{itemize}
\item A vacuum set $V_{1}(t) := \{n(t, \cdot) = 0\} \cap \{W(t, \cdot) \ge w\}$ on which $W$ is positive and strictly bounded away from zero, with 
\begin{equation}\label{V1prop}
\text{meas}(V_{1}(t)) \ge 3r/4, \qquad \text{div}(\bd{u}) \le -\mu^{-1}w < 0 \text{ on } V_{1}(t).
\end{equation}
\item A vacuum set $V_{2}(t) := \{n(t, \cdot) = 0\} \setminus V_{1}(t)$ of measure $\text{meas}(V_{2}(t)) \le r/2$, on which we only know that $\text{div}(\bd{u}) \le 0$. 
\end{itemize}
Let $V_{t} := \{x \in \Omega : n(t, x) = 0\}$ be the vacuum. By properties of the ODE \eqref{nbarode} and \eqref{ncomparison}, we have that $\Phi_{t}(V_{s}) = V_{s + t}$ and since $\text{div}(\bd{u}) \le 0$ by \eqref{rewritten} on the vacuum set, we conclude that $\displaystyle \frac{d}{dt}\Big(\text{meas}(V_{t})\Big) \le 0$. In addition, using \eqref{measidentityode}, for $t \ge \tau$:
\begin{align*}
\frac{d}{dt}(\text{meas}(V_{t})) &= \int_{\Omega} 1_{V_{\tau}}(x) \text{div}\Big(\bd{u}(t, \Phi_{t - \tau}(x))\Big) \text{exp}\left(\int_{0}^{t - \tau} \text{div}\Big(\bd{u}(s, \Phi_{s}(x))\Big) ds\right) dx \\
&\le \int_{\Omega} 1_{\Phi_{t - \tau}^{-1}(V_{1}(t))}(x) \text{div}\Big(\bd{u}(t, \Phi_{t - \tau}(x))\Big) \text{exp}\left(\int_{\tau}^{t} \text{div}\Big(\bd{u}(s, \Phi_{s}(x))\Big) ds\right) dx,
\end{align*}
since $\Phi_{t - \tau}^{-1}(V_{1}(t)) \subset \Phi_{t - \tau}^{-1}(V_{t}) = V_{\tau}$, $\Phi_{t - \tau}(V_{\tau}) = V_{t}$, and $\text{div}(\bd{u}) \le 0$ on the vacuum set. Since $\text{div}(\bd{u}) \le -\mu^{-1}w < 0$ on $V_{1}(t)$, we estimate that for all $t \ge \tau$:
\begin{align*}
\frac{d}{dt}(\text{meas}(V_{t})) &\le -\mu^{-1}w \int_{\Omega} 1_{\Phi_{t - \tau}^{-1}(V_{1}(t))}(x) \text{exp}\left(\int_{0}^{t - \tau} \text{div}\Big(\bd{u}(s, \Phi_{s}(x))\Big) ds\right) dx \\
&= -\mu^{-1}w \cdot \text{meas}(V_{1}(t)) \le -\frac{3}{4}\mu^{-1}wr < 0,
\end{align*}
by \eqref{V1prop}. Since $\displaystyle -\frac{3}{4}\mu^{-1}wr$ is a negative constant, this contradicts the initial assumption that $F_{t}(0) := \text{meas}(V_{t}) \ge r$ for all $t \ge 0$. This establishes the result that $r = 0$.
\end{proof}

Combining Lemma \ref{step1}, Lemma \ref{step2}, and Lemma \ref{step3} establishes Proposition \ref{vanishingcdf}, that there exists some $\nu > 0$ such that $\lim_{t \to \infty} F_{t}(\nu) := \lim_{t \to \infty} \text{meas}\Big(\{x \in \Omega : 0 \le n(t, x) \le \nu\}\Big) = 0$.

\subsection{Proof of Theorem \ref{longtime}}\label{finalproof}

Equipped with Proposition \ref{vanishingcdf}, we can now establish the long-time behavior result.

\begin{proof}[Proof of Theorem \ref{longtime}]

By Theorem \ref{vanishingcdf}, there exists $\nu > 0$, such that $0 < \nu < \eta$ and 
\begin{equation}\label{nuresult}
\lim_{t \to \infty} F_{t}(\nu) := \lim_{t \to \infty} \text{meas}\{x \in \Omega : n(t, x) \le \nu\} = 0.
\end{equation}
Furthermore, by Proposition \ref{maxprinciple}, for each $\lambda > n^{*}$, there exists a corresponding $T_{\lambda} > 0$ for which $0 \le n(t, x) \le \lambda$ for all $t \ge T_{\lambda}$. (Though Proposition \ref{maxprinciple} is for smooth solutions, this maximum bound holds for strong solutions with $n \in C(0, T; W^{1, p}(\Omega))$ and $W \in C(0, T; W^{3, p}(\Omega))$, which are unique, by an approximation argument, for example involving artificial viscosity as in \cite{WeberTrivisa, TrivisaWeberDrug}.)

Fix an arbitrary choice of $\lambda > n^{*}$ and note that for $\hat{G}(n) := \alpha n - \beta n^{1 + \gamma \theta}$, where $\hat{G}(n) = nG(n)$ for $G(n)$ as defined in \eqref{Gdef}:
\begin{equation*}
\frac{d}{dt}\left(\int_{\Omega} (\lambda - n) dx \right) = -\int_{\Omega} (\alpha n - \beta n^{1 + \gamma \theta}) = -\text{meas}(\Omega) \cdot \hat{G}(\lambda) + \int_{\Omega} \Big(\hat{G}(\lambda) - \hat{G}(n)\Big). 
\end{equation*}
Then, there exists a positive constant $c_{\nu}$ (independent of $\lambda$) such that for $\hat{G}(n) := \alpha n - \beta n^{1 + \gamma \theta}$:
\begin{equation}\label{lipG}
0 \le c_{\nu}(\lambda - n) \le \hat{G}(n) - \hat{G}(\lambda), \qquad \text{ for all } \nu \le n \le \lambda.
\end{equation}
To see this, we observe that due to the fact that $\hat{G}(z) := \alpha n - \beta n^{1 + \gamma \theta}$ is concave down, 
\begin{multline*}
0 \le \frac{\hat{G}(\nu)}{n^{*} - \nu}(\lambda - n) = \frac{\hat{G}(\nu) - \hat{G}(n^{*})}{n^{*} - \nu}(\lambda - n) \le \frac{\hat{G}(\nu) - \hat{G}(\lambda)}{\lambda - \nu}(\lambda - n) \le \hat{G}(n) - \hat{G}(\lambda), \\
\qquad \text{ for all } \nu \le n \le \lambda,
\end{multline*}
so we can set $\displaystyle c_{\nu} := \frac{\hat{G}(\nu)}{n^{*} - \nu}$. Then, we estimate that for $t \ge T_{\lambda}$:
\begin{equation}\label{Tlambdadiff}
\frac{d}{dt}\left(\int_{\Omega} (\lambda - n) dx\right) \le c_{\nu} \lambda F_{t}(\nu) - \text{meas}(\Omega) \cdot \hat{G}(\lambda) -c_{\nu}\int_{\Omega} (\lambda - n).
\end{equation}
We emphasize that $-\text{meas}(\Omega) \cdot \hat{G}(\lambda) > 0$ since $\hat{G}(\lambda) < 0$.

We want to show that $\displaystyle \lim_{t \to \infty} \int_{\Omega} (n^{*} - n) = 0$. Thus, consider an arbitrary $\epsilon > 0$, and choose $\lambda > n^{*}$ sufficiently close to $n^{*}$ such that 
\begin{equation}\label{lambdaeps}
\max\Big(\text{meas}(\Omega) \cdot (\lambda - n^{*}), c_{\nu}^{-1} \text{meas}(\Omega) \cdot (-\hat{G}(\lambda))\Big) < \frac{\epsilon}{4}. 
\end{equation}
Then, given this choice of $\lambda$, we have an associated time $T_{\lambda}$ for which the differential inequality \eqref{Tlambdadiff} holds for all $t \ge T_{\lambda}$. By \eqref{nuresult}, we choose a time $T_{\epsilon}$ sufficiently large such that 
\begin{equation}\label{Teps}
T_{\epsilon} \ge T_{\lambda}, \quad \text{ and } \quad \lambda F_{t}(\nu) < \frac{\epsilon}{4} \text{ for all } t \ge T_{\epsilon}.
\end{equation}
Then, we can consider the differential inequality \eqref{Tlambdadiff} from the initial time $T_{\epsilon}$, where we note that because $0 \le n(t, \cdot) \le \lambda$ for all $t \ge T_{\epsilon} \ge T_{\lambda}$, we have that $\displaystyle \int_{\Omega} (\lambda - n) \ge 0$ for all $t \ge T_{\epsilon}$. We hence obtain from \eqref{Tlambdadiff}, \eqref{lambdaeps}, and \eqref{Teps} that for all $t \ge T_{\epsilon}$:
\begin{equation*}
\frac{d}{dt}\left(\int_{\Omega} (\lambda - n) dx\right) \le \frac{\epsilon}{2}c_{\nu} - c_{\nu} \int_{\Omega} (\lambda - n) dx.
\end{equation*}
By solving this differential inequality, we obtain that
\begin{equation*}
0 \le \int_{\Omega} (\lambda - n)(t) \le \frac{\epsilon}{2} + Ce^{-c_{\nu}(t - T_{\epsilon})}, \qquad \text{ for } t \ge T_{\epsilon},
\end{equation*}
for some constant $C$, depending on the initial data $\displaystyle \int_{\Omega} (\lambda - n)(T_{\epsilon}, \cdot) dx$ at time $T_{\epsilon}$.
So for a sufficiently large time $\tau_{\epsilon} \ge T_{\epsilon}$, we have that
\begin{equation}\label{3eps4}
0 \le \int_{\Omega} (\lambda - n)(t) \le \frac{3\epsilon}{4}, \qquad \text{ for all } t \ge \tau_{\epsilon} \ge T_{\epsilon} \ge T_{\lambda}.
\end{equation}
We have that for all $t \ge T_{\lambda}$, for which we recall that $0 \le n(t, x) \le \lambda$:
\begin{align*}
\left|\int_{\Omega} (\lambda - n) - \int_{\Omega} |n^{*} - n|\right| &\le \left|\int_{n \ge n^{*}} \Big((\lambda - n) - (n - n^{*})\Big)\right| + \left|\int_{n < n^{*}} (\lambda - n^{*})\right| \\
&\le \int_{n \ge n^{*}} (\lambda - n^{*}) + \int_{n < n^{*}} (\lambda - n^{*}) \\
&\le \text{meas}(\Omega) (\lambda - n^{*}),
\end{align*}
so combining this with \eqref{lambdaeps} and \eqref{3eps4}, we obtain that $\displaystyle 0 \le \int_{\Omega} |n^{*} - n(t, x)| \le \epsilon$ for $t \ge \tau_{\epsilon}$, which establishes that $\displaystyle \lim_{t \to \infty} \int_{\Omega} |n(t, x) - n^{*}| = 0$. By using the upper bound from Proposition \ref{maxprinciple} that $0 \le n(t, x) \le M$, we also have that:
\begin{equation}\label{Lpconv}
\lim_{t \to \infty} \int_{\Omega} |n(t, x) - n^{*}|^{q} = 0, \qquad \text{ for all } 1 \le q < \infty.
\end{equation}

Consider arbitrary $1 \le q < \infty$. To show that $\lim_{t \to \infty} \|W(t, x) - W^{*}\|_{W^{2, q}(\Omega)} = 0$ for $W^{*} = a(n^{*})^{\gamma}$, note that
\begin{equation*}
-\mu \Delta (W - W^{*}) + (W - W^{*}) = a(n^{\gamma} - (n^{*})^{\gamma}), \qquad \nabla (W - W^{*}) \cdot \bd{n}|_{\partial \Omega} = 0.
\end{equation*}
Since $n \to n^{*}$ in $L^{q}(\Omega)$ as $t \to \infty$, we have by the mean value theorem and the maximum principle (Proposition \ref{maxprinciple}) that $\lim_{t \to \infty} \|a(n^{\gamma} - (n^{*})^{\gamma})\|_{L^{q}(\Omega)} = 0$ and thus by Calder\'{o}n-Zygmund estimates (Proposition \ref{CZest} and \ref{nonnegativeW}), we conclude that $\lim_{t \to \infty} \|W - W^{*}\|_{W^{2, q}(\Omega)} = 0$ also, for all $1 \le q < \infty$.
\end{proof}

\section{Conclusion}

Finally, we summarize the results in this manuscript and mention some interesting directions for future investigation that arise naturally from our results. We have shown, for sufficiently large viscosity parameter $\mu$, the existence of nontrivial global strong (and even smooth) solutions for data with sufficiently small gradient, which can intuitively be thought of as initial data with small deviations, but with no restriction on the smooth initial data nonnegative $n_{0}(x)$ itself (only on its gradient), other than that it is strictly positive (bounded away from vacuum). This is possible due to the additional dissipation arising from the effects of cell death in the growth rate function $G(p)$. While it was previously known that global weak solutions exist \cite{WeberTrivisa, TrivisaWeberDrug} for $n_{0} \in L^{\infty}(\Omega)$, we present in this manuscript new results on global strong solutions. This is a delicate matter, in particular because it is well-known that transport equations may exhibit singularities and hence, in general are best understood in a weak (or even a renormalized weak) sense, see \cite{DeLellis, DiPernaLions}. Even though it has been shown that the presence of damping may give rise to strong solutions for small data, for example to the 3D compressible Euler equations, even for these equations, singularities can form for arbitrary initial data, see the discussion in \cite{Sideris} for example. 

It would be particularly interesting to characterize more specific conditions for which global strong solutions exist for \eqref{equations}. The fact that our result in Theorem \ref{global} holds for sufficiently large viscosity $\mu$, initial data away from vacuum $n_{0}(x) > 0$, and small gradient initial data, suggests future investigation of understanding whether nontrivial strong solutions exist for arbitrary viscosity parameters $\mu$, or in the case of vacuum. The difficulty for considering initial data with vacuum is that the dissipative term $\displaystyle -\beta(\gamma \theta + 1)p \int_{0}^{t} \int_{\Omega} n^{\gamma \theta} n^{p}_{j}$ in \eqref{gradequality} vanishes when $n = 0$. In addition, to make this dissipative term have ``enough" dissipation to absorb the effects of the other nonlinear terms in the estimate, we require $\mu$ to be sufficiently large. It would therefore be interesting to consider arbitrary $\mu$, and also the singular limit as $\mu \to 0$. 

We have also shown a long-time convergence result for global strong solutions in Proposition \ref{longtime}. It would be interesting to see if this result extends more generally to weak solutions, which always exist globally for $L^{\infty}(\Omega)$ nonnegative initial data. In particular, the level set analysis in Section \ref{vacuumlongtime} relies primarily on the flow map $\Phi_{t}(x)$ for the vector field $\bd{u} = -\nabla W$. Even at the weak solution level, we have that $n \in L^{\infty}(0, T; L^{\infty}(\Omega))$, and hence, $W \in L^{\infty}(0, T; W^{2, p}(\Omega))$ for all $1 \le p < \infty$ by Calder\'{o}n-Zygmund estimates. Then, $\text{div}(\bd{u}) = -\Delta W = \mu^{-1}(an^{\gamma} - W)$ is in $L^{\infty}(0, T; L^{\infty}(\Omega))$. This is sufficient regularity to apply the transport theory of DiPerna and Lions \cite{DeLellis, DiPernaLions}, and hence obtain a \textit{generalized flow map} for the Sobolev vector field $\bd{u} := -\nabla W \in L^{\infty}(0, T; W^{1, p}(\Omega))$ for all $1 \le p < \infty$. While the low (non-Lipschitz) regularity of the vector field $\bd{u}$ and the generalized flow map may not be sufficient to directly obtain the explicit expansion/contraction of volume formulas in Proposition \ref{divuprop} for example, a similar argument of approximating/regularizing the equations \eqref{equations} and passing to a limit, might give insight into long-time behavior of global weak solutions, via a result analogous to Theorem \ref{longtime}.

\section*{Acknowledgements}

J. Kuan is supported by the National Science Foundation under the Mathematical Sciences Postdoctoral Research Fellowship (MSPRF) DMS-2303177.
K. Trivisa gratefully acknowledges the support by the National Science Foundation under the grants DMS-2231533 and DMS-2008568.

\section{Appendix: Analysis of an elliptic equation}\label{appendix}

In this appendix, we analyze the elliptic equation with Neumann boundary conditions:
\begin{equation}\label{appendixeqn}
-\mu \Delta W + W = F \text{ on } \Omega, \qquad \nabla W \cdot \bd{n}|_{\partial \Omega} = 0.
\end{equation}
The goal of the analysis will be to deduce properties of the Green's function (in particular, continuity and positivity properties) that will enable use to show that the vacuum vanishes in long time for global strong solutions to \eqref{equations}, in Lemma \ref{step3}. To do this, we first discuss the fundamental solution for the operator $(-\mu \Delta + I)$ on $\R^{d}$ for $d = 2, 3$, and then use these results on the fundamental solution to deduce important properties of the Green's function for \eqref{appendixeqn}, which we need for the proof of our long-time behavior result.

\subsection{The fundamental solution on $\R^{d}$} 

First, we find the fundamental solution for the operator $(-\mu \Delta + I)$ on $\R^{d}$ for $d = 2, 3$, which is a function $\Psi(x)$ such that 
\begin{equation*}
-\mu \Delta \Psi + \Psi = \delta_{0} \qquad \text{ on } \R^{d}
\end{equation*}
in the distributional sense. We can find such a solution by noting that $-\mu \Delta \Psi + \Psi = 0$ away from the origin. Writing the equation in polar coordinates for a radially symmetric function, $\Psi(x) = w(\mu^{-1/2} r)$, we obtain for $r \ne 0$:
\begin{equation*}
r^{2}w''(r) + (d - 1) r w'(r) - r^{2} w(r) = 0.
\end{equation*}

\medskip

\noindent \textbf{Fundamental solution for $d = 2$.} By the definition of Bessel functions, we obtain the fundamental solution in terms of a \textbf{modified Bessel function of the second kind} for a constant $c_{2}$:
\begin{equation*}
w(r) = c_{2} K_{0}(r), \qquad \Psi(x) = c_{2}K_{0}(\mu^{-1/2} |x|),
\end{equation*}
where $K_{0}$ is a modified Bessel function of the second kind, see Section 10.25 of \cite{DLMF}.

We can solve for $c_{2}$ by considering an arbitrary test function $\phi \in C_{c}^{\infty}(\R^{2})$, integrating by parts, and omitting a ball around the origin:
\begin{align*}
\phi(0) = -\mu \int_{\R^{2}} \Psi(x)& \Delta \phi + \int_{\R^{2}} \Psi(x) \phi \\
&= \lim_{\epsilon \to 0} \left(\int_{\R^{2} \setminus B_{\epsilon}(0)} (-\mu \Delta \Psi + \Psi) \phi + \int_{\partial B_{\epsilon}(0)} \Psi(x) \frac{\partial \phi}{\partial \nu} - \int_{\partial B_{\epsilon}(0)} \frac{\partial \Psi}{\partial \nu} \phi\right) \\
&= \lim_{\epsilon \to 0} \left(\int_{\partial B_{\epsilon}(0)} \Psi(x) \frac{\partial \phi}{\partial \nu} - \int_{\partial B_{\epsilon}(0)} \frac{\partial \Psi}{\partial \nu} \phi\right).
\end{align*}
By (10.29.3), (10.30.2), and (10.30.3) in \cite{DLMF}, $K_{0}(r) \sim -\log(r)$ and $K_{0}'(r) \sim -1/r$, as $r \to 0$. Therefore, since $\nabla \phi$ is globally bounded due to the compact support of $\phi$, we have
\begin{align}\label{smoothgreen1}
\left|\int_{\partial B_{\epsilon}(0)} \Psi(x)\frac{\partial \phi}{\partial \nu}\right| \le -c_{2} \Big(\log(\mu^{-1/2} \epsilon)\Big) \cdot (2\pi \epsilon) \max_{x \in \R^{2}} |\nabla \phi| \to 0, \qquad \text{ as } \epsilon \to 0,
\end{align}
\begin{equation*}
-\lim_{\epsilon \to 0} \int_{\partial B_{\epsilon}(0)} \frac{\partial \Psi}{\partial \nu} \phi = -c_{2} (\mu^{-1/2} \epsilon)^{-1} \cdot \phi(0) \cdot 2\pi \epsilon = -2\pi\mu^{1/2} c_{2}\phi(0).
\end{equation*}
So we conclude that
\begin{equation*}
\phi(0) = -\mu \int_{\R^{2}} \Psi(x) \Delta \phi + \int_{\R^{2}} \Psi(x) \phi = \lim_{\epsilon \to 0} \left(\int_{\partial B_{\epsilon}(0)} \Psi(x) \frac{\partial \phi}{\partial \nu} - \int_{\partial B_{\epsilon}(0)} \frac{\partial \Psi}{\partial \nu} \phi\right) = -2\pi\mu^{1/2} c_{2}\phi(0).
\end{equation*}
Therefore, $\displaystyle c_{2} = -\frac{1}{2\pi \mu^{1/2}}. $
\bigskip

\noindent \textbf{Fundamental solution for $d = 3$.} In the case of $d = 3$, we obtain by inspection $w(r) = e^{-r}/r$. Therefore, the fundamental solution for a dimensional constant $c_{3}$ is:
\begin{equation*}
\Psi(x) = c_{3} \mu^{1/2} \left(\frac{e^{-\mu^{-1/2}|x|}}{|x|}\right).
\end{equation*}

As in the case of $d = 2$, we calculate the constant $c_{3}$ by noting that
\begin{equation}\label{smoothgreen2}
\phi(0) = -\mu \int_{\R^{3}} \Psi(x) \Delta \phi + \int_{\R^{3}} \Psi(x) \phi = \lim_{\epsilon \to 0} \left(\int_{\partial B_{\epsilon}(0)} \Psi(x) \frac{\partial \phi}{\partial \nu} - \int_{\partial B_{\epsilon}(0)} \frac{\partial \Psi}{\partial \nu} \phi\right).
\end{equation}
We compute that
\begin{equation*}
\left|\int_{\partial B_{\epsilon}(0)} \Psi(x) \frac{\partial \phi}{\partial \nu}\right| \le c_{3}\mu^{1/2}\epsilon^{-1} e^{-\mu^{-1/2}\epsilon} 4\pi \epsilon^{2} \max_{x \in \R^{3}} |\nabla \phi(x)| \to 0, \qquad \text{ as } \epsilon \to 0,
\end{equation*}
\begin{equation}\label{normalizationK}
-\lim_{\epsilon \to 0} \int_{\partial B_{\epsilon}(0)} \frac{\partial \Psi}{\partial \nu} \phi = \lim_{\epsilon \to 0}  c_{3}\mu^{1/2} \left(\epsilon^{-2} e^{-\mu^{1/2}\epsilon} + \mu^{-1/2} \epsilon^{-1} e^{-\mu^{1/2}\epsilon}\right) 4\pi \epsilon^{2} \phi(0) = 4\pi \mu^{1/2} c_{3} \phi(0).
\end{equation}
Therefore, we conclude that $\displaystyle c_{3} = \frac{1}{4\pi \mu^{1/2}}$. 

\bigskip

\noindent \textbf{Conclusion.} Thus, we have that the fundamental solution is:
\begin{equation}\label{fundsoln}
\Psi(x) = \begin{cases}
    -c_{2}K_{0}(\mu^{-1/2}|x|), \qquad \text{ for } d = 2, \\
    \displaystyle c_{3} \left(\frac{e^{-\mu^{1/2}|x|}}{|x|}\right), \qquad \text{ for } d = 3,
\end{cases}
\end{equation}
where the constants $c_{2}$ and $c_{3}$ are given by $c_{2} = (2\pi)^{-1}\mu^{-1/2}, c_{3} = (4\pi)^{-1}\mu^{-1/2}$.

\subsection{Analysis of the Green's function}

We use the fundamental solution \eqref{fundsoln} to analyze the Green's function for the operator $-\mu \Delta + I$ posed on a bounded domain $\Omega \subset \R^{d}$ with smooth boundary. We recall that the Green's function $G(x, y): \Omega \times \Omega \to \R$ associated to the Neumann problem for the operator $-\mu \Delta + I$ on $\Omega$, is a function satisfying for each fixed but arbitrary $x \in \Omega$:
\begin{equation*}
-\mu \Delta_{x} G(x, y) + G(x, y) = \delta_{y}, \qquad \nabla_{x}G(x, y) \cdot \bd{n}|_{\partial \Omega} = 0.
\end{equation*}
Here, $\delta_{x}$ denotes the Dirac delta distribution centered at the point $x \in \Omega$. The Green's function is useful because it can be used to construct the (unique) solution to the inhomogeneous Neumann problem $-\mu \Delta W + W = f(x)$ and $\nabla W \cdot \bd{n}|_{\partial \Omega} = 0$ via the following integration:
\begin{equation*}
W(x) = \int_{\Omega} G(x, y) f(y) dy.
\end{equation*}

Given a general bounded domain $\Omega$ with smooth boundary, we can construct the Green's function $G(x, y)$ via the fundamental solution $K(x)$ computed in \eqref{fundsoln} as follows. Define a \textit{corrector function} $\varphi(x, y)$ as the solution to the Neumann problem:
\begin{equation}\label{corrector}
-\mu \Delta_{x} \varphi(x, y) + \varphi(x, y) = 0, \qquad \nabla_{x} \varphi(x, y) \cdot \bd{n}|_{\partial \Omega} = \nabla_{x} \Psi(\cdot - y) \cdot \bd{n} \Big\vert_{\partial \Omega}.
\end{equation}
Then, the Green's function can be expressed in terms of the fundamental solution as
\begin{equation*}
G(x, y) = \Psi(x) - \varphi(x, y).
\end{equation*}
We can use this characterization of the Green's function to establish some important properties of $G(x, y)$ that will be useful in our analysis of the tumor growth model under consideration. Related to this is the positivity of the Green's function, and to show this, we first show that the operator $(-\mu \Delta + 1)W$ in \eqref{elliptic} has a strict positivity property:

\begin{lemma}\label{ellipticmax}
    Suppose that $W \ge 0$ satisfies
    \begin{equation}\label{ellipticzero}
    -\mu\Delta W + W = 0
    \end{equation}
    on an open ball $B$, where $W \in C^{2}(B) \cap C(\overline{B})$. Then, either $W(x) > 0$ everywhere on $B$, or $W$ is identically zero on $B$.
\end{lemma}

\begin{proof}
By the strong maximum principle (cf. Theorem 4 on pg.~352 in \cite{Evans}), if $W(\bd{x}_0) = 0$ at an interior point of $B$ and is not identically zero on $B$, then $W$ attains a nonpositive minimum on the interior of $B$ and hence $W$ must be constant. However, this is a contradiction, since any constant solution to the equation \eqref{ellipticzero} must be identically zero.
\end{proof}

Now, we have the components needed to establish these properties of the Green's function.

\begin{proposition}\label{greenproperties}
The Green's function $G(x, y)$ for $-\mu \Delta + I$ has the following properties:
\begin{itemize}
    \item \textbf{Symmetry.} $G(x, y) = G(y, x)$ for $(x, y) \in (\Omega \times \Omega) - D$, where the diagonal $D := \{(x, y) \in \Omega \times \Omega : x = y\}$.
    \item \textbf{Continuity off of the diagonal.} $G(x, y)$ is continuous as a function of $(x, y) \in (\Omega \times \Omega) \setminus D$. Furthermore, $\lim_{(x, y) \to (z, z)} G(x, y) = \infty$ for any point $z \in \Omega$, where the limit is taken along points in $\Omega \times \Omega$.
    \item \textbf{Positivity of the Green's function.} $G(x, y) > 0$ for all $x, y \in \Omega$.
\end{itemize}
\end{proposition}

\begin{proof}
\noindent \textbf{Proof of symmetry.} Consider arbitrary points $y, z \in \Omega$ such that $y \ne z$ and let $B(x, r)$ denote the open ball of radius $r$ centered at the point $x$. Consider $r > 0$ sufficiently small so that $B(x, r) \cap B(z, r) = \varnothing$ and apply Green's identity to the set $$U_{y, z} := \Omega \setminus (B(y, r) \cup B(z, r))$$ to obtain:
\begin{footnotesize}
\begin{align}\label{greensymcalc}
0 &= \int_{U_{y, z}} \Big(-\mu\Delta_{x}G(x, y) + G(x, y)\Big)G(z, y) dx \nonumber \\
&= \int_{U_{y, z}} \Big(-\mu \Delta_{x} G(x, z) + G(x, z)\Big) G(x, y) ds + \mu \int_{\partial B(y, r) \cup \partial B(z, r)} \Big(G(x, z) \nabla_{x} G(x, y) - G(x, y) \nabla_{x} G(x, z)\Big) \cdot \bd{n} dS(x) \nonumber \\
&= \mu \int_{\partial B(y, r) \cup \partial B(z, r)} \Big(G(x, z) \nabla_{y}G(x, y) - G(x, y) \nabla_{y}G(x, z)\Big) \cdot \bd{n} dS(x),
\end{align}
\end{footnotesize}
where we used the definition of the Green's function to conclude that $-\mu \Delta_{x} G(x, y) + G(x, y) = 0$, along with integration by parts and the zero Neumann boundary condition on the Green's function. Here, we emphasize that $\bd{n}$ denotes the outer unit normal to the boundary of the balls $B(y, r)$ and $B(z, r)$. Because $G(x, y)$ for fixed $y$ is smooth away from $x = y$, we can conclude that $\displaystyle \lim_{r \to 0} \int_{\partial B(z, r)} G(z, y) \nabla_{y} G(x, y) \cdot \bd{n} = 0$ and similarly, $\displaystyle \lim_{r \to 0} \int_{\partial B(y, r)} G(x, y) \nabla_{y} G(x, z) \cdot \bd{n} = 0$, as in \eqref{smoothgreen1} and \eqref{smoothgreen2}. The result will be established if we show that
\begin{equation*}
\lim_{r \to 0} \int_{\partial B(y, r)} G(x, z) \nabla_{x} G(x, y) = -G(y, z), \qquad \lim_{r \to 0} \int_{\partial B(z, r)} G(x, y) \nabla_{x} G(x, z) = -G(z, y),
\end{equation*}
because then taking the limit as $r \to 0$ in \eqref{greensymcalc} gives the result that $G(y, z) = G(z, y)$.

So it suffices to show that
\begin{multline}\label{symreduction}
\lim_{r \to 0} \int_{\partial B(y, r)} G(x, z) \nabla_{x}G(x, y) \\
= \lim_{r \to 0} \int_{\partial B(y, r)} G(x, z) \nabla_{x}\Psi(x - y) + \lim_{r \to 0} \int_{\partial B(y, r)} G(x, z) \nabla_{x}\varphi(x, y) = -G(y, z),
\end{multline}
since an analogous calculation works for the same integral with $y$ and $z$ interchanged. By \eqref{corrector}, $\varphi$ is the solution to an elliptic equation with constant coefficients and smooth Neumann data on a smooth domain, so $\varphi(x, y)$ for fixed $y$ is a smooth function of $x$ on $\Omega$. Thus, the second term on the right-hand side of \eqref{symreduction} is zero. By the normalization of the fundamental solution $\Psi(x)$ in \eqref{normalizationK}, the first term on the right-hand side of \eqref{symreduction} is $-G(y, z)$, which concludes the proof of \eqref{symreduction} and hence, of the symmetry property.

\medskip

\noindent \textbf{Proof of off-diagonal continuity.} The goal is to show that $G(x, y) := \Psi(x - y) + \varphi(x, y)$ is continuous as a function of $(x, y) \in (\Omega \times \Omega) \setminus D$. To see this, one can see from the explicit form of the fundamental solution in \eqref{fundsoln} that $\Psi(x - y)$ is continuous on $(\Omega \times \Omega) \setminus D$. So it suffices to show that $\varphi(x, y)$ is continuous on $\Omega \times \Omega$ (and in fact, this is true, including the diagonal). Since $\lim_{(x, y) \to (z, z)} \Psi(x - y) = \infty$, this would also prove the assertion about the limit towards the diagonal being positive infinity since $\varphi(x, y)$ would be a bounded in value along on the diagonal, and continuous on all of $\Omega \times \Omega$ (including the diagonal).

So we show that $\varphi(x, y)$ is continuous on $\Omega \times \Omega$. Fix $x_0 \in \Omega$ and $\epsilon > 0$, and note that by smoothness properties of the fundamental solution $\Psi(x)$ in \eqref{fundsoln} (away from the singularity at the origin), there exists some $\delta$ sufficiently small with $B(x_0, \delta) \subset \Omega$ such that for all $x \in B(x_0, \delta)$:
\begin{equation*}
\|\nabla_{x} \Psi(x_0 - y) \cdot \bd{n} - \nabla_x \Psi(x - y) \cdot \bd{n}\|_{C^{1, \alpha}(\partial \Omega)} \le \epsilon, \qquad \text{ for } 0 < \alpha < 1.
\end{equation*}
Then, by Schauder estimates (Theorem 6.30 in \cite{Gilbarg}), we can show using the linearity of the equation \eqref{corrector} and the bound on the closeness of the Neumann boundary data, that for all $x \in B(x_0, \delta)$:
\begin{equation*}
\|\varphi(x_0, \cdot) - \varphi(x, \cdot)\|_{C(\overline{\Omega})} \le C\epsilon,
\end{equation*}
which shows that $\varphi(x, y)$ is continuous on $\Omega \times \Omega$ and hence completes the proof of the claim.

\medskip

\noindent \textbf{Proof of positivity.} Let $S := \{(x, y) \in \Omega \times \Omega : G(x, y) > 0\}$. $S$ is nonempty since we proved that $\lim_{(x, y) \to (z, z)} G(x, y) = \infty$ in the previous property. This statement about the limit towards the diagonal being infinity (along with the continuity of $\varphi(x, y)$ on all of $\Omega \times \Omega$), along with the fact that $G(x, y)$ is continuous off the diagonal show that $S$ is an open set. Then, Proposition \ref{ellipticmax} also shows that $S$ is a closed set, so since $\Omega \times \Omega$ is connected, we conclude that $S = \Omega \times \Omega$, which completes the proof. 

\end{proof}

\printbibliography
\end{document}